\documentclass[11pt]{article}

\RequirePackage[l2tabu, orthodox]{nag}
\usepackage{color} 
\usepackage[usenames,dvipsnames]{xcolor}
\usepackage{amsrefs,amsthm,amssymb,amsmath}
\usepackage{graphicx}
\usepackage{pifont}
\renewcommand{\checkmark}{\ding{52}}
\usepackage{enumerate}
\numberwithin{equation}{section}

\def\comment#1{}
\newcommand\numberthis{\addtocounter{equation}{1}\tag{\theequation}}
\newcommand{\beqn}[2]{\begin{equation}\label{eq:#1}#2\end{equation}}

\newcommand{\R}{\mathbf{R}}
\newcommand{\N}{\mathbf{N}}

\renewcommand{\P}{\mathbf{P}_p}
\newcommand{\mN}{\mathcal{N}}
\newcommand{\ba}{A}

\newcommand{\at}{\mathbf{x}}
\newcommand{\rt}{\bullet}
\newcommand{\bm}{\bar{\mu}}
\newcommand{\tmu}{\widetilde{\mu}}

\newcommand{\wbm}{\widetilde{\mu}}
\newcommand{\Fbm}{F_{\bm}}
\newcommand{\Fbmd}{F_{\bm,\alpha,\delta}}

\newcommand{\I}{\mathbf{1}}
\newcommand{\E}{\mathbf{E}}

\newcommand{\In}{I}
\newcommand{\Out}{O}
\newcommand{\G}{\Gamma}

\newcommand{\T}{\mathcal{T}}

\newcommand{\mF}{\mathcal{F}}
\newcommand{\mcM}{\mathcal{M}}
\newcommand{\mC}{\mathcal{C}}

\newcommand{\tX}{\widetilde{X}}
\newcommand{\wX}{\widetilde{X}}  
\newcommand{\tnu}{\widetilde{\nu}}  
\newcommand{\cP}{\mathcal{P}}

\newcommand{\tR}{\widetilde{R}}

\newcommand{\lcr}{\lambda_{cr}}
\newcommand{\lo}{\preccurlyeq}
\newcommand{\go}{\succcurlyeq}
\newcommand{\tG}{\widetilde{G}}

\newcommand{\io}{I\to O}

\newcommand{\eps}{\varepsilon}
\newcommand{\law}{\mathop{\mathrm{law}}} 
\newcommand{\kal}{\kappa_{\alpha}}
\newcommand{\kone}{\kappa_{(1-\alpha)}}
\newcommand{\Pc}{\Phi_{\lcr}}

\newcommand{\ddelta}{\mathrm{Dirac}}

\newcommand{\bDe}{\overline{\Delta}}
\newcommand{\hf}{\frac{1}{2}}

\newcommand{\MG}{\mathbf{M}} 

\newcommand{\bF}{\mathbf{F}}
\newcommand{\mX}{\mathcal{X}}
\newcommand{\mXi}{\mathbf{\Upsilon}}
\newcommand{\mR}{\mathcal{R}}
\newcommand{\mPhi}{\mathbf{\Phi}}
\newcommand{\mm}{\mathbf{m}}
\newcommand{\mbm}{\overline{\mathbf{m}}}
\newcommand{\wmu}{\widetilde{\mu}}
\newcommand{\diam}{\mathop{\mathrm{diam}}}
\newcommand{\BRW}{\mathrm{BRW}}

\newcommand{\vs}{B}

\newtheorem{thm}{Theorem}
\newtheorem*{thm*}{Theorem}
\newtheorem{lem}{Lemma}
\newtheorem{prop}{Proposition}

\newtheorem*{qn*}{Question}
\theoremstyle{definition}
\newtheorem{dfn}{Definition}
\theoremstyle{remark}
\newtheorem{rem}{Remark}
\newtheorem{ex}{Example}

\usepackage[margin=2.5cm]{geometry}

\usepackage{indentfirst}
\usepackage{microtype}

\usepackage[colorlinks=true,linkcolor=blue,citecolor=magenta]{hyperref}

\title{
Stationary random metrics on hierarchical graphs via $(\min,+)$-type recursive distributional equations}
\date{\today}
\author{Mikhail Khristoforov
\footnote{
Chebyshev Laboratory, St.~Petersburg University,
14th Line 29B, Vasilyevsky Island, St.~Petersburg 199178, Russia. E-mail: micvog@mail.ru}
\footnote{
	Universit\'e de Gen\`eve, Section de Math\'ematiques,
	2-4 rue du Li\`evre, CH-1227 Gen\`eve-Acacias, Suisse.} \and
 Victor Kleptsyn
\footnote{
	Institut de Recherches Math\'ematiques de Rennes, UMR 6625 CNRS,
	Campus Beaulieu, 35042 Rennes, France. E-mail: victor.kleptsyn@univ-rennes1.fr} \and
 Michele Triestino
\footnote{
  Departamento de Matem\'atica PUC-Rio, Rua Marqu\^es de S\~ao Vicente, 225, G\'avea, Rio de Janeiro 
  CEP 22451-900, Brasil. E-mail: mtriestino@mat.puc-rio.br}
}

\begin{document}

\maketitle

\begin{abstract}
This paper is inspired by the problem of understanding in a mathematical sense the Liouville quantum 
gravity on surfaces. Here we show how to define a stationary random metric on self-similar 
spaces which are the limit of nice finite graphs: these are the so-called hierarchical graphs. They possess a 
well-defined level structure and any level is built using a simple recursion. Stopping the construction at any 
finite level, we have a discrete random metric space when we set the edges to have random length 
(using a multiplicative cascade with fixed law~$m$). 

We introduce a tool, the cut-off process, by means of which one finds that renormalizing the sequence of metrics by an 
exponential factor, they converge in law to a non-trivial metric on the limit space. Such 
limit law is stationary, in the sense that glueing together a certain number of copies of the random limit 
space, according to the combinatorics of the brick graph, the obtained random metric has the same 
law 
when rescaled by a random factor of law~$m$. In other words, the stationary random metric is the solution of a 
distributional equation.
When the measure~$m$ has continuous positive density on~$\R_+$, the stationary 
law is unique up to rescaling and any other distribution tends to a rescaled stationary law under the 
iterations of the hierarchical transformation. We also investigate topological and geometric properties of the
random space when~$m$ is $\log$-normal, detecting a phase transition influenced by the branching random walk
associated to the multiplicative cascade.
\end{abstract}

\newpage

\tableofcontents

\section{Presentation}
\subsection{Introduction}
The main motivation for this work is to make a little step towards the mathematical 
understanding of Liouville quantum gravity: the problem of giving a meaning
to the ``metric tensor'' defined on a surface $\Omega$ by the exponential of the 
Gaussian Free Field (GFF for short), has been drawing the attention of many mathematicians 
up to very recent and important works (we shall give a concise review in \S\ref{ssc:lft}). 

For planar domains 
$\Omega$ carrying a particular geometrical structure, it is possible to (attempt to) define a 
similar ``metric tensor'' using \emph{multiplicative cascades}. To give a concrete
picture, let~$\Omega$ be the unit square in the plane, then using  dyadic coordinates, 
it is possible to identify it with the ends of a rooted quaternary tree
(with the little care that this identification is clearly not one-to-one). 
If we assign positive random weights to the edges of this tree 
(with the random factors that are i.i.d.), we can define a \emph{formal} weight for almost 
any end, namely the infinite product of the factors that we read along the path 
connecting the root to the end.
The purpose is to understand whether it is possible to define a measurable 
(pseudo-)metric, whose metric tensor at a point $z$ is the corresponding infinite product, properly renormalized. 

In a different formulation, we want to understand the limit of the sequence of piecewise-flat 
Riemannian metrics, obtained by considering the factors on the tree up to a certain depth. 
There is no evident reason for this limit to be defined and non-degenerate. Though, notice that
the resulting limit random metric $d$, if it exists, must be \emph{stationary} (or stochastically self-similar): 
considering four different independent samplings of the random metric space $(Q_i,d_i)$ and a new independent
random factor $\xi$, we can construct a new random metric space~$(\tilde Q,\tilde d)$ which must have
the same law. More precisely, we glue the four squares~$(Q_i,\xi\cdot d_i)$ together along their 
sides and obtain a new square $\tilde Q$, defining the pseudo-metric $\tilde d$ on it by looking at the
shortest path between points.

The task of defining of a random \emph{measure}, associated to a multiplicative cascade, 
is much easier due to the additivity and martingale type arguments, and 
has been successfully studied since the pioneering work by Kahane and Peyri\`ere \cite{KP,K}. 
In the 1D setting, where the notion of metric and measure almost coincide, the problem of defining a random metric 
turns out to be much simpler; the random geometry of the limit metric has been studied by Benjamini and Schramm 
in~\cite{BS}. Pursuing these ideas, the different approaches by Duplantier and Sheffield \cite{DS} and Garban, 
Rhodes and Vargas \cite{LBM} are successful attempts to try to retrieve information about the metric from the 
random measure. Finally, Durrett and Liggett in~\cite{DL} were the pioneers in applying the fixed point technique 
to the problem on the interval, and this approach is quite close to the one we will be using in the present work.

One of the main results of this article is the construction of a stationary random metric on fractal objects that are
limits of finite graphs, that is \emph{hierarchical graphs}, whose nice self-similar structure allows to
define multiplicative cascades.
This is a challenging problem: contrary to an
interval, the number of possible geodesic paths joining any two points is infinite (as in the ``complete'' 2D problem). 
Though, these objects still save some peculiarity of the one-dimensional world, and thus are easier to be treated. 

\begin{figure}[ht]
\[\includegraphics[width=.4\textwidth]{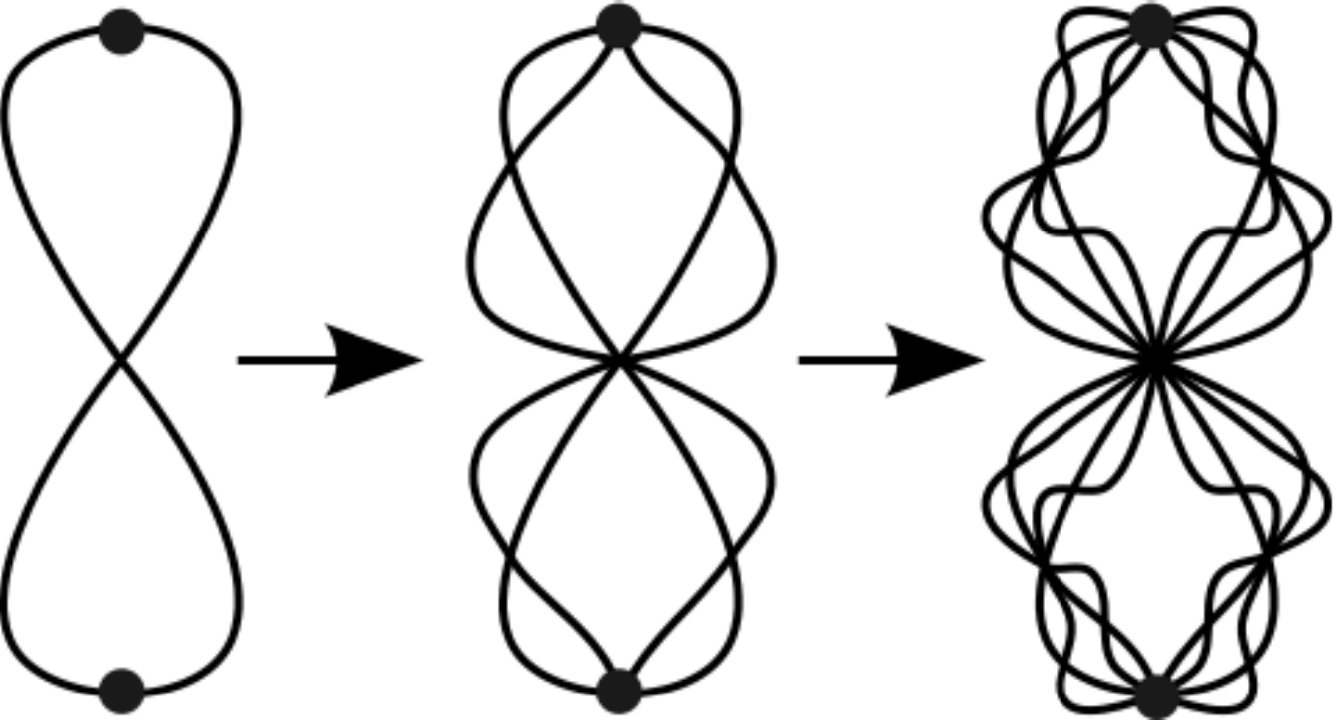}\]
\caption{Hierarchical figure eight construction procedures.} \label{f:eight}
\end{figure}

Hierarchical graphs and models have been widely
studied in physics and engineering for many years, often as successful toy models, and we are highly indebted to
Benjamini who focused our attention on such examples. To make the problem more
definite, he asked whether it is possible to solve it for a not too complicated object: start
with the \emph{figure eight-graph} $\G$ and build the associated hierarchical graph.
This means the following: we mark the top and bottom vertices $I$ (\emph{in}) and $O$ (\emph{out}) of the figure 
eight-graph and recursively build the sequence of combinatorial graphs $(\G_n,I,O)$, replacing every edge in 
$\G_{n-1}$ by a copy of $\G$. The \emph{hierarchical figure eight-graph} is the limit combinatorial 
object~$\G_{\infty}$ (see Figure~\ref{f:eight}). We will recall the precise definition in 
Section~\ref{s:hierarchical}, where we discuss hierarchical graphs in broader generality.

\subsection{Random metrics defined by Mandelbrot multiplicative cascades}\label{ss:definition}

Our model for studying random metrics on the hierarchical figure eight-graph reflects its 
self-similar structure, like in the example of the square. Assume that we are given a probability
distribution $m$ on the positive real numbers $\R_+=(0,+\infty)$. 
We start with the graph $\Gamma_0=(V_0,E_0)$ which is an interval of length~$1$, and
make the process of construction step by step. At each step $n\ge 1$, for each edge $I\in E_{n-1}$ of length $|I|$ 
in the graph $\Gamma_{n-1}=(V_{n-1},E_{n-1})$, we take a random variable~$\xi_I$, distributed with respect to~$m$;
these variables are mutually independent, for all intervals and for all steps. We then replace
$I$ by four edges of length $|I|\cdot \xi_I$ each, arranged in eight-shaped figure between
the original endpoints of~$I$. There is a natural metric $d_n$ on the set of vertices $V_n$, defined by taking 
the length of the shortest path between any two points.

It is natural to expect (and try to prove) that the
appropriately normalized metrics on these graphs converge to a random limit metric on
the limit object. In fact the normalizing constant behaves as $\lambda^{-n}$ for some $\lambda\ge 0$; 
the following intuitive argument for the case when the law $m$ has finite first moment has been shown to us by 
Nicolas Curien. Let $d_n(\In,\Out)$ be the random variable denoting the random $d_n$-distance between the vertices 
$\In$ and $\Out$ in the graph $\Gamma_n$ equipped with the random metric. It is easy then to remark that the 
sequence~$\E[d_n(\In,\Out)]$ is sub-multiplicative: for any positive integers $i$ and $j$
\[\E[d_{i+j}(\In,\Out)]\le \E[d_{i}(\In,\Out)]\,\E[d_{j}(\In,\Out)]\, .\]
This implies that the limit $\lambda=\lim_{n\rightarrow \infty} \E[d_{n}(\In,\Out)]^{-1/n}$ exists.
The point is to show that $\lambda$ is finite and no correction term appears for the exponential 
growth of the normalizing constant.

\medskip

As in the example of the square, we remark that this multiplicative cascade procedure can be reversed,
at least to study the law of the limit random metric.
Namely, a level~$n$ figure is glued out of four independent samples of level~$n-1$ figures,
that replace the intervals in the eight shape, and the obtained new metric is multiplied by a
random constant~$\xi$ (or by~$\lambda\xi$ if we want to include the normalizing constant). Passing to the limit 
(if it exists) we observe that the limit random metric must be stationary and the existence of stationary 
random metric can be translated into a \emph{fixed point problem}. To do so, we need some preliminary notations 
and definitions.

\subsection{The renormalization operator}

 Denote by $\Gamma_{\infty}$ the Gromov--Hausdorff limit of the 
graphs $\Gamma_n$, equipped with the ``Euclidean'' metrics (that is, each edge has length $1/2^n$). 
There is a natural inclusion $V_n\hookrightarrow V_{n+1}$ for the sets of vertices, allowing to define 
$V_{\infty}=\bigcup_n V_n$, that is then identifiable with a subset of the space~$\Gamma_{\infty}$ 
(borrowing the concept from the interval, one can consider~$V_{\infty}$ as the set of ``dyadic rational'' 
points of~$\Gamma_{\infty}$).

Consider the set $\MG$ of complete metric spaces $(\mX,d)$ that contain $V_{\infty}$ as a dense subset.
Given~\mbox{$\lambda >0$}, we can define a \emph{glueing map}
\[\mR_\lambda  : \MG^4 \times \R_+ \to \MG\]
which takes four metric spaces $(\mX_1,d_1),\ldots,(\mX_4,d_4)$, a positive factor $\xi$ and gives a 
new metric space~$(\mX,d)=\mR_\lambda (\mX_1,\ldots,\mX_4;\xi )$ defined as follows: 
\begin{itemize}
\item the space $\mX$ is obtained topologically by glueing the four spaces $\mX_i$ in a figure-eight shape;
\item the metric $d$ is the metric obtained by the glueing of the piecewise-defined metric $\{ d_i\}$ on 
$\mX$ and rescaling by multiplication by~$\lambda \xi$. 
\end{itemize}
Here by the \emph{glued metric}, we mean that we define the distance between any two given points as the
length of the shortest ``discrete path'' connecting them, see e.g.~\cite[\S{}3.1]{BBI}. There is actually 
no need of supposing that the spaces $\mX_i$ are path-connected: since we glue the four spaces at single points 
(namely at their~$\In$ and~$\Out$ vertices), the shortest distance is well-defined.
 
Remark that the glueing map commutes with scalar rescaling of distances: let $\xi\in\R_+$ and 
$\{(\mX_i,d_i)\}_{i=1}^4$ be four metric spaces, then for any constant $c>0$ the glued space
\[(\mX,c\cdot d)=\mR_\lambda\left ((\mX_1,c\cdot d_1),\ldots,(\mX_4,c\cdot d_4);\xi\right ),\]
carries the distance of
$(\mX,d)=\mR_\lambda\left ((\mX_1,d_1),\ldots,(\mX_4,d_4);\xi\right )$ rescaled by the factor~$c$. 
Hence, it is worth introducing the \emph{rescaling map} 
\[\mathbf{r}_c:\MG\to\MG\]
which takes the space $(\mX,d)$ to $(\mX,c\cdot d)$.

We will often use the notion of \emph{push-forward} of a measure. For a measure $\mu$ on some 
(Borel) space $X$ and a Borel map $T:X\to Y$, the push-forward $T_*\mu$ is a measure on $Y$, that 
is formally defined by the relation
\[
(T_*\mu)(A)=\mu(T^{-1}(A))
\]
for any Borel set $A\subset Y$. This is a way of saying that we are taking a $\mu$-distributed mass on $X$ 
and then transporting it via $T$ in order to obtain a measure on $Y$.

\smallskip

Given a probability measure $m$ on $\R_+$, the map $\mR_\lambda$  defines a transformation 
$\mPhi_\lambda$ as a push-forward on the space of Radon probability measures on $\MG$:
\[\mPhi_\lambda [\mm ] := (\mR_\lambda )_*  \bigl( \mm^{ \otimes 4} \otimes  m\bigr).\]
This is indeed the formal way of defining ``the law of the new glued metric'' for a given law $\mm$, 
four ``old'' ones and a given rescaling factor $\lambda$.

Then the stationarity condition for the law of the limit random metric (that is a measure on the space $\MG$) 
reads as the condition that this measure is a fixed point of this operator:
\begin{equation}\label{eq:Pl-stat}
\mPhi_\lambda [\mm ]=\mm.
\end{equation}
Further on, we will say that a measure $\mm$ satisfying~\eqref{eq:Pl-stat} is \emph{$\mPhi_{\lambda}$-stationary 
random metric} (even though, technically speaking, $\mm$ is a probability measure on the space of metrics).

A naive example is the Euclidean distance on $\Gamma_\infty$ (or, to be more precise, the Dirac measure concentrated 
at this point), that is a stationary ``random'' metric when $m$ is the Dirac mass at~$1/2$. 

Finally, note that the rescaling maps $\mathbf{r}_c$ define \emph{rescaling operators} $\mXi_c$ on the space of 
Radon probability measures on $\MG$ by 
\[\mXi_c[\mm]:=(\mathbf{r}_c)_*\mm\]
(simply rescaling the random metric by~$c$).
It is easy to remark that for any $c>0$, if $\mm$ is a $\mPhi_\lambda$-stationary random metric, then so is its 
$c$-rescaled image $\mXi_c[\mm]$, since the glueing $\mPhi_\lambda$ and the $c$-rescaling $\mXi_c$ commute.

\subsection{The main result}

Our principal result claims the existence of non-trivial stationary random metrics. 

\begin{thm}\label{p:metric}
For any non-atomic, fully supported probability measure $m$ on $\R_+$ there exists a
normalizing constant $\lcr>0$ for which there is a non-atomic $\mPhi_{\lcr}$-stationary random metric $\mbm$.
\end{thm} 

We can go further when the probability distribution $m$ is absolutely continuous
with respect to the Lebesgue measure on $\R_+$. In this case, under some additional assumptions on 
the density 
(we write $m(dx)=\rho(\log x) \frac{dx}{x}$, so that the function $\rho$ describes the density in the 
logarithmic coordinates), we show that any two 
stationary random metrics~$\mbm$,~$\mbm'$ are essentially the
same and in fact, any starting distribution converges to a stationary random metric under the
iterations of $\mPhi_{\lcr}$:
\begin{thm}\label{t:uniqueness}
Let $m(dx)=\rho(\log x) \frac{dx}{x}$ be an absolutely continuous probability measure on $\R_+$, where the 
function $\rho$ is strictly positive and continuous on $\R$ and tends to zero  as $x$ tends to $\pm\infty$.

Let $\mbm$ be a $\mPhi_{\lcr}$-stationary random metric. Then, for any probability measure $\mm$ on $\MG$ there 
exists a constant $c>0$ such that the iterations of this measure weakly-$*$ converge to the $c$-rescaled
stationary random metric $\mXi_c[\mbm]$, where we equip $\MG$ with the topology of pointwise convergence 
on $V_{\infty}\times V_{\infty}$.

In particular the $\mPhi_{\lcr}$-stationary random metric is unique up to a rescaling.
\end{thm}

\begin{rem}
Note that the convergence in Theorem~\ref{t:uniqueness} is weaker than the one we would dispose if we were 
using the Gromov--Hausdorff distance. Though, under some additional assumptions one can also claim this 
stronger version of convergence. This happens in the subcritical case (see~\S{}\ref{s:geometry}) for 
$\log$-normal law~$m$ with the initial law $\mm$, supported on a bounded set of metrics.
\end{rem}

\subsection{Geometric properties of the stationary random metric space}\label{s:geometry}

The next important step is to study the properties of the stationary random metric~$\mbm$: for instance, 
is it concentrated on the spaces homeomorphic to~$\Gamma_{\infty}$? Our next result answers to this 
question in a particular case, when the measure $m$ is $\log$-normal 
($m=\exp_*\mN(0,\sigma^2)$), or has a tail behaviour  of this type (both at $0$ and $\infty$, 
see Remark~\ref{r:linear}). To state it, for $m=\exp_*\mN(0,\sigma^2)$ we write 
$\gamma_{\BRW}=\gamma_{\BRW}(m):=\sqrt{2 \log 4} \cdot \sigma$. Then we have the following theorem, 
depicting \emph{two} different regimes:
\begin{thm}\label{t:metric}
Let $m=\exp_*\mN(0,\sigma^2)$ be a $\log$-normal distribution, and~$\lcr$ and
a stationary random metric $\mbm$ be given by Theorem~\ref{p:metric}. Then:
\begin{itemize}
\item If $\gamma_{\BRW}+\log \lcr<0$, then the stationary random metric $\mbm$ is supported on the set of spaces for 
which the inclusion map $V_{\infty}\to\mX$ extends continuously to a homeomorphism between 
$\Gamma_{\infty}$ and $\mX$. Moreover, $\mbm$-almost surely, the Hausdorff dimension of $\mX$ is at most
\[\frac{2\log 2}{|\gamma_{\BRW}+\log\lcr|}\,.\]
\item If $\gamma_{\BRW}+\log \lcr>0$, then the diameter of the space $\mX$ is $\mbm$-almost surely 
infinite.
\end{itemize}
\end{thm}

The reason behind this result is the following. Our random metric space $\mX$ is glued out of  four 
$\lcr \xi$-rescaled independent random spaces, each one glued out of four rescaled random spaces, and so on. 
This descent can be described alternatively in this formal way: we have an  infinite rooted quaternary tree 
$\T$ with root $\bullet$, and independent random variables $\{\xi_t\}_{t\in \T}$ are associated to its vertices, 
all distributed with respect to~$m$. In order to group $n$ glueing steps at once, we take $4^n$ independent 
spaces $\mX_t$, indexed by vertices $t$ of depth $n$ and following the same law $\mbm$. Then we rescale each 
space $\mX_t$ by the factor $\prod_{j=0}^{n-1} \lcr \xi_{t_j}$, where $t_0,\dots,t_{n-1},t$ is a path from the 
root vertex $t_0=\bullet$ to $t$, and glue the rescaled spaces together.

The behaviour of the rescaling factors influences the geometry of the resulting space, dictating in particular 
the two regimes in Theorem~\ref{t:metric}. Namely, the first case occurs if the maximum of the factors tends 
to $0$ exponentially as $n$ tends to infinity, and the second one if such maximum explodes.

Using terms which are more familiar to the probabilists, we note that as $n$ increases, the collection of 
logarithms of the factors behaves like a branching random walk ($\BRW$) with increments given by the 
law $\log_*m$, shifted by an additional linear drift with speed $\log \lcr$.

The drift of the maximum of a $\BRW$ (under some mild assumptions on the law of the increments) is itself 
described by a classical result in the theory of branching random walks: the Hammersley-Kingman-Biggins 
Theorem~\cite{hammersley,kingman,biggins} (see also~\cite{lig,zeitouni_BRW}).
Namely, let $M_n$ be the maximum of this $\BRW$ after $n$ steps, then there exists 
a constant $\gamma_{\BRW}=\gamma_{\BRW}(m)$ such that 
\begin{equation}\label{eq:BRWconstant}
\lim_{n\rightarrow\infty}\frac{M_n}{n}=\gamma_{\BRW} \quad \text{almost surely}.
\end{equation}
When the increments follow the normal law $\mN(0,\sigma^2)$, the speed of the drift of the maximum is equal 
to $\gamma_{\BRW}(m)=\sqrt{2\log 4}\cdot \sigma$. 

\begin{rem}\label{r:linear}
In fact, as the reader can check, the arguments of the proofs of Theorems~\ref{t:metric}, as well as of 
Theorem~\ref{t:geodesic} below, will only use the behaviour of the density of~$m$ at infinity (for instance, 
supposing the density asymptotically equivalent to a $\log$-normal one is enough). These theorems thus stay 
valid for other measures $m$ with log-normal tails, with $\gamma_{\BRW}$ the drift speed for the maximum of 
the $\BRW$ associated to~$\log_* m$.
\end{rem}

\begin{rem}
As a consequence of Proposition~\ref{prop:dirac-perturbed} below, the first possibility takes place for sufficiently
small values of $\sigma$.\end{rem}

\begin{rem}\label{r:critical}
The bound given on the Hausdorff dimension is far to be optimal (see the discussion at the end of \S\ref{ssc:lft}) 
but it is a glimpse of the  r\^ole of the parameter $\gamma_{\BRW}+\log\lcr$ in distinguishing the two scenarios,
which have to be considered as \emph{subcritical} and \emph{supercritical} respectively. What happens at 
\emph{criticality} (\emph{i.e.}~when $\gamma_{\BRW}+\log \lcr=0$) seems to be a more delicate problem: 
indeed, it is classical \cite{BRW} that the asymptotic behaviour of~$M_n$ shows an additional negative 
logarithmic term, with a factor depending on the critical value of~$\sigma$, and so on the corresponding~$\lcr$. 
This correction suggests that the $\In\Out$-distances at depth~$n$ should decrease, but it could be not fast enough 
to ensure finite diameter at criticality.
\end{rem}

The random geometry in the supercritical case of Theorem~\ref{t:metric} has to be better understood. 
We discuss our intuition on it in Section~\ref{s:metric}. We can however state one precise result
that underlines the difference between stationary random metrics on hierarchical graphs like the 
figure eight and simpler ones, like the interval.
Observe that in the latter case, in the supercritical regime (e.g.~$\sigma>\sqrt{2\log 2}$ 
if~$m=\exp_*\mN(0,\sigma^2)$) the random metric space is almost surely homeomorphic to a Cantor set, even 
though there is no canonical way of constructing this so-called ``atomic multiplicative chaos''
\cite{DL,BJRV}. Surprisingly enough, for the hierarchical graphs the behaviour of random metric in the 
supercritical regime is quite different from what is known in the interval case:

\begin{thm}\label{t:geodesic}
Let $m=\exp_*\mN(0,\sigma^2)$ be a $\log$-normal distribution. Take the value~$\lcr$ and 
a stationary random metric $\mbm$ given by Theorem~\ref{p:metric}. Then the  
stationary random metric $\mbm$ is supported on connected spaces.
\end{thm}

\section{Strategy of the proof}

\subsection{Reduction to a one-dimensional problem}\label{ss:reduction}

In order to prove Theorem~\ref{p:metric} we highly exploit the particular geometry of the hierarchical figure 
eight-graph: it allows to recover the law $\mbm$ of the stationary random metric from the only knowledge of 
the ``marginal'' law $\bm$ of the distance between the vertices $\In$ and $\Out$. Such marginal law should 
then satisfy an analogous stationarity relation, that turns out to be much simpler to analyse as it concerns 
now usual (and not metric-space-valued) random variables.

Indeed, if there is any stationary random metric $\mbm$, it is easy to see that the $\In\Out$-distance also 
ought to be stationary under the glueing operation. Given $\lambda>0$, the map $\mR_\lambda$ projects to the map
\[R_\lambda\,:\,\R_+^4\times \R_+\to\R_+\]
which assigns to any four positive numbers $X_1,\ldots,X_4$ (thought as the $\In\Out$-distances inside some four metric 
spaces $(\mX_1,d_1),\ldots,(\mX_4,d_4)$ respectively) and positive factor $\xi$ the quantity
\begin{equation}\label{eq:Phi}
R_{\lambda}(X_1,X_2,X_3,X_4;\xi)=\lambda\,\xi\cdot (\min(X_1,X_2)+\min(X_3,X_4)).
\end{equation}
The geometrical interpretation of this projected glueing map is straightforward: any path going from $\In$ 
to $\Out$ in the glued graph must pass through the middle point and so we select the fastest way to go from 
the vertex $\In$ to the middle point and add its length to the one of the fastest way from $\Out$ to the 
same middle point.

Also the rescaling maps $\mathbf{r}_c$ project to maps $r_c$ defined on the space of $\In\Out$-distances 
$\R_+$, acting as plain scalar multiplication: $r_c(X)=c\cdot X$. 

In the core part of this paper we show that such marginal law $\bm$ exists (and is non-trivial). Before 
going further, it is more convenient to state also this latter problem as a \emph{fixed point problem} 
for a ``marginal'' renormalization operator $\Phi_\lambda$.

Let $\cP$ be the space of Radon probability distributions
on the compactified half-line~$[0,+\infty]$ equipped with the weak-$*$ topology. The ``marginal'' 
renormalization operator is defined as the push-forward
\[\Phi_\lambda[\mu]:=(R_\lambda)_*\left (\mu^{\otimes 4}\otimes m\right ),\]
while the ``marginal'' rescaling operator is
\[\Upsilon_c[\mu]:=(r_c)_*\mu.\]
Obviously, for any $\lambda$, $c>0$, the operators $\Phi_\lambda$ and $\Upsilon_c$ commute.
Theorem~\ref{p:metric} will be easily deduced from the following simpler statement:

\begin{thm}\label{t:stat}
For any non-atomic, fully supported probability measure $m$ on $\R_+$ there exists a
normalizing constant $\lcr>0$ and a non-atomic probability measure $\bm$ on
$\R_+$ such that $\bm$ is a fixed point for the operator $\Pc$:
\begin{equation}
\label{eq:iter}
\Pc[\bm]=\bm.
\end{equation}
\end{thm}

The uniqueness result of Theorem~\ref{t:uniqueness} is also a direct consequence of the analogue result 
for ``marginal'' random $\In\Out$-distances:

\begin{thm}\label{t:conv}
Let $m(dx)=\rho(\log x) \frac{dx}{x}$ be an absolutely continuous probability measure on $\R_+$, where the 
function $\rho$ is strictly positive and continuous on $\R$ and tends to zero as $x$ tends to~$\pm\infty$.

Let $\bm$ be a  probability measure on $\R_+$ such that
$\bm=\Pc[\bm]$. Then, for any probability measure $\mu$ on $\R_+$ there exists a
constant $c>0$ such that the iterations of this measure converge to the $c$-rescaled
measure~$\Upsilon_c[\bm]$.

In particular the $\Pc$-stationary probability measure is unique up to a rescaling.
\end{thm}

\begin{rem}
Equation~\eqref{eq:iter} is an example of \emph{recursive distributional equation} (RDE) of
$(\min,+)$-type (here we refer to~\cite{RDE} for an introduction to RDEs). 
Formulated in other words, Theorem~\ref{t:stat} gives a (non-atomic) solution 
for this RDE.
\end{rem}

\subsection{The cut-off method}\label{s:cut-sketch}

The main idea in the proof of Theorem~\ref{t:stat} is to use the following cut-off process: assume that after the
replacement and the multiplication, the distance $\In\Out$ is shortcut by
an ``exterior'' path of length $\ba$, where $\ba\in \R_+$ is a fixed constant. In other
words, instead of the map $R_{\lambda}$ in~\eqref{eq:Phi} consider the
map $R_{\ba,\lambda}:=\min(R_\lambda,\ba)$, defining the corresponding operator
$\Phi_{\ba,\lambda}$. 

For sufficiently large $\lambda$, there should be a 
$\Phi_{\ba,\lambda}$-stationary measure. Indeed, for such $\lambda$, it is natural to expect that 
measures $\mu$, concentrated ``away'' from $A$ on the interval $(0,A)$, will -- in a sense -- drift to the right; 
meanwhile, their support stays uniformly bounded above by~$A$, so the iterations of a starting measure will 
not go to $+\infty$.

Moreover one can construct it starting with a Dirac measure
concentrated on the infinite distance: it is easy to see that the distribution functions of its
images form a pointwise monotonely increasing sequence. 
The only question for such modified procedure is whether this limit becomes a Dirac
measure concentrated on the zero distance, in other words, whether the metric collapses.

It turns out that there exists a critical value $\lcr$ such that for $\lambda>\lcr$ the process
$\Phi_{\ba,\lambda}$ admits a stationary measure $\bm_{\ba,\lambda}$, while for
$\lambda \le \lcr$ the whole metric space collapses into a point. A key remark here is that the
collapse happens even for $\lambda = \lcr$. Then, the stationary measure for the process~$\Pc$ can be 
found as a limit of $\bm_{\ba(\lambda),\lambda}$ as $\lambda$
approaches $\lcr$ from the right, where the length~$\ba(\lambda)$ tending to $\infty$ is
chosen by the normalizing restriction $\bm_{\ba(\lambda),\lambda}([0,1])=1/2$. In other
words, we are at the same time decreasing the scaling parameter $\lambda$ to the critical
value (that would lead to the collapse if~$\ba$ stayed fixed) and making the shortcut length
$\ba$ tend to infinity (that would explode the metric if $\lambda$ stayed fixed). In the limit
(that we force to be non-trivial by our normalizing condition), we find the desired stationary
measure, that does not ``feel'' the shortcuts any more.

\subsection{Related models}\label{ssc:related}

It is important to say that the cut-off method can 
be adapted to some other situations, concerning random metrics or different problems
that can be stated in terms of solutions of RDEs of $(\min,+)$-type.
A very simple generalization is the content of Section~\ref{s:hierarchical}: our results extend
to a huge class of hierarchical graphs. 

A more involved example, still finite-dimensional, is the one of the hierarchical model on the Sierpi\'nski
Gasket (and fractals with similar properties): the information that we are keeping on the metric in this 
case are three distances between the vertices of the triangle. However, working with the distributions on 
the non one-dimensional space of such distances becomes more elaborate (there are no more partition functions 
that sometimes simplify the coupling). 

Though, there is an important property of the figure eight-graph: it does not have neither straight 
$\In\Out$-edges (``shortcuts''), nor edges through which any $\In\Out$-path is obliged to path (``bridges''). 
In presence of such edges, the behaviour becomes slightly different (we only deal briefly with such graphs 
in Section~\ref{s:hierarchical}), though our technique (with some slight modifications) is still applicable. 
It turns out (see Section~\ref{s:sierpinski}) that the properties of the Sierpi\'nski Gasket are closer to 
those of the hierarchical graphs with bridge edges.

\section{Open questions and overview}

\subsection{Critical parameter value $\lcr$ as a function of the measure $m$}

Consider the critical parameter $\lcr$ as a function of the probability distribution $m$, verifying the 
assumptions of Theorem~\ref{p:metric}. Contrary to the case of a Mandelbrot multiplicative cascade (MMC) 
on the interval, for any other graph there seems to be no analytic expression neither for the stationary 
measure~$\mbm$, nor for the critical parameter $\lcr$ as a function of $m$ even for the case of $\log$-normal 
measures~$m$ (there is, however, one case when such an expression can be found: see Example~\ref{e:Ch}). 
Our result in this direction is the following proposition that is useful when considering random perturbations 
of the ``Euclidean'' metric (that satisfies the stationarity equation with $\lcr(\ddelta_1):=\frac{1}{2}$).

\begin{prop}\label{prop:dirac-perturbed}
The function $\lcr=\lcr(m)$, defined according to Theorem~\ref{p:metric}, satisfies 
\[\lim_{m\rightarrow \ddelta_1}\lcr(m)=\tfrac{1}{2}.\]
\end{prop}

We have made some numerical computations of the critical parameter. In Figure~\ref{f:exp}, 
the results of such simulations for two different hierarchical graphs, the interval and the figure eight ones, 
for the $\log$-normal measure $m$, are presented. Recall~\cite{DL,BJRV} that for the interval, 
$\log (2\lcr)=-\frac{\sigma^2}{2}$ if $\sigma<\sigma_{cr}=\sqrt{2\log 2}$, while for $\sigma>\sigma_{cr}$ 
the associated MMC provides an almost surely \emph{atomic} measure, with $\log (2\lcr)=\log 2-\sigma_{cr} \sigma$; 
the corresponding theoretical values are shown in Figure~\ref{f:exp} on the left by a (red) curve. Note that in the 
supercritical regime $\sigma>\sigma_{cr}=\sqrt{2\log 2}$, the numerical approximation becomes quite unstable; at 
the same time, the simulations for the eight-shaped graph is apparently quite stable even for much larger values 
of~$\sigma$: according to Theorem~\ref{t:metric}, the critical parameter~$\sigma_{8}$ corresponds to the intersection of the 
dotted curve with the blue line $y=-\sqrt{2\log 4}\cdot\sigma+\log 2$, which gives a numerical estimate for $\sigma_{8}$ close to $0.30$.
The stability in this latter case seems to come from the fact that the figure eight-graph has no \emph{pivotal} edge 
(cf.~\S\ref{ssc:related} and \S\ref{ssc:pivotal}): the tails of the distribution function have a very small 
influence on the shape of its $\mPhi_{\lcr}$-image. Indeed, the effect of a single small distance will most 
probably be negligible, as any $\In\Out$-path passes by at least two edges, as well as the effect of a single 
large distance (contrary to the interval case!) as any edge can be by-passed by following the other edges. 

Note also that~$\log \lcr(\sigma)$ for the figure eight-graph seems to have a linear asymptotics as~$\sigma$ 
goes to~$\infty$. Indeed, for large values of~$\sigma$, the distances become so much dispersed in the 
logarithmic scale, that one can approximate the sum of lengths of edges along the path by the maximum of 
these lengths; in other words, we can replace~\eqref{eq:Phi} by 
\[
\tR_{\lambda}(X_1,X_2,X_3,X_4;\xi)=\lambda\,\xi\cdot \max(\min(X_1,X_2),\min(X_3,X_4)).
\]
Taking the logarithm, we get
\[
\log \tR_{\lambda}(X_1,X_2,X_3,X_4;\xi)=\log \lambda + \log \xi + 
\max(\min(\log X_1,\log X_2),\min(\log X_3,\log X_4)).
\]
This (approximated) problem, viewed in the logarithmic scale, is completely linear in~$\sigma$. Namely, 
suppose that for a certain~$\sigma>0$, the probability measure~$\log_* \mu$ is stationary for this modified 
problem, with associated drift~$\log\lcr(\sigma)$. Then it is easy to see that for any other value~$\sigma'$, 
the rescaled measure~$\Upsilon_{\sigma'/\sigma}[\log_* \mu]$ 
will be stationary with drift $\log \lcr(\sigma')=\frac{\sigma'}{\sigma}\cdot \log \lcr(\sigma)$.

Surely, these arguments are non-rigorous, and to formally prove the asymptotic linear behaviour 
of~$\log \lcr(\sigma)$, one has to prove that such an approximation has an error that is indeed 
sublinear in $\sigma$ (as well as to establish an analogue of Theorem~\ref{t:stat} for the modified problem).

\begin{figure}[ht]
\[
\includegraphics[width=.47\textwidth]{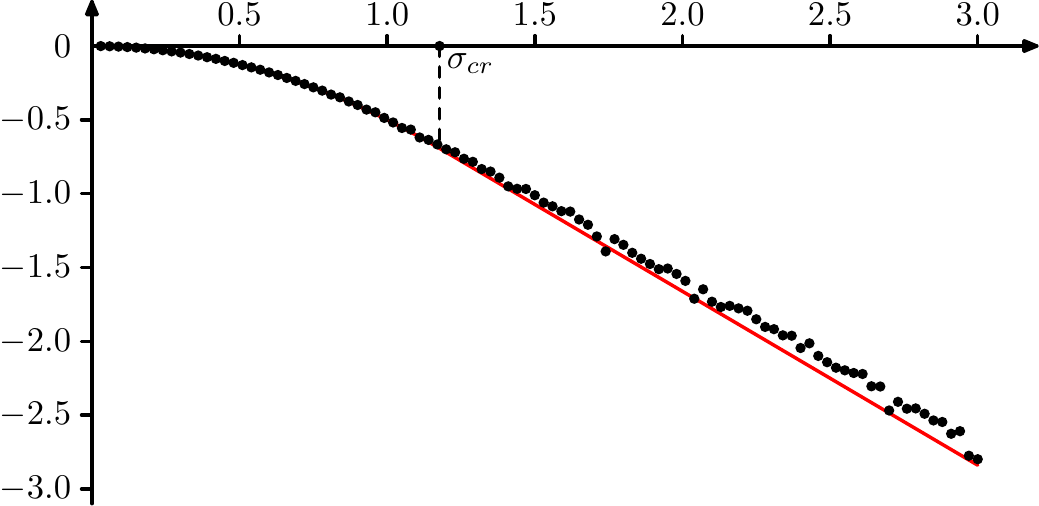} 
\qquad 
\includegraphics[width=.47\textwidth]{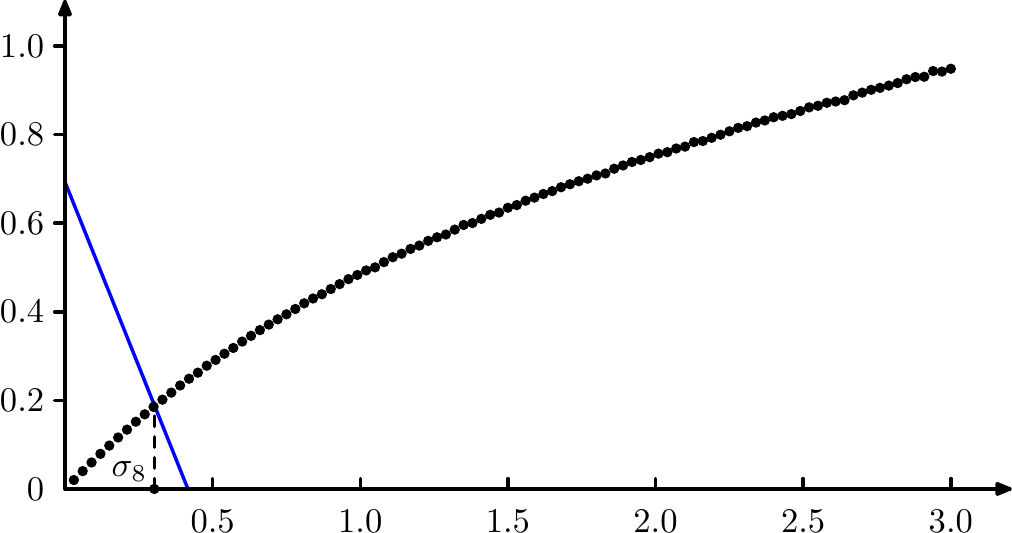} 
\]
\caption{\emph{Numerical modelling of the glueing processes for $m=\exp_*\mN(0,\sigma^2)$}. The random measure 
is modelled by a sample of size $N=50\,000$; the new sample is glued out of intervals chosen uniformly among 
the previous ones. The mean drift of the logarithm of the median of the sample after $k=50$ iterations is 
taken as an approximation for $\log(\lcr)$. Then the new sample is used as the starting one for the next slightly increased value of $\sigma$.
\emph{Left:} the numerical data for $\log(2\lcr)$ for the 
two-edge interval (black dots) and the theoretical (red) curve; notice the instability that starts in 
the region $\sigma>\sigma_{cr}$. \emph{Right:} the numerical data for $\log(2\lcr)$ for the figure eight-graph (black dots) and the (blue) line $y=-\sqrt{2\log 4}\cdot\sigma+\log 2$ that locates the critical parameter $\sigma_{8}\sim 0.30$; 
note the stability even for  large values of $\sigma$.}
\label{f:exp}
\end{figure}

\subsection{The inverse problem}

Instead of looking for a stationary law $\bm$ and a rescaling factor $\lcr$ for a given measure~$m$, one can 
reverse the problem: starting with the desired stationary distribution $\bm$, try to find the probability 
measure $m$ and the value $\lcr$ (or simply  the rescaled measure $m(\lcr^{-1}\cdot)$ that is the distribution 
of the random variable $\lcr \xi$, where $\xi$ has law $m$). In such a situation, we know the law of one of the 
two independent factors in~\eqref{eq:Phi}, the second one, as well as the law of the product (as it should 
coincide with~$\bm$). Thus, one can determine the law of $\lcr\xi$ by means of standard probabilistic tools. 
Namely, passing to the logarithmic coordinates transforms the product to the sum, and hence one can reconstruct 
the characteristic function of $\log(\lcr\xi)$ (which is the Mellin transform of $\lcr\xi$) as a quotient of 
the characteristic functions of $\log X$ and of $\log (\min(X_1,X_2)+\min(X_3,X_4))$, where $X$, $X_i$ have law 
$\bm$. We leave aside the evident eventual difficulties in the application of this procedure, such as treating 
the zeroes of the characteristic function or the fact that it is not  guaranteed at all that the quotient will 
be a characteristic function of a probability measure. 

The following example has been shown to us by Christophe Sabot:

\begin{ex}\label{e:Ch}
Take the measure~$\bm$ to be an exponential distribution, say for example of parameter~$\tfrac12$. 
Following the steps described above and using classical properties of Gamma distributions, one can 
find that it is stationary with respect to the random factor~$\lcr\xi/2$ that has the uniform law on~$[0,1]$.
\end{ex}

\subsection{Origins: a toy model for Liouville Field Theory}
\label{ssc:lft}
Mainly based on the work of Polyakov \cite{polyakov1,polyakov2}, the study of \emph{2D-quantum gravity} 
has been increasingly drawing  attention in physics and mathematics during the past three  decades.
For mathematicians, this often means studying random Riemannian surfaces: in fact, the intuition of 
Polyakov was to ``replace the old-fashioned (and extremely useful) sums over random paths'' with sums 
over random surfaces \cite{polyakov1}. Presented like this, there is much ambiguity on the significance 
of ``random''. Formally it becomes more definite when introducing the \emph{Liouville action} $S_L$ for 
a Riemannian surface $(\Omega,g_0)$ ($g_0$ is a fixed \emph{background metric}):
\[
S_L(h)=\frac{1}{4\pi}\int_\Omega\left ( g_0( \nabla_{g_0}h ,\nabla_{g_0}h )+Q\,R_{g_0}h
+4\pi\mu e^{\gamma h} \right )d\mathrm{vol}_{g_0},
\]
where $R_{g_0}$ is the curvature tensor, $\gamma\in (0,2]$, $Q=\frac{2}{\gamma}+\frac{\gamma}{2}$, and 
$\mu$ is the \emph{cosmological constant}. Like for path-integrals, the random surface $(\Omega,e^{\gamma h} g_0)$ 
will be chosen with a probability proportional to~$e^{-S_L(h)}dh$. When the cosmological constant $\mu$ is zero, 
we lose the interaction with matter (pure gravity situation). When $g_0$ is a flat metric, in pure gravity the 
action reduces to
\begin{equation}\label{eq:action-GFF}
S_L(h)=\frac12 \| h\|^2_{\nabla_{g_0}},
\end{equation}
where $\|h\|_{\nabla_{g_0}}=\left (\frac{1}{2\pi}\int_{\Omega}g_0( \nabla_{g_0} h,\nabla_{g_0}h)d\mathrm{vol}_{g_0}
\right )^{1/2}$ is the Dirichlet energy of $h\in H^1_0(\Omega)$ ($H^1$-functions orthogonal to constants). 
Mathematicians call the random field $h$, distributed according to this action, the \emph{Gaussian Free Field} 
(GFF for short), whereas it has different names in the physical literature. 

The link between 
2D-quantum gravity and GFF was made explicit in the independent works by David and Distler, Kawai 
\cite{david,distler-kawai} which extended to the usual conformal gauge the previous results of Knizhnik, 
Polyakov and Zamolodchikov (KPZ) \cite{KPZ} (dealing with the light-cone gauge, which means very roughly 
that they were considering a 2D world with many symmetries). For some very good introductions to 2D-Liouville 
quantum gravity we suggest the reading of \cite{ginsparg,nakayama,teschner,erbin}, although they are addressed 
to people with a background knowledge in quantum field theory and string theory.

Mathematically, there is a huge obstacle in considering random ``functions'' distributed according to the action 
\eqref{eq:action-GFF}, as the random field $h$ is almost surely only a distribution (in the sense of Schwartz). 
The GFF can be defined in a few different ways (for a good introduction, though not exhaustive, we recommend 
\cite{sheffield-GFF}). The one which will be relevant for our model is the following: let $\Delta$ be the 
Laplace operator on the (flat) Riemannian surface $(\Omega,g_0)$ and let $\left (\Phi_n\right )_{n\in \N}$ 
be an orthonormal basis of the Sobolev space $\left(H^1_0(\Omega),\langle\,\cdot\,,\,\cdot\rangle_{\nabla_{g_0}}\right)$,
 then the series
\[h=\sum_{n\in\N}a_{n}\Phi_n\]
defines the GFF when the coefficients $a_n$ are independent normally distributed random variables.

When $\Omega$ is the unit square, using the classical \emph{Haar basis} of wavelets $\left\{\Phi_n^{(0)}\right\}$, 
leads to the dyadic GFF, already introduced ``by hand'' at the very beginning of this work. Passing from the Haar basis
to $\left\{\Phi_n\right\}$, with $\Phi_n=(-\Delta)^{-1/2}\Phi_n^{(0)}$, we have an orthonormal basis of $H^1_0(\Omega)$
that gives the true GFF.

Though the problem of defining rigorously the random Riemannian metric ``$e^{\gamma h}|dz|$'' is still out 
of reach, the multiplicative chaos approach \cite{KP,K} works fine for the random \emph{measure} 
``$e^{\gamma h}d^2z$''. In particular, the celebrated results of Duplantier and Sheffield \cite{DS}, as well 
as the work of Rhodes and Vargas \cite{RV-KPZ}, established mathematically the famous KPZ relation for 
measures, relation first stated in \cite{KPZ} and relating, in a very simple formula, the scaling exponents 
of the random metric (\emph{measure} here) with the Euclidean one (for further reading, we suggest the 
review \cite{bourbaki}). These results have been the starting point for many important others in recent 
years \cite{BJRV,DRSV,BKNSW,QLE,DMS,DKRV} (this list is certainly non-exhaustive).

\subsection{Random metrics and random measures}

Consider the Hausdorff dimension of the obtained random metric space. If (how it is quite natural to 
expect) the random metric is endogenous (in the sense of~\cite{RDE}), the $0$--$1$ law implies 
that this dimension is constant almost surely; denote it by~$\alpha_0$. Consider then 
the associated Hausdorff measure. As a multiplication by $\lcr \xi$ multiplies the 
$\alpha_0$-dimensional Hausdorff measure by the factor of $(\lcr \xi)^{\alpha_0}$, we see
that the total random volume $\mu_H(\Gamma_{\infty})$ satisfies the RDE 
\begin{equation}\label{eq:V}
V=(\lcr \xi)^{\alpha_0} \sum_{i=1}^4 V_i'.
\end{equation}

This is a very classical RDE  (related to Galton-Watson processes). Assume, in addition, that the total 
random volume is almost surely positive, finite and has a finite expectation (it seems reasonable 
assuming that the law of~$\xi$ is not too strongly dispersed). Then, taking the expectation on 
the both sides, we see that a necessary condition for the existence of a stationary 
solution to~\eqref{eq:V} is 
\[
4\, \E (\lcr\xi)^{\alpha_0}=1.
\]
This gives us a (formally speaking, conjectural) relation between the rescaling constant $\lcr$, the rescaling 
law~$m$ and the Hausdorff dimension~$\alpha_0$ of the resulting space. Moreover, it allows us 
to construct the Hausdorff measure via the Mandelbrot Multiplicative Cascade (MMC); again, here 
we have to assume that the measure $m$ is not too strongly dispersed. 

A precise understanding of this problem is very practical for example for establishing the KPZ relation for 
the stationary random metric.

\smallskip

Consider now the RDE
\[
Y=\lambda_{\mathrm{MMC}}(\alpha) \xi^{\alpha} \sum_{i=1}^4 X^i
\]
for other values of $\alpha$. For $\alpha$ sufficiently small, choosing $\lambda_{\mathrm{MMC}}(\alpha)=\frac{1}{4 \E \xi^{\alpha}}$, 
one can construct its stationary solution via the MMC technique, and consider the associated random measure on~$\Gamma_{\infty}$.

Thus we obtain a family of measures on the same space $\Gamma_{\infty}$; they are the analogues of 
the MMC measure for the Liouville quantum gravity measures given 
by $e^{\gamma h} d^2 z$ for different values of $\gamma<2$ for the same Gaussian Free Field~$h$. 
Also, these (``$\alpha$-conformal) measures should be the Hausdorff measures associated to 
different metrics, the ones that appear when we consider glueing with the rescaling by~$\xi^{\gamma}$ 
(for different values of~$\gamma$) instead of~$\xi$.

\subsection{Perspective: a path to the two dimensional world}
In order to give further motivations for the need for a good understanding of the properties of the random 
metrics on the simple examples that we are examining in this paper, let us explain where we are aiming to. 

The passage from the 1D to the higher dimensional setting (see Section~\ref{s:sierpinski}) places us in front of a 
path that could lead to obtain a non-trivial random metric on a surface. As a \emph{caveat}, nothing suggests 
that the analytic difficulties encountered in other well-known approaches (via the 2D GFF or the Brownian map) 
would be overcome.

To be consistent with the hierarchical structure that helped us so far, we shall take the illustrative example 
of the 2D square already treated at the very beginning of our introduction (which is, conformally, a closed 
disk in the plane).

Let us start with a self-similar fractal $\Gamma^{(1)}$, which looks like a square at first approximation, 
and for which the renormalization/cut-off method guarantees the existence of a non-trivial stationary random 
metric $d^{(1)}$. Let us consider then another self-similar fractal $\Gamma^{(2)}$, closer to be a square, 
and a non-trivial stationary random metric $d^{(2)}$.

And so on: if we are able to choose ``consistently'' the sequence of fractal spaces $\Gamma^{(i)}$ 
``approximating'' the square, in such a way that the family of stationary random metrics $d^{(i)}$ is 
relatively compact, then we can extract a limit random metric for the square. This analytic step is by 
far the most difficult.

\smallskip

In Figure~\ref{fig:squares} we exhibit the first steps of what we consider a good candidate for the 
sequence of fractals approximating the square: instead of glueing side-by-side the four squares, we 
only use bridges connecting an increasing finite number of points on different sides (the self-similarity 
forces this to happen at every scale).

\begin{figure}[ht]
\[\includegraphics[width=.4\textwidth]{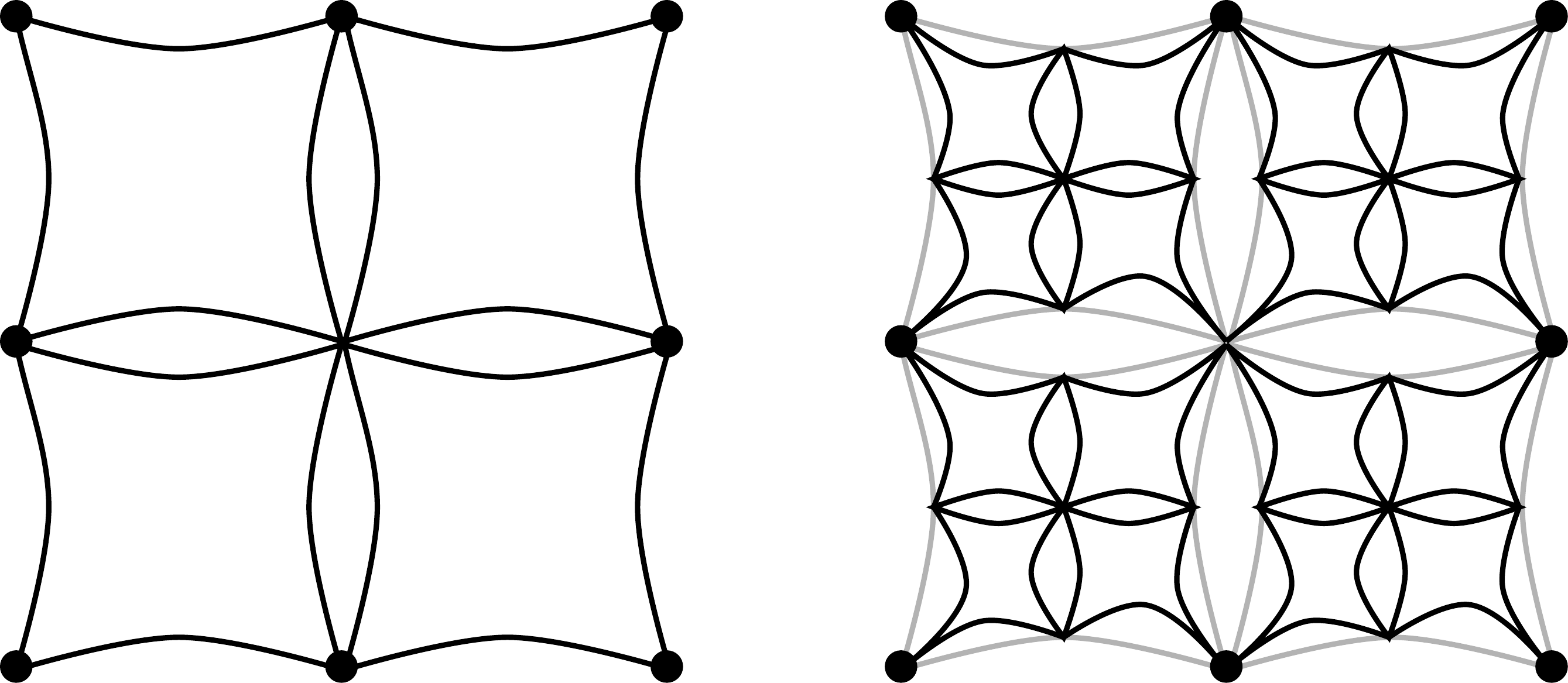}\]
\[\includegraphics[width=.4\textwidth]{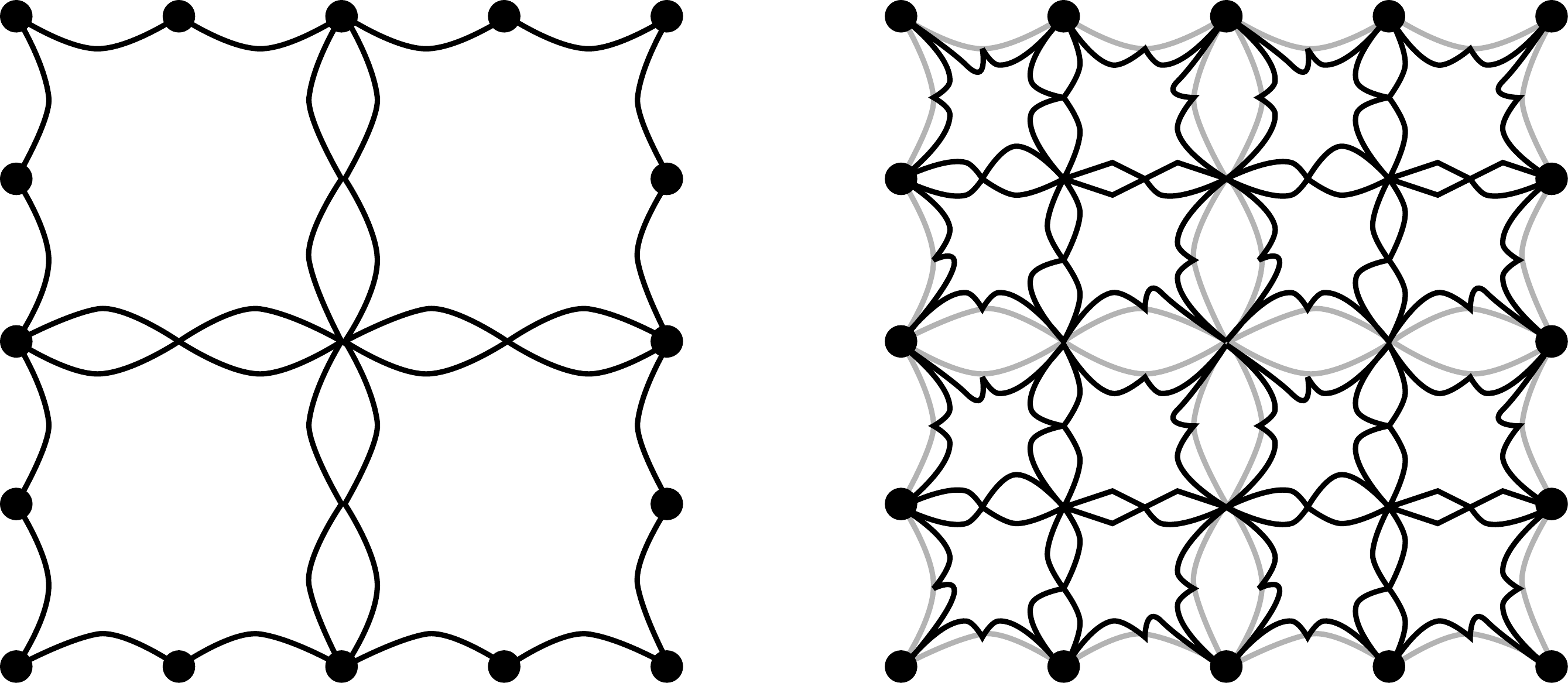}\]
\caption{The construction of square-like self-similar fractals: the sequences defining $\Gamma^{(1)}$ 
and $\Gamma^{(2)}$.}\label{fig:squares}
\end{figure}

\section{Existence of a stationary random metric}\label{s:existence}

\subsection{Construction of the random metric space from the marginal distance}\label{ss:random-metric}

As explained in \S\ref{ss:reduction}, we shall prove that the existence of a stationary random $\In\Out$-distance 
(Theorem~\ref{t:stat}) implies the existence of stationary random metric (Theorem~\ref{p:metric}). We keep the 
notations previously introduced.

\begin{prop}\label{p:equiv}
Let $m$ be a probability measure on $\R_+$. Then the following statements are equivalent:
\begin{enumerate}
\item\label{i:left} There exists a normalizing constant $\lcr>0$ and a non-atomic probability measure $\mbm$ 
on $\MG$ which is a fixed point for the operator $\mPhi_{\lcr}$.
\item\label{i:right} There exists a normalizing constant $\lcr>0$ and a non-atomic probability measure $\bm$ 
on $\R_+$ which is a fixed point for the operator $\Phi_{\lcr}$.
\end{enumerate}
\end{prop}

From what we have explained in \S\ref{ss:reduction} it should be clear that the existence of a stationary 
random metric implies the existence of a fixed point for the marginal renormalization operator, so that we 
only have to prove one half of the statement.

Before going on, it is convenient to set up some more notation. Let~$m$ be a probability measure on~$\R_+$ 
and suppose that~$\bm$ is a fixed point for the operator $\Pc$. Consider the quaternary rooted tree $(\T,\rt)$;
let $p:\T\setminus \{\rt\}\to \T$ be the map that associates to each non-root vertex of the tree 
$\T$ its parent, and let $\|\cdot\|:\T\to\N$ be the distance to the root. 
Following \cite[\S 2.3]{RDE}, we can consider 
the so-called \emph{invariant recursive tree process} (RTP): we have a tree of pairs of positive random 
variables $\left \{\left (X_{t},\xi_{t}\right )\right \}_{t\in\T}$, such that
\begin{itemize}
\item for every $t\in \T$ the law of $\xi_t$ is $m$,
\item for every $t\in \T$ the law of $X_t$ is $\bm$,
\item for every $t\in \T$ the equality 
\[ X_t=R_{\lcr}(X_{t_1},\dots,X_{t_4};\xi_t)\]
holds, where the $t_i$'s are the four descendants of $t$ in $\T$,
\item for every $n\in\N$ the random variables $\{\xi_t\}_{\|t\|\le n}$ and $\{X_t\}_{\|t\|=n}$ are independent 
altogether (in particular the random variables $\{\xi_t\}_{t\in \T}$ are all independent).
\end{itemize}

\smallskip

For any $n$, we can see the limit space $\Gamma_{\infty}$ as glued out of the $4^n$ its $2^{-n}$--rescaled 
copies $\left \{\Gamma_t\right \}_{\|t\|=n}$, corresponding to the edges of~$\Gamma_n$ (which are naturally 
indexed by the vertices in the quaternary tree at distance $n$ from the root). We call them the 
\emph{level $n$ copies} of $\Gamma_{\infty}$.

Given the invariant RTP defined above, we can reverse the bottom-to-top arrow and consider the tree process 
$\{Y_t\}_{t\in\T}$, defined as
\begin{equation}\label{eq:Y-rec}
Y_t=X_t\cdot \prod_{j=1}^{\|t\|} (\lcr \xi_{p^j(t)}).
\end{equation}
These random variables satisfy 
\begin{equation}\label{eq:Y-rec2}
Y_t=R_{1}(Y_{t_1},\dots,Y_{t_4};1),
\end{equation}
and we can interpret them as ``$\In\Out$-distances \emph{inside} $\Gamma_t$''.

\begin{proof}[Proof of \ref{i:right}) $\Rightarrow$ \ref{i:left})]
At the beginning of \S\ref{ss:definition} we used the multiplicative cascade on the hierarchical graph to 
define the sequence of random distance functions $d_n$ on $V_n$ as in First Passage Percolation models 
(\emph{i.e.}~considering the shortest path in the weighted graph). Furthermore, these distances have the 
nice \emph{inductive} property
\begin{equation}\label{eq:Y-rec3}
d_{n'} |_{V_n\times V_n} = d_n\quad\text{a.s.~for any }n'>n,
\end{equation}

As a matter of fact, the tree process $\{Y_t\}_{t\in\T}$ allows to define an inductive sequence of distances 
$\{d_n:V_n\times V_n\to \R_+\}_{n\ge 0}$ in a similar way: for any $t\in\T$ with $\|t\|=n$, set the length of 
the edge in~$\Gamma_n$ corresponding to $t$, to be equal to~$Y_t$, and consider the resulting (FPP) metric 
induced on $\Gamma_n$. Restricting it to $V_n$, we obtain the desired random distance function~$d_n$. Moreover, 
the relation~\eqref{eq:Y-rec2} implies that these distances agree with each other (in the sense of \eqref{eq:Y-rec3}).

Thus, there is a random distance function $d_{\infty}:V_{\infty}\times V_{\infty}\to \R_+$ such that 
$d_{\infty} \vert_{V_n\times V_n}=d_n$ almost surely for any $n$, and completing the random metric space 
$(V_{\infty},d_{\infty})$, we obtain the desired random metric space $\mX$, which belongs to~$\MG$ by 
construction. Let us write $\mbm$ for the law of~$\mX$: such a measure has to be $\mPhi_{\lcr}$-stationary 
because the RTP that we are considering is invariant.
\end{proof}

\begin{rem}
From the previous proof, we can observe that the law of any finite-dimensional restriction 
$d_{\infty} \vert_{V_n\times V_n}$ is uniquely determined as function of $n$ and the law $\bm$ only, 
so that we deduce that there is a unique measure $\mbm$ such that the distance $d_\infty(\In,\Out)$ 
has marginal distribution~$\bm$.
\end{rem}

The remaining part of this section deals with the proof of Theorem~\ref{t:stat}.

\subsection{More notations and definitions}\label{sc:morenotations}
The operator $\Phi_\lambda$ defined by \eqref{eq:Phi} induces a 
continuous dynamics on the 
space $\cP$ of probability measures on the extended half-line $[0,\infty]$.
The space $\cP$ carries a natural partial order $\preccurlyeq$, which is the well-known \emph{stochastic domination}:
\begin{dfn}
Given $\mu$ and $\nu$ in $\cP$, we write $\mu\preccurlyeq \nu$ if there is a coupling $(X,Y)$ of these measures (that 
is, a random vector $(X,Y)$ with marginal laws $\mu$ and $\nu$ respectively), in such a way that $X\le Y$ almost 
surely.
\end{dfn}
We also define $\cP_0\subset \cP$ to be subspace of probability measures which have no atoms at $0$ nor $\infty$. 
Then every probability measure $\mu\in \cP$ can be uniquely decomposed into the convex combination
\begin{equation}\label{eq:convex-combination}
\mu=p_0(\mu)\ddelta_0+p_\infty(\mu)\ddelta_\infty+\left (1-p_0(\mu)-p_\infty(\mu)\right )\mu_0,
\end{equation}
of its (eventual) atoms at $0$ and $\infty$ with the respective weights $p_0(\mu), p_{\infty}(\mu)\in [0,1]$ and 
of the remaining component $\mu_0\in \cP_0$.

The stochastic domination turns out to be helpful in many cases when dealing with RDEs (see~\cite[\S 2.2]{RDE}). 
It is a crucial remark that this partial order is well adapted to the dynamics induced by the operator $\Phi_\lambda$.

\begin{lem}\label{lem:order}~\begin{enumerate}
\item For any $\lambda>0$, the operator $\Phi_\lambda$ is \emph{order-preserving}: if
$\mu\preccurlyeq \nu$ then $\Phi_\lambda[\mu]\lo \Phi_\lambda[\nu]$.
\item For any $\lambda \le\lambda'$ and 
$\mu\in \cP$ we have 
$\Phi_\lambda[\mu]\preccurlyeq \Phi_{\lambda'}[\mu]$.
\end{enumerate}
\end{lem}

\begin{proof}
Let us prove the first statement only, since the second one is rather evident.
Suppose that~$X_i^\mu$ and 
$X_i^\nu$, $i=1,\ldots,4$ are respectively four independent random variables distributed 
according to $\mu$ and $\nu$ and defined on the same probability space 
$(\Omega,\mathbf{P})$, in such a way that for every $i$ and almost every 
$\omega\in \Omega$ the relation $X_i^\mu(\omega)\le X_i^\nu(\omega)$ holds. Then the claimed 
inequality easily follows: given another independent random variable $\xi$ of law $m$,  for 
almost every $\omega\in \Omega$ we have
\begin{align*}
& \lambda\xi(\omega) \,\left (\min\left (X_1^\mu(\omega),X_2^\mu(\omega)\right )
+\min\left ( X_3^\mu(\omega),X_4^\mu(\omega)\right )\right )  \\
\le \,& \lambda\xi(\omega) \,\left (\min\left (X_1^\nu (\omega),X_2^\nu (\omega)\right )
+\min\left ( X_3^\nu (\omega),X_4^\nu (\omega)\right )\right )  .
\end{align*}
In other words,  
$\Phi_{\lambda}[\mu]\lo \Phi_{\lambda}[\nu]$ as desired.
\end{proof}

\begin{rem}
The reader can verify in the same way that also the rescaling operators $\Upsilon_c$ preserve the stochastic
domination.
\end{rem}

What is commonly known as \emph{Strassen's Theorem} \cite{strassen,kamae}, asserts that the 
condition $\mu\lo\nu$ is equivalent to the fact that for any \emph{increasing} bounded real valued-function~$f$ 
on~$[0,+\infty]$,
\[
\int_{[0,\infty]} f\,d\mu\le \int_{[0,\infty]} f\,d\nu.
\]
Taking functions $f$ of the form $\mathbf{1}_{(x,\infty]}$ ($x\in \R_+$) it is easy to see that the inequality 
$\mu\lo\nu$ coincides with the \emph{reversed} inequality $F_\mu\ge F_\nu$ for distribution functions.
The  point of view of distribution functions will be sometimes very useful, as well as the following function 
$\theta:[0,1]\to [0,1]$ that is naturally associated to the figure eight-graph $\Gamma$.

\begin{dfn}\label{dfn:psi}
Take a figure-eight graph and for each of its four edges decide randomly and independently, whether to keep it or to 
remove it. Denote by~$\theta(p)$ the probability that there is at least one $\In\Out$-path, if the edges are kept with 
probability~$p$.
\end{dfn}
An easy computation shows that $\theta(p)=p^2(2-p)^2$. Note that the map $\theta$ on 
$[0,1]$ has two attracting fixed points, $0$ and $1$ (moreover, these points are 
\emph{super-attracting}: $\theta'(0)=\theta'(1)=0$), and one repelling fixed point 
$p_{cr}\in (0,1)$. Actually, it is easy to find that $p_{cr}=\varphi^{-2}$, where 
$\varphi\,=\frac{1+\sqrt{5}}{2}$ is the golden ratio.

The following lemma immediately relates this function to the dynamics of $\Phi_{\lambda}$:
\begin{lem}\label{l:psi}
For any $\lambda>0$ and $\mu\in \cP$ we have
\[
p_0\left ( \Phi_{\lambda}[\mu] \right )=\theta\left (p_0( \mu)\right ), \quad 1-p_\infty\left ( 
\Phi_{\lambda}[\mu] \right )=\theta\left (1-p_\infty( \mu)\right ).
\]
\end{lem}
\begin{proof}
Considering the figure eight-graph $\Gamma$, we look for the probability that there exists a zero-length 
$\In\Out$-path when assigning random distances to the edges according to the probability distribution~$\mu$.
This means that we look for a path from $\In$ to $\Out$ passing from zero-length edges only. The definition 
of the function $\theta$ then shows the first equality. 
For the second one, we can argue in the same way, replacing the word \emph{zero} by \emph{finite}.
\end{proof}

\subsection{Cut-off process}

The main difficulty in finding a stationary measure is that the metric could 
blow up or collapse at different scales. In order to tame such problems, we artificially impose a 
``regular'' behaviour to the random metric, using cut-off procedures.

\smallskip

For this reason we introduce the following 
operator that forces the support to be included in a subinterval. For any~$A\in\R_+$, define the cut-off 
operator~$\Phi_{A,\lambda}$ by assigning to any probability measure $\mu\in \cP$, the law of
\[
\min(R_{\lambda}(X_1,\dots,X_4;\xi),A),
\]
where the $X_i$'s are i.i.d.~variables with law $\mu$, and 
$\xi$ is distributed with respect to~$m$ and independent of them. Geometrically, we are adding an $\In\Out$-shortcut
of length $A$, when glueing together four independent samples of the space. The reader will remark that also 
the new operator~$\Phi_{A,\lambda}$ preserves the partial order on $\cP$ 
due to the same coupling arguments as those used in the proof of Lemma~\ref{lem:order}.

\begin{rem}\label{rem:coupling-cutoff}
It should be evident that for any $\lambda>0$ and $A\in \R_+$ we have $\Phi_\lambda\go \Phi_{A,\lambda}$. 
To check this, we construct the following natural coupling: fix a measure $\mu\in \cP$ and let $Y$ be a 
random variable of law $\Phi_\lambda[\mu]$ on the probability space $(\Omega,\mathbf{P})$. We can define 
the new random variable
\[Y^A(\omega)=\min(A, Y(\omega) ),\]
whose law is $\Phi_{A,\lambda}[\mu]$. The inequality $Y\ge Y^A$ implies $\Phi_\lambda\go \Phi_{A,\lambda}$, as claimed.
\end{rem}

An important property of the cut-off process is that, as it is very natural to expect and as we 
will show later (see Lemma~\ref{l:Lambda}), it has a non-trivial stationary distribution for 
all sufficiently large~$\lambda$'s. This makes the situation more ``flexible'': for this 
particular process, we do not have to look for a precise value of $\lambda$ where it does 
neither explode nor collapse. Instead (as we will show later) non-trivial stationary measures for 
this process exist for all sufficiently large values of~$\lambda$.

\subsubsection{Definition of $\nu_{A,\lambda}$}

Consider the following sequence $\{\mu^{A,\lambda}_{n}\}_{n\in\N}$ in 
$\cP$. Start with $\mu^{A,\lambda}_{0}:=\ddelta_\infty$, which is the greatest probability measure in $\cP$. Then,
apply $\Phi_{A,\lambda}$ repeatedly in order to get a monotone 
decreasing sequence in $\cP$.
That is, we set
\begin{equation}\label{eq:def-G}
\begin{cases}
\mu^{A,\lambda}_0
= \ddelta_\infty \\
\mu^{A,\lambda}_{n+1} := \Phi_{A, \lambda}[\mu^{A,\lambda}_{n}] & \textrm{for }n\in \N.
\end{cases}
\end{equation}
\begin{lem}\label{l:nu}
The sequence $\{\mu^{A,\lambda}_{n}\}_{n\in\N}$ is monotone decreasing and hence converges (in the weak-$*$ topology) 
to a certain probability measure.
\end{lem}
\begin{dfn}\label{def:nu}
We denote the limit of the sequence $\{\mu^{A,\lambda}_{n}\}_{n\in\N}$ by $\nu_{A,\lambda}$:
$$
\nu_{A,\lambda}:=\lim_{n\to\infty} \mu^{A,\lambda}_{n}.
$$
\end{dfn}
\begin{proof}[Proof of Lemma~\ref{l:nu}]
Since $\mu_0^{A,\lambda}$ is the greatest probability measure in $\cP$, we must have 
$\mu^{A,\lambda}_0 \go \mu^{A,\lambda}_1$. Then for any~$n\in\N$, the relation 
$\mu^{A,\lambda}_n \go \mu^{A,\lambda}_{n+1}$ recursively holds, applying 
Lemma~\ref{lem:order}.
This implies that the weak-$*$ limit $\nu_{A,\lambda}$ of the sequence $\left \{
\nu^{A,\lambda}_n\right \}_{n\in\N}$ exists, for the sequence of distribution functions of $\mu^{A,\lambda}_{n}$ 
is pointwise monotone, and hence pointwise converges.
\end{proof}
\begin{rem}This monotone construction 
is somehow classical and gives the greatest $\Phi_{A,\lambda}$-invariant measure. The reader can compare for 
example \cite[\S 2.2]{RDE}.
\end{rem}
To verify that $\nu_{A,\lambda}$ is a fixed point for the operator $\Phi_{A,\lambda}$, we only need the additional 
elementary property.
\begin{lem}
The operator $\Phi_{A,\lambda}$ is continuous.
\end{lem}
\begin{proof}
It is a composition of the map $\Phi_\lambda$, associating to a measure $\mu$ the measure 
$(R_\lambda)_*\left (\mu^{\otimes 4} \otimes m\right )$, and the pushforward by the continuous map 
$
\min (\,\cdot\, , A).
$
Both these operations are clearly continuous in the sense of $*$-weak convergence for measures on~$[0,\infty]$ 
(note that the ``$+$'' and ``$\min$'' operations are continuous even on the compactified half-line~$[0,\infty]$).
\end{proof}
\begin{rem}\label{r:cont-three}
In fact, the same argument shows that the map $(A,\lambda,\mu)\mapsto \Phi_{A,\lambda}[\mu]$ is continuous in 
all the three variables $A\in [0,\infty]$, $\lambda\in \R_+$, $\mu\in\cP$, with $A=\infty$ corresponding to 
$\Phi_{\infty,\lambda}[\mu]=\Phi_{\lambda}[\mu]$.
\end{rem}
Note that for every $n\ge 1$, the support of $\mu^{A,\lambda}_{n}$ and hence of $\nu_{A,\lambda}$
is contained in $[0,A]$. We remark also that since 
$\nu_{A,\lambda} = \Phi_{A, \lambda}[\nu_{A,\lambda}]$, the measure $\nu_{A,\lambda}$ has no atom on the 
open interval~$(0,A)$.

\begin{rem}\label{r:scaling}
The definition of $\nu_{A,\lambda}$ certainly depends  on the cut-off value $A$, but the condition  
$\nu_{A,\lambda}=\ddelta_0$ does not.
Indeed, for any $A$, $A'\in\R_+$, the rescaling operator $\Upsilon_{A/A'}$ fixes the probability measure 
$\ddelta_0$ and \emph{conjugates} $\Phi_{A,\lambda}$ to $\Phi_{A',\lambda}$: since
\[
\min\left (R_\lambda\left (\tfrac{A}{A'}X_1,\ldots,\tfrac{A}{A'}X_4;\xi\right ),A\right )=
\tfrac{A}{A'}\min \left (R_\lambda\left (X_1,\ldots,X_4;\xi\right ),A'\right ),
\]
we can write
\[\Phi_{A,\lambda}\Upsilon_{A/A'}=\Upsilon_{A/A'}\Phi_{A',\lambda}.\]
Iterating this equality and taking $\mu=\ddelta_\infty$, we get
\beqn{scaling}
{
\Upsilon_{A/A'}[\nu_{A',\lambda}]=\nu_{A,\lambda}.
}
\end{rem}

\subsubsection{Geometrical construction via RTP}
Relying on the notion of invariant RTP, we have a nice interpretation for the construction of the measure 
$\nu_{A,\lambda}$, whose main consequence at this stage is the possibility to apply Kolmogorov's $0$--$1$ 
law in order to get a dichotomy for the measures $\nu_{A,\lambda}$ (see Lemma~\ref{l:G-0-1}).

As before, let $\T$ be the rooted quaternary tree, $\rt$ its root, and let $\at:\T\to \R_+$ be a 
function associating to any  vertex a positive number. Let us put $\at$ into correspondence 
with a family of functions $D_n=D_n(\at):\T\to (0,\infty]$ defined inductively backwards as 
\[
D_n(t)=\begin{cases}
\small{\infty} & \footnotesize{\textrm{if }t\textrm{ is at distance at least }n
\textrm{ from the root, otherwise}}\\
\\
\small{\min(R_\lambda(D_n(t_1),\dots,D_n(t_4);\at(t)),A)}& 
\footnotesize{\textrm{where }t_1,\dots,t_4\in \T\textrm{ are the four descendants of~}t.}
\end{cases}
\]
In other words, $\left \{D_n\right \}_{n\in\N}$ is the family of distances obtained in the following way: everything 
deeper than~$n$ levels is declared to be infinite, while closer to the root we are applying the 
procedure indicated by the figure~eight-graph with the factor $\at$ at the corresponding vertex, 
adding a shortcut of length~$A$. Then we have the following 
\begin{prop}\label{p:cut-off}
The family of functions $D_n(\at)$ is pointwise monotonely decreasing with respect to~$n$, and their limit 
$D_{\infty}(\at)(t)=\lim_{n\to\infty}D_n(\at)(t)$ is, for every $\at$, a 
map from $\T$ to $[0,A]$, satisfying the following properties:
\begin{itemize}
\item Let $t\in \T$ be a vertex, and let $\at_t$ be the restriction of $\at$ on the subtree 
$\T_t$ of $\T$ rooted at $t$. Then the restriction of $D_{\infty}(\at)$ on $\T_t$ coincides 
with $D_{\infty}(\at_t)$. Moreover one has
\beqn{tree}{
D_{\infty}(\at)(\rt)=\min( R_\lambda(D_\infty(\at_1)(\rt_1),\dots,D_\infty(\at_4)(\rt_4);
\at(\rt)), A),
}
where $\at_i$ is the function $\at_{\rt_i}$, with $\rt_1,\ldots,\rt_4$ the direct descendants 
of the root $\rt$.
\item If $\left \{\at_t \right \}_{t\in\T}$ are i.i.d.~random variables with law $m$, then the law
of $D_{\infty}(\rt)$ is~$\nu_{A,\lambda}$.
\end{itemize}
\end{prop}
\begin{proof}
Let us fix $n\in \N$, we want to prove that $D_n(t)\le D_{n-1}(t)$ for every $t\in \T$. If the 
depth of $t$ is larger than or equal to $n$, this is evident; proceeding inductively 
upwards, if $t$ is at level $n-k$, we have
\[
\min( R_\lambda(D_n(t_1),\dots,D_n(t_4);\at(t)),A)\le
\min(R_\lambda(D_{n-1}(t_1),\dots,D_{n-1}(t_4);\at(t)),A),
\]
as desired. Hence, the pointwise limit $D_{\infty}(\at)(t)$ exists. The first claimed property 
may be easily verified, since the definition of $D_n(t)$ only depends on the values of $\at$ 
on the subtree rooted at $t$.

To prove the second one, we observe that the sequence of random variables 
$\left \{D_n(\bullet)\right \}_{n\in \N}$ is related to the sequence of  
measures $\left \{\mu^{A,\lambda}_{n}\right \}_{n\in\N}$. Indeed, 
the law of $D_0(\bullet)$ is the atomic mass $\ddelta_\infty$ and we can argue 
by induction, using the relation \eqref{eq:tree}, to see that $\mu^{A,\lambda}_{n}$ is the law of $D_n(\bullet)$.
\end{proof}

An immediate corollary of this representation and the $0$--$1$ law is the following

\begin{lem}\label{l:G-0-1}
For any $\lambda>0$ the mass $p_0(\nu_{A,\lambda})$ equals either zero or one, and this value does not depend 
on the cut-off value $A\in\R_+$.
\end{lem}

\begin{proof}
The event $\{D_{\infty}(\bullet)=0\}$ is a tail event for the family of random 
variables~$\left \{\at_t\right \}_{t\in \T}$; indeed, altering the values of the multipliers $\at_t$ on 
any finite set of vertices can change the value of the distance $D_{\infty}(\bullet)$, but not 
the fact that this distance vanishes. Hence, due to the Kolmogorov's $0$--$1$ law, the 
probability of this event, which is exactly $\nu_{A,\lambda}(\{0\})=p_0(\nu_{A,\lambda})$,
equals either~$0$ or~$1$.
The last statement is a consequence of Remark~\ref{r:scaling}.
\end{proof}

\subsubsection{Definition of the supercritical set $\Lambda$} 
We now begin to implement the approach described in~\S{}\ref{s:cut-sketch}. For this,
we consider the set $\Lambda$ formed by all $\lambda>0$ for which we can obtain a 
non-trivial limit measure~$\nu_{A,\lambda}$,
for some $A\in\R_+$ (and hence for every, cf.~Remark~\ref{r:scaling}):
\begin{align*}
\Lambda &= \{\lambda\mid \exists\, A\in\R_+ \textrm{ such that } \nu_{A,\lambda}
\neq \ddelta_0\}\\
&=\{\lambda\mid \forall\, A\in\R_+ \textrm{ one has }\nu_{A,\lambda} 
\neq \ddelta_0\}.
\end{align*}
We call \emph{supercritical} such values of $\lambda$. An easy application of the second statement in 
Lemma~\ref{lem:order} shows (quite naturally) that  for any $\lambda'>\lambda$ we have 
$\nu_{A,\lambda} \lo \nu_{A,\lambda'}$, and hence if $\lambda\in\Lambda$, then $\lambda'\in \Lambda$. So 
$\Lambda$ can be either empty, or~$(0,+\infty)$, 
or $[\lcr,\infty)$ or $(\lcr,\infty)$ for some $\lcr>0$.

Intuitively, we know that for large values of $\lambda$ the distances in the limit graph~$\Gamma_\infty$ explode, 
and therefore for such $\lambda$, there should be a non-trivial measure $\nu_{A,\lambda}$, generated by a 
balance between the expansion and 
the cut-offs. The opposite situation should arise when $\lambda$ is very close to $0$, as in this case the 
diameter of $\Gamma_\infty$ should be equal to~$0$, and hence~$\nu_{A,\lambda}=\ddelta_0$. Thus it is 
natural to look for a \emph{phase-transition parameter}~$\lambda_{cr}$.

\begin{lem}\label{l:Lambda}
The set $\Lambda$ is nonempty, as well as its complement. In other words, 
$\Lambda$ is a nonempty half-line, starting at a positive real number $\lcr$.
\end{lem}

\begin{rem}
This result can be achieved more directly when the law $\mu$ has finite first positive and negative 
moments (see the sketches of the proofs of Theorems~\ref{t:pivotal-bridge} and~\ref{t:pivotal-shortcut}).
\end{rem}

Before passing to the proof of Lemma~\ref{l:Lambda}, we introduce the following tool. Roughly speaking, 
it is a measure that ``prevents'' the sequence $\mu^{A,\lambda}_n$ from concentrating at $0$ (and hence 
$\nu_{A,\lambda}$ from being~$\ddelta_0$):
\begin{dfn}\label{def:down}
We say that a probability measure $\mu\in \cP$ is $\lambda$-\emph{zooming out} if:
\begin{itemize}
\item 
$\mu \lo\Phi_{\lambda}[\mu]$,
\item $\mu$ is not the Dirac mass at $0$,
\item $\mu$ is supported on some closed subinterval $[0,x]$, with $x\in \R_+$.
\end{itemize}
\end{dfn}

Then, we have the following
\begin{lem}\label{l:out}
The following three conditions are equivalent:
\begin{enumerate}
\item\label{i:supercr} $\lambda$ is supercritical: $\nu_{A,\lambda}\neq\ddelta_0$;
\item\label{i:z-o} there exists a $\lambda$-zooming-out measure;
\item\label{i:z-o-strict} there a $\lambda$-zooming-out measure, having no atom at~$0$.
\end{enumerate}
\end{lem}
\begin{proof}
As the measure $\nu_{A,\lambda}$ is supported on $[0,A]$, when it is not equal to $\ddelta_0$, it is
a $\lambda$-zooming out measure. Also, by Lemma~\ref{l:G-0-1}, in this case $p_0(\nu_{A,\lambda})=0$, 
and hence \ref{i:supercr}) implies both \ref{i:z-o}) and~\ref{i:z-o-strict}). The implication 
\ref{i:z-o-strict}) $\Rightarrow$ \ref{i:z-o}) is immediate. Finally, assume~\ref{i:z-o}): let there 
be a $\lambda$-zooming out measure $\mu$. 
Then it is easy to show that the sequence of measures $\{\mu_n^{A,\lambda}\}$ 
defined by~\eqref{eq:def-G} satisfies
\begin{equation}\label{eq:mu-nu}
\mu\lo \mu_n^{A,\lambda}\quad\text{ for every }n.
\end{equation}
Indeed, for $n=0$ we have $\mu \lo \ddelta_{\infty} = \mu_0^{A,\lambda}$.
Now, if~\eqref{eq:mu-nu} holds for some $n$, applying $\Phi_{A,\lambda}$ we get
\[
\mu\lo \Phi_\lambda[\mu] \lo \Phi_\lambda[\mu_n^{A,\lambda}]=\mu_{n+1}^{A,\lambda}.
\]
The induction argument then shows that~\eqref{eq:mu-nu} holds for all $n$, and passing to the limit 
as $n\to\infty$ we obtain $\mu\lo \nu_{A,\lambda}$. Hence, $\nu_{A,\lambda}\neq \ddelta_0$.
\end{proof}

\smallskip

\begin{proof}[Proof of Lemma~\ref{l:Lambda}]
Instead of using directly the definition of supercriticality, we will use different equivalent conditions 
from Lemma~\ref{l:out}.

\paragraph{Non-emptiness of $\Lambda$.} Choose any $p\in (0,p_{cr})$ and 
consider the probability measure $\mu\in \cP$
defined as \mbox{$\mu=p\cdot \ddelta_0+(1-p)\cdot \ddelta_1$.}
We want to show that such $\mu$ is $\lambda$-zooming out
for a certain~$\lambda>0$, thus showing that $\lambda\in\Lambda$.

This time it is more convenient to work with the distribution function of $\mu$
\[F_\mu=p\cdot \I_{x<1}+\I_{x\ge 1}.\]
In order to keep notations not too heavy, we prefer writing $F$ for $F_\mu$ and $\Phi_\lambda[F]$ instead 
of $F_{\Phi_\lambda[\mu]}$.
\begin{rem}\label{r:distribution}
This abuse of notation will appear often in the rest of the paper: we will specify the probability measure 
associated to a distribution function only when this choice could generate some confusion.
\end{rem}
Using Lemma~\ref{l:psi} we remark that $\Phi_{1}[F](0)=\theta(F(0))=\theta(p)$; 
since $p$ is smaller than~$p_{cr}$ and~$0$ is an attracting fixed point for $\theta$, we get 
$\Phi_{1}[F](0)<F(0)$, for $F(0)$ is equal to $p$. We observe also that for any $\lambda>0$,  
$\Phi_{\lambda}[F](1)$ is equal to $\Phi_1[F](\lambda^{-1})$ and as $\lambda$ goes to 
infinity,~$\Phi_{\lambda}[F](1)$ tends to~$\theta(p)=\Phi_1[F](0)$. This means that for some 
sufficiently large $\lambda$ we must have~$\Phi_{\lambda}[F](1)< p$. Taking such 
a $\lambda$, we notice that for $x<1$ we have 
\[
\Phi_\lambda[F](x)\le \Phi_\lambda[F](1)<p=F(x),
\]
and similarly, when $x\ge 1$, we find 
\[
\Phi_\lambda[F](x)\le 1=F(x).
\]
Therefore, we conclude that $\lambda$ belongs to $\Lambda$, and hence that $\Lambda$ is nonempty.

\begin{figure}[ht]
\[
\includegraphics[width=.5\textwidth]{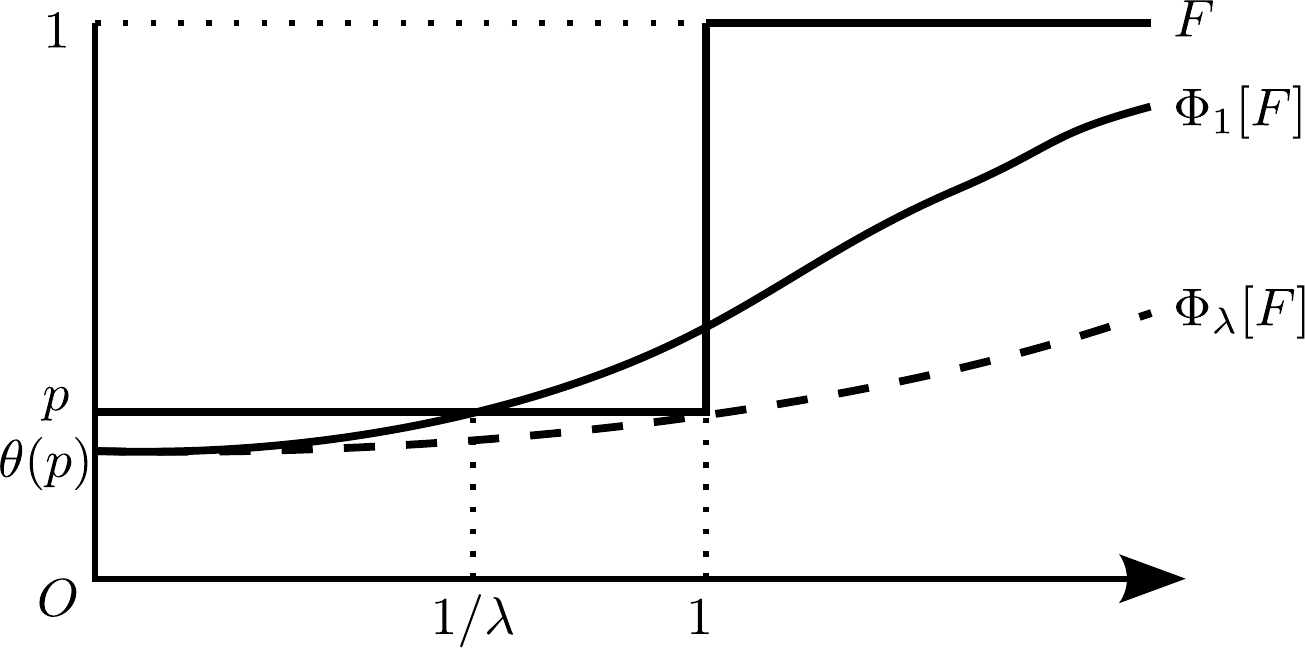}
\]
\caption{The proof of the existence of a $\lambda$ such that there is a 
$\lambda$-zooming out measure.}
\end{figure}

\paragraph{Non-emptiness of $(0,+\infty)\setminus\Lambda$.} Take any 
$q\in (p_{cr},1)$ and consider the probability measure
\[\mu_-=q\cdot \ddelta_1+(1-q)\cdot \ddelta_\infty,\]
whose distribution function is $F_-=q\cdot \I_{x\ge 1}+(1-q)\cdot \I_{\infty}$. 

Again Lemma~\ref{l:psi} gives $\lim_{x\rightarrow+\infty}\Phi_1[F_-](x)=\theta(q)$ and so 
$\lim_{x\rightarrow+\infty}\Phi_1[F_-](x)$ is 
larger than~$q$. Proceeding as before, we see that for some sufficiently small $\lambda$ 
we have $\Phi_{\lambda}[F_-](1)>q$. Hence, by the very expression of $F_-$, we have just 
shown that $F_-$ moves up under $\Phi_{\lambda}$:
\[
\Phi_{\lambda}[F_-]\ge F_-\quad\text{(or }\Phi_{\lambda}[\mu_-]\lo \mu_-\text{)}.
\]
Moreover due to the scale-invariance, given any positive $c$, we have 
$\Phi_{\lambda}[F_c]\ge F_c$ for the function~$F_c$ defined 
by the rescaling~\mbox{$F_c(x)=F_-(x/c)$} (and associated to the probability measure $\mu_c$ defined 
by~$\mu_c=\Upsilon_c[\mu_-]$):
\[\Phi_\lambda[\mu_c]=\Phi_{\lambda}\Upsilon_c[\mu_-]=\Upsilon_c\Phi_{\lambda}[\mu_-]\lo \Upsilon_c[\mu_-]=\mu_c.
\]

We want to show that any sufficiently small $\lambda$ does not belong to 
$\Lambda$, by showing that the condition~\ref{i:z-o-strict}) of Lemma~\ref{l:out} cannot be satisfied. 
To do so, fix any probability measure $\mu\in \cP_0$ and denote by $F=F_\mu$ its distribution function.
We shall show that $\mu$ is not $\lambda$-zooming out. Define $c$ to be its $q$-quantile: 
\[
c=\min\{x\,:\,F(x)=\mu([0,x])\ge q\}.
\]
Then for such $c$, we tautologically have $F(x)\ge F_c(x)$; by the monotonicity of 
$\Phi_\lambda$, we get the relation $\Phi_{\lambda}[F] \ge \Phi_{\lambda}[F_c]$. 
Now, if $\mu$ is $\lambda$-zooming out, we shall have $q\ge \Phi_\lambda[F](c^-)$. But this 
is not possible, since $\Phi_{\lambda}[F](c^-)$ is larger than $\Phi_{\lambda}[F_c](c^-)$ 
which is larger than $q$.
\end{proof}

\begin{figure}[ht]
\[
\includegraphics[width=\textwidth]{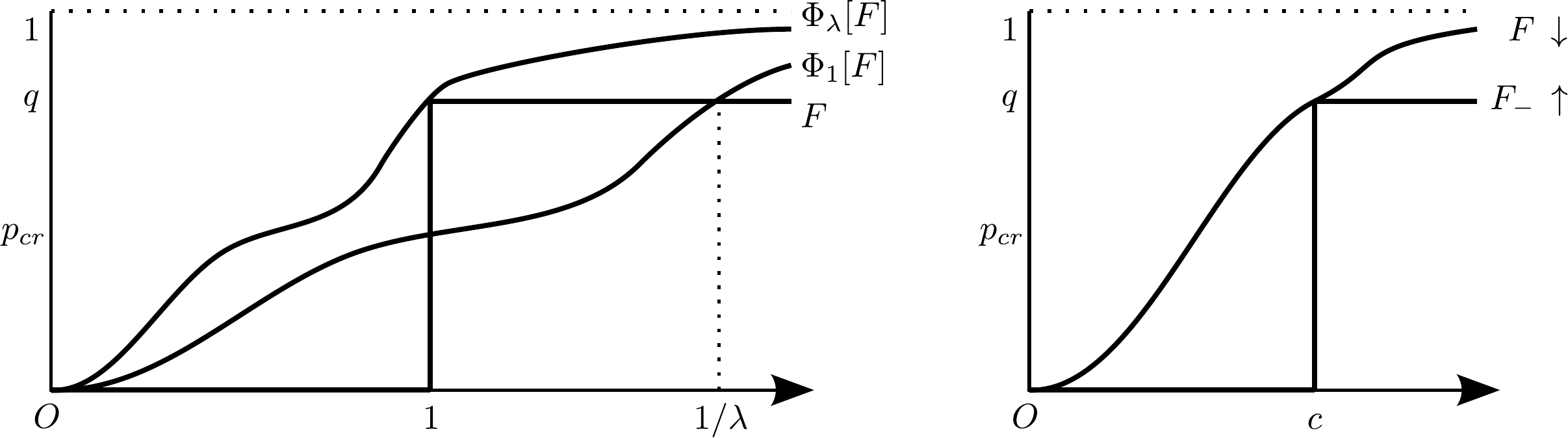}
\]
\caption{The proof that the complementary set of $\Lambda_2$ is nonempty.}
\end{figure}

The following lemma completes the description of the set $\Lambda$.

\begin{lem}[Key Lemma]\label{l:open}
The set $\Lambda$ is open; thus, $\Lambda = (\lcr, +\infty)$ for some 
$\lambda_{cr}>0$\,. 
\end{lem}

We first deduce Theorem~\ref{t:stat} from it:

\begin{proof}[Proof of Theorem~\ref{t:stat}]
The proof consists of two parts: we begin showing the existence of a non-trivial stationary measure 
(and the technical arguments are hidden in the Key Lemma \ref{l:open}), then we explain that such 
non-trivial stationary measure cannot have atoms on $[0,+\infty]$. The latter part of the proof is 
longer and rather technical (rather due to the need of handling the point $0$), but contains 
essentially the arguments that we shall use when proving Lemma \ref{l:open}.

\paragraph{Stage 1:} \textit{Existence of a non-trivial stationary measure.} The first 
observation is that the measures~$\nu_{A,\lambda}$ degenerate as $\lambda\in \Lambda$ approaches the 
boundary value $\lcr$:
\begin{lem}\label{l:to-zero}
For any given $A\in\R_+$ we have 
\beqn{collapse}
{\nu_{A,\lambda}\to\ddelta_0 \, \text{ as } \lambda\searrow\lcr.
}
\end{lem}
\begin{proof}
As $\lcr\notin\Lambda$, for any given $A\in\R_+$ we have $\nu_{A,\lcr}=\ddelta_0$. On the other hand, 
for any given~$A\in \R_+$ the family of measures $\mu^{A,\lambda}_n$ is $\lo$-decreasing \emph{both} 
as $n$ tends to~$\infty$ and as $\lambda$ tends to~$\lcr$ from the right. Hence, the double limit 
$\lim_{n\to\infty, \, \lambda \searrow\lcr} \, \mu^{A,\lambda}_n$ exists, and is equal to each of the 
two repeated limits:
\begin{equation}\label{eq:double}
\lim_{n\to\infty, \, \lambda\searrow\lcr} \, \mu^{A,\lambda}_n = 
\lim_{n\to\infty} \lim_{\lambda \searrow\lcr} \mu^{A,\lambda}_n = 
\lim_{\lambda \searrow\lcr} \lim_{n\to\infty} \mu^{A,\lambda}_n.
\end{equation}
Recall that by definition $\mu^{A,\lambda}_n= \Phi_{A,\lambda}^n[\ddelta_{\infty}]$. 
Due to the continuity of the transformation $\Phi_{A,\lambda}[\mu]$ in both $\lambda$ and
$\mu$ (Remark~\ref{r:cont-three}), for any fixed $n$ one has $\mu^{A,\lambda}_n\to \mu^{A,\lcr}_n$ as
$\lambda\searrow\lcr$. Hence, the first of the two repeated limits is equal to $\nu_{A,\lcr}=\ddelta_0$:
\[
\lim_{n\to\infty} \lim_{\lambda \searrow\lcr} \mu^{A,\lambda}_n = \lim_{n\to\infty} \mu^{A,\lcr}_n = \nu_{A,\lcr}
\]
The other repeated limit in~\eqref{eq:double} is 
\[
\lim_{\lambda\searrow\lcr} \lim_{n\to\infty} \mu^{A,\lambda}_n = \lim_{\lambda\searrow\lcr} \nu_{A,\lambda}.
\]
Hence, for any given $A$ we have $\nu_{A,\lambda}\to\ddelta_0$ as $\lambda\searrow\lcr$. 
\end{proof}

Now, let us properly rescale these measures, in order to keep them non-trivial in the (subsequential) limit. 
Namely, consider these measures for the particular choice $A=1$, and take the $\hf$-quantile 
\[
\kappa_{\hf}(\lambda) := \min\left \{x \,:\, \nu_{1,\lambda}([0,x])\ge \tfrac12\right \}.
\]
\begin{rem}
The value $\hf$ is chosen here for simplicity; as the reader will see later, we could replace it by any value 
in the interval $(p_{cr},1)$.
\end{rem}

From Lemma~\ref{l:to-zero} we have
\begin{prop}
\begin{equation}\label{eq:median}
\lim_{\lambda\searrow\lcr}\kappa_{\hf}(\lambda)= 0.
\end{equation} 
\end{prop}
\begin{proof}
For any $\eps>0$, the weak convergence~\eqref{eq:collapse} implies that $\nu_{1,\lambda}([0,\eps])\to 1$ 
as $\lambda\searrow \lcr$, and hence~$\kappa_{\hf}(\lambda)<\eps$ for all $\lambda$ in some right 
neighbourhood of $\lcr$.
\end{proof}

The absence of atoms for $\nu_{1,\lambda}$ on $[0,1)$ implies that $\kappa_{\hf}(\lambda)>0$ for 
any $\lambda$. Therefore, once $\kappa_{\hf}(\lambda)<1$ (which holds in some right neighbourhood 
of~$\lcr$), we have the equality $\nu_{1,\lambda}([0,\kappa_{\hf}(\lambda)])=\hf$. 

Take $A(\lambda):=\frac{1}{\kappa_{\hf}(\lambda)}$; the scaling relation~\eqref{eq:scaling} implies 
that for $\lambda$ in the same right neighbourhood of~$\lcr$,
\begin{equation}\label{eq:1-med}
\nu_{A(\lambda),\lambda}([0,1]) = \nu_{1,\lambda} ([0,\tfrac{1}{A(\lambda)}]) = 
\nu_{1,\lambda} ([0,\kappa_{\hf}(\lambda))]) = \tfrac12,
\end{equation}
while~\eqref{eq:median} implies that $A(\lambda)\to +\infty$ as $\lambda\searrow \lcr$. 

\smallskip

The family $\nu_{A(\lambda),\lambda}$ is a family of probability measures on the compactified 
half-line $[0,\infty]$, hence there exists a convergent subsequence $\nu_{A(\lambda_j),\lambda_j}$ 
for some subsequence $\lambda_j\searrow \lcr$. 

\begin{rem}
Note that this family is no longer $\lo$-monotone, as increasing $A$ and decreasing  $\lambda$ lead 
to $\lo$-inequalities in opposite directions; apparently, there is no easy and direct way to prove 
the convergence of the whole family $\left \{\nu_{A(\lambda),\lambda}\right \}_{\lambda}$.
\end{rem}

We conclude the first step of this proof with the following Lemma.
\begin{lem}\label{l:subseq-diagonal}
Any subsequential limit $\bm=\lim_{j\to \infty} \nu_{A(\lambda_j),\lambda_j}$, with $\lambda_j\searrow \lcr$, 
is $\Pc$-stationary:
\[
\bm=\Pc[\bm].
\]
\end{lem}
\begin{proof}
Recall that for any $\lambda$ the measure $\nu_{A(\lambda),\lambda}$ is $\Phi_{A(\lambda),\lambda}$-stationary
\begin{equation}\label{eq:nu-stat}
\nu_{A(\lambda),\lambda}=\Phi_{A(\lambda),\lambda}[\nu_{A(\lambda),\lambda}].
\end{equation}
Passing in~\eqref{eq:nu-stat} to the limit along the subsequence $\lambda_j$, and recalling that the 
operator on the right hand side is continuous in all the three arguments (Remark~\ref{r:cont-three}), we have 
\[
\bm=\Phi_{\infty,\lcr}[\bm]=\Phi_{\lcr}[\bm].
\]
That is, the measure $\bm$ is $\Phi_{\lcr}$-stationary. 
\end{proof}

Now, the measure $m$ is non-atomic, and hence the measure $\bm$ has no atoms on $(0,+\infty)$ due to 
its $\Phi_{\lcr}$-stationarity. Hence, the relation~\eqref{eq:1-med} gives us $\bm([0,1])=\hf$ when 
passing to the weak limit. This shows the non-triviality of~$\bm$.

\paragraph{Stage 2:} \textit{The constructed non-trivial stationary measure has no atoms.}
As we have just noticed, the measure $\bm$ has no atoms on $(0,+\infty)$, so in order to conclude the
proof of Theorem~\ref{t:stat}, we have to show that the constructed measure $\bm$ is supported on 
$(0,+\infty)$ (and does not charge neither $0$ nor~$\infty$).

Lemma~\ref{l:psi} implies that both $p_0(\bm)$ and $1-p_{\infty}(\bm)$ (\emph{i.e.}~the probabilities 
of  zero and finite lengths respectively) are fixed points of the map~$\theta$, and so belong to 
$\{0,p_{cr},1\}$. The two inequalities
\[
\begin{cases}
p_0(\bm)\le \bm([0,1]) \le 1-p_{\infty}(\bm), \\
\bm([0,1])=\frac{1}{2}>p_{cr},
\end{cases}
\]
imply $1-p_{\infty}=1$ and hence the measure $\bm$ does not charge $+\infty$. 

\smallskip

More involved is to prove that the measure $\bm$ does not charge the point~$0$. The above argument implies 
that there are two possible values for $p_0(\bm)$: either $p_0(\bm)=0$ or $p_0(\bm)=p_{cr}$. We shall assume 
for the rest of this proof 
that~$p_0(\bm)=p_{cr}$, and try to get a contradiction.

Under this standing assumption we can decompose the probability law $\bm$ according to the convex 
combination~\eqref{eq:convex-combination}: 
\[
\bm=(1-p_{cr})\,\bm_0+p_{cr}\,\ddelta_0\,.
\]

\smallskip

We want to find a measure $\mu$ which $\lcr$-zooms out, contradicting the Key Lemma. We will construct 
the measure $\mu$ in two steps:
\begin{enumerate}[i)]
\item We reduce the weight of the atom at zero: that is, we consider a family of measures
\[\bm_p=(1-p)\,\bm_0+p\,\ddelta_0.\]
We prove then that for sufficiently small $p$, this measure satisfies the condition $\bm_p\lo \Pc[\bm_p]$. 
Moreover, we obtain here a stronger strict inequality for their partition functions (see Lemma~\ref{l:x-delta} below):
\beqn{pcr2}{
F_{\bm_p}(x)> F_{\Pc [\bm_p]}(x)\quad\text{for every }x\in [0,+\infty).}
In particular, for any $x\in \R_+$, there exists some $\delta=\delta(x)$ such that
\begin{equation}\label{eq:weak-ineq}
F_{\Pc[\bm_p]}(y)+\delta<F_{\bm_p}(y)\quad \textrm{for every }y\in [0,x].
\end{equation}
If the measure $\bm_p$ was supported on some finite interval, this would immediately mean that $\bm_p$ zooms out; 
as it is not, we have to modify it to make it compactly supported (while not losing the inequality $\bm_p\lo\Pc[\bm_p]$ 
during this process). From the inequality~\eqref{eq:weak-ineq} one can figure out why such a modification can be done.
\item We  modify the measure $\bm_p$ ``near infinity'', so that it becomes compactly supported, without destroying 
the inequality~\eqref{eq:pcr2}. The problem here is that the partition functions $F_{\bm_p}$ and its image are 
approaching each other at infinity, so there is no immediate cut-off-like argument. Such a modification is given 
by Lemma~\ref{l:comp} below, and it provides us with a $\lcr$-zooming out measure and hence with the desired 
contradiction.
\end{enumerate}

Following the road-map above, let us decompose the $\Pc$-image of the measure $\bm$ into several components. 
Namely  we have the following three possibilities for the four lengths $X_1,\dots,X_4$:
\begin{enumerate}[(a)]
\item There is an $\In\Out$-path of zero length. This happens exactly with probability 
$\theta(p_{cr})=p_{cr}$.
\item All of them are non-zero. This happens with probability $(1-p_{cr})^4$ and 
conditionally on this, the law of $R_{\lcr}(X_1,\dots,X_4;\xi)$ is the $\Pc$-image of the conditional measure~$\bm_0$.
\item There is no $\In\Out$-path of zero length but there is at least one collapsed length.
Conditionally on this case, the law of $R_{\lcr}(X_1,\dots,X_4;\xi)$ is given by the law 
of~$\lcr\xi\cdot \min (X'_1,X'_2)$ where the independent random variables $X'_1$, $X'_2$ 
are both distributed with respect to~$\bm_0$.
\end{enumerate}
The last case suggests to define an operator $\Phi'_{\lambda}$ on $\cP_0$ ($\lambda>0$), assigning to 
any measure $\mu_0\in \cP_0$ the law of $\lambda\xi\cdot \min (X_1,X_2)$, where the independent random 
variables $X_1$, $X_2$ are both distributed with respect to~$\mu_0$.

Immediately, we observe:
\begin{lem}\label{l:Pc-prime}
\[
\Pc[\mu_0]\lo \Phi'_{\lcr}[\mu_0]
\]
for any probability measure $\mu_0\in \cP_0$.
\end{lem}
\begin{proof}
Indeed, this means geometrically that we have a longer path if we 
have to pass along two edges instead of one only. More formally, we can use a coupling 
argument: we  draw four i.i.d.~random variables $X_1,\ldots,X_4$ distributed 
with respect to $\mu_0$
and $\xi$ distributed with respect to~$m$ and independent 
of the previous variables, defined on a probability space $(\Omega,\mathbf{P})$. Then there 
is an evident inequality
\[
\lambda\,(\min(X_1(\omega),X_2(\omega))+\min(X_3(\omega),X_4(\omega) ))
\ge \lambda\,\min(X_1(\omega),X_2(\omega))
\]
for almost every $\omega\in (\Omega,\mathbf{P})$ and hence
\[
\lambda\xi(\omega)\,(\min(X_1(\omega),X_2(\omega))+\min(X_3(\omega),X_4(\omega) ))
\ge \lambda\xi(\omega)\,\min(X_1(\omega),X_2(\omega)).
\]
\end{proof}

We concentrate the remaining arguments for the first step in the next Lemma.
\begin{lem}\label{l:x-delta}
Assume that 
\[\bm=p_{cr}\, \ddelta_0 + (1-p_{cr})\, \bm_0\]
is $\Pc$-invariant, and let $\bm_p:=p\,\ddelta_0+(1-p)\,\bm_0$. 
Then, for any sufficiently small $p$, we have~$\bm_p\lo \Pc[\bm_p]$, and moreover there is a strict inequality 
for the partition functions, bounded away from zero on compact intervals: for any $x\in \R_+$ there exists 
$\delta=\delta(x)>0$ such that
\[
F_{\Pc[\bm_p]}(y)+\delta<F_{\bm_p}(y)\quad \text{for every }y\in [0,x].
\]
\end{lem}
\begin{proof}
The analysis of the different possibilities (a-c) shows the equality
\beqn{pcr}{
\Pc[\bm_p]
=\theta(p)\,\ddelta_0
+(1-p)^4\,\Pc[\bm_0]
+(1-\theta(p)-(1-p)^4)\,\Phi'_{\lcr}[\bm_0].
}
Let us rewrite~\eqref{eq:pcr} in the following way:
\beqn{pcr1}{
\Pc[\bm_p]=\theta(p)\,\ddelta_0 + (1-\theta(p))\cdot \left( (1-q(p))\, \Pc'[\bm_0]+q(p)\,\Pc[\bm_0] \right),
}
with $q(p)=\frac{(1-p)^4}{1-\theta(p)}$. The stationarity of $\bm=\bm_{p_{cr}}$ then implies
\[
p_{cr}\,\ddelta_0+ (1-p_{cr})\bm_0= 
\theta(p_{cr})\,\ddelta_0 + (1-\theta(p_{cr}))\cdot \left( (1-q(p_{cr})) \,\Pc'[\bm_0]+q(p_{cr})\,\Pc[\bm_0] \right),
\]
and as $p_{cr}=\theta(p_{cr})$, we have the convex combination
\beqn{convexbm0}{
\bm_0=(1-q(p_{cr})) \,\Pc'[\bm_0]+q(p_{cr})\,\Pc[\bm_0].
}
Due to Lemma~\ref{l:Pc-prime}, we have
\[
\Pc'[\bm_0]\lo \bm_0\lo \Pc[\bm_0].
\]
Now, using the expression \eqref{eq:convexbm0}, we see that $q(p)\to 1>q(p_{cr})$ as $p\to 0$; hence, for any 
sufficiently small $p$ we have 
\beqn{pcr3}{
\bm_0 \lo  (1-q(p))\, \Pc'[\bm_0]+q(p)\,\Pc[\bm_0],
}
and thus, using \eqref{eq:pcr1} and \eqref{eq:pcr3},
\begin{equation}\label{eq:p-theta}
\Pc[\bm_p]  \go \theta(p)\,\ddelta_0 + (1-\theta(p))\, \bm_0 = \bm_{\theta(p)}.
\end{equation}

As $\theta(p)<p$ for $0<p<p_{cr}$, we have $\bm_p\lo \bm_{\theta(p)} \lo \Pc[\bm_p]$ for all sufficiently small $p$. 
Moreover, for any sufficiently small $p$, using~\eqref{eq:p-theta}, we have for any $x\in \R_+$ 
\begin{align*}
\inf_{y\in[0,x]} (F_{\bm_p}(y) - F_{\Pc[\bm_p]}(y)) \ge\,
&
\inf_{y\in [0,x]} (p-\theta(p)) \cdot (F_{\ddelta_0}(y) - F_{\bm_0} (y)) \\
=\,& (p-\theta(p)) \cdot (1- F_{\bm_0}(x)) =:\delta(x)>0.
\end{align*}
This concludes the proof of the lemma.
\end{proof}

\smallskip

Let us fix $p>0$ given by Lemma~\ref{l:x-delta}, and let $\mu:=\bm_p$. The second step of the proof is  to 
modify~$\mu$ so that it becomes compactly supported, without destroying the inequality $\mu\lo\Pc[\mu]$, 
and thus to obtain a $\lcr$-zooming out measure. The following construction allows to do so.

\begin{lem}\label{l:comp}
Let $\mu\neq\ddelta_0$ be a probability measure on $[0,+\infty)$ with no atoms on $(0,+\infty)$ such that 
$\mu\lo \Phi_{\lambda}[\mu]$ for some $\lambda>0$. Let $X$ be a random variable (defined on some probability 
space~$(\Omega,\mathbf{P})$) distributed according to~$\mu$ and denote by $x_1$ its $\tfrac34$-quantile. 
Assume that there exists $\delta>0$ such that
\[F_{\Phi_{\lambda}[\mu]}(y)+\delta<F_{\mu}(y) \quad \text{for all }y\in [0,x_1],\]
where $F_{\Phi_{\lambda}[\mu]}$ and $F_{\mu}$ stand for the partition functions of $\Phi_{\lambda}[\mu]$ 
and $\mu$ respectively.

Take $x_2$ to be the $(1-\delta)$-quantile of $\mu$, and take the measure $\wbm$ to be the law of the 
random variable $\wX$ defined (on the same probability space) as
\begin{equation}\label{eq:tX}
\widetilde{X}(\omega)=
\begin{cases}
X(\omega) & \textrm{if }X(\omega)\le x_2,\\
x_1 & \textrm{otherwise.} 
\end{cases}
\end{equation}
Then the measure $\wbm$ is $\lambda$-zooming out:
$\Phi_{\lambda}[\widetilde{\mu}]\go \widetilde{\mu}$. 
\end{lem}
\begin{proof}
We start with the technical observation that $x_2\ge x_1$: indeed, the inequality 
$F_{\Phi_{\lambda}[\mu]}(x_1)> F_{\mu}(x_1)+\delta$ implies that $F_{\mu}(x_1)<1-\delta$. In particular, 
this implies $\wX\le X$ almost surely, and hence~$\wbm\lo\mu$. It is also not difficult to find an 
explicit expression for the partition function of $\wbm$: 
\[
F_{\wbm}(x) = 
\begin{cases} 
F_{\mu}(x) & \text{if }x\in [0,x_1), \\
F_{\mu}(x)+\delta &\text{if }x\in [x_1,x_2), \\
1&\text{if }x\in [x_2,+\infty).
\end{cases}
\]
Indeed, the coupling~\eqref{eq:tX} says that the measure $\mu$ is transformed in the following way: 
all the mass from $[x_2,+\infty)$ is collapsed into a single atom put at the point~$x_1$. 
Thus, the partition function of $\wbm$ coincides with that of $\mu$ on $[0,x_1)$, differs from it by a 
constant on $[x_1,x_2)$, and is identically equal to~$1$ on~$[x_2,+\infty)$; this easily implies the 
above representation.

Since $\widetilde{\mu}\lo \mu$, applying $\Phi_{\lambda}$ gives the relation 
$\Phi_\lambda[\widetilde{\mu}]\lo \Phi_{\lambda}[\mu]$. Though, the direction of the latter inequality 
does not help, so we want to \emph{quantify} it: 
\begin{equation}\label{eq:tX-diff}
F_{\Phi_{\lambda}[\wbm]}(x) -F_{\Phi_{\lambda}[\mu]}(x) \le \delta\quad\text{for every }x\in [0,+\infty).
\end{equation}

To prove such an estimate, take four independent random vectors $(X_i,\tX_i)$, $i=1,\dots, 4$, 
coupling the measures $\mu$ and $\wbm$ in the way described by~\eqref{eq:tX}, as well as a 
random variable~$\xi$, independent of them all. The random variables $R_{\lambda}(X_1,\dots, X_4; \xi)$ 
and $R_{\lambda}(\tX_1,\dots, \tX_4; \xi)$ are distributed according to $\Phi_{\lambda}[\mu]$ 
and $\Phi_{\lambda}[\wbm]$ respectively; in particular, the probability of the event
\beqn{event}{
\left \{R_{\lambda}(X_1,\dots, X_4; \xi) \neq R_{\lambda}(\tX_1,\dots, \tX_4; \xi)\right \}
}
is an upper bound for the left hand side of the inequality in~\eqref{eq:tX-diff}.
A first remark is that, as for any~$i=1,\dots, 4$ one has 
\[
\mathbf{P}(X_i\neq \tX_i) = \mathbf{P}(X_i > x_2) = \delta,
\]
the probability of the event~\eqref{eq:event} does not exceed 
\[
\mathbf{P}(\exists\, i : \, X_i\neq \tX_i) \le \sum_{i=1}^4 \mathbf{P}(X_i \neq \tX_i) = 4\delta.
\]
This estimate is weaker than the one we want to prove (and with such a weaker estimate, the 
arguments at the end of this proof would not work). So we need an additional idea: for the 
shortest $\In\Out$-paths to have different lengths in graphs with edges of length $X_j$'s and 
$\tX_j$'s respectively, not only there should exist $j$ such that $X_j\neq \tX_j$, but also the 
shortest path in one of the graphs should pass through such edge $j$. And as the inequality
$X_j\neq \tX_j$ implies that both these lengths are relatively large (no less than $x_1$ which 
was chosen as the $\tfrac34$-quantile), this is even less likely. 

Formalizing this idea, we note that the event~\eqref{eq:event} is covered by a union of four events of the kind:
``on a given edge $j$ we have $X_j> \tX_j$, and on the edge $i$ parallel to $j$ the associated $X_{i}$-length 
is not smaller than~$x_1$''. Indeed, if the length of the parallel edge is smaller than $x_1$, 
its presence erases both the larger lengths $X_j$ and $\widetilde X_j$.

The probability of each of these events is equal to
\[
\mathbf{P}(X_j\neq \tX_j)\cdot \mathbf{P}(X_i\ge x_1) =  \delta \cdot \tfrac{1}{4} .
\]
Hence, the total probability does not exceed $4\cdot \tfrac{1}{4}  \delta = \delta$,
as claimed. 

Let us now deduce from~\eqref{eq:tX-diff} the conclusion of the lemma. Within the interval $[0,x_1)$ we have 
the chain of inequalities
\[
F_{\wbm}(y)=F_{\mu}(y)>F_{\Phi_\lambda[\mu]}(y)+\delta\ge F_{\Phi_{\lambda}}[\wbm](y)
\]
and similarly on the interval $[x_1,x_2]$ we have 
\[
F_{\wbm}(y)=F_{\mu}(y)+\delta> F_{\Phi_\lambda[\mu]}(y)+\delta\ge F_{\Phi_{\lambda}}[\wbm](y)
\]
Lastly, for $y>x_2$ we have $F_{\wbm}(y)=1>F_{\Phi_{\lambda}[\wbm]}(y)$.

Joining the three
estimates together, we obtain $F_{\wbm}(y)>F_{\Phi_{\lambda}[\widetilde\mu]}(y)$
for all $
y\in\R_+$, hence 
$\wbm\lo \Phi_{\lambda}[\wbm]$.
By its definition, $\wbm$ is supported on a finite interval $[0,x_2]$, so
we conclude that $\wbm$ is a $\lambda$-zooming out probability measure.
\end{proof}

With the previous lemma we have shown that $\lcr$ is a supercritical value, providing us with the desired 
contradiction, that comes from the assumption $p_0(\bm)\neq 0$. We can then conclude that~$p_0(\bm)=\bm(\{0\})$
is equal to~$0$ and hence $\bm$ is non-atomic. This concludes the proof of the theorem.
\end{proof} 

\subsection{Openness of the supercritical set~$\Lambda$}\label{ss:openness}

We prove now the Key Lemma. The main idea here will be the following: assume 
$\lambda\in\Lambda$ and take the measure $\nu_{A,\lambda}$ which is not trivial. Then 
increasing slightly the cut-off value $A$ and iterating this measure, we obtain a measure $\widetilde\nu$ verifying 
$F_{\Phi_{\lambda}[\widetilde\nu]}(x)<F_{\widetilde\nu}(x)$ for every $x\in \R_+$\,.
It is then natural to use the previous strict inequality to find a $(\lambda-\eps)$-zooming out measure.
The main problem on this path is to control the inequality in a neighbourhood of $0$; this will require 
arguments which are analogous to those of Lemma~\ref{l:x-delta}.

\begin{proof}[Proof of Lemma~\ref{l:open}]
Assume $\lambda\in\Lambda$. Our goal is to show that for sufficiently small 
$\eps>0$ we have $(\lambda-\eps)\in\Lambda$, thus establishing the openness 
of~$\Lambda$. 

The measure $\nu_{A,\lambda}$ is a fixed point for the cut-off 
operator $\Phi_{A,\lambda}$ and hence, denoting by $G_{A,\lambda}$ its distribution function, we have
\[
G_{A,\lambda}=\max(\Phi_\lambda[G_{A,\lambda}], \I_{x\ge A}).
\]
Moreover $G_{A,\lambda}\not\equiv 1$ (as~$\nu_{A,\lambda}\neq \ddelta_0$) implies that $\Phi_\lambda[G_{A,\lambda}]$ 
is everywhere strictly less than one (due to the full support assumption for the measure~$m$). 
Hence, $G_{A,\lambda}(x)=\Phi_\lambda[G_{A,\lambda}](x)$ for $x<A$ and 
$G_{A,\lambda}(x)>\Phi_\lambda[G_{A,\lambda}](x)$ for~$x\ge A$ (see the picture on the 
left in Figure~\ref{f:Co}).

Take now $A'>A$ and define $\tG:=\max(\Phi_\lambda[G_{A,\lambda}], \I_{x\ge A'})$ (corresponding to 
$\widetilde\nu = \Phi_{A',\lambda}[\nu_{A,\lambda}]$). Then we have $\tG\le G_{A,\lambda}$ and moreover 
$\tG(x)<G_{A,\lambda}(x)$ for~$x\in (A,A')$. Again due to the full support assumption for the measure~$m$, 
this implies 
\begin{equation}\label{eq:tnu-strict}
\Phi_\lambda[\tG](x)<\Phi_\lambda[G_{A,\lambda}](x) \le \Phi_{A',\lambda}[G_{A,\lambda}](x)= 
\tG(x)\quad\text{for every }x\in \R_+\,.
\end{equation}
\begin{figure}[ht]
\[
\includegraphics[width=\textwidth]{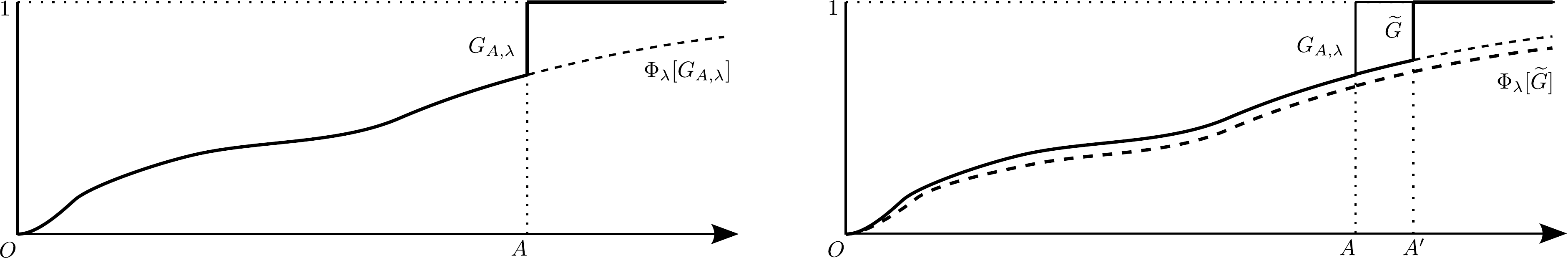}
\]
\caption{The fixed function $G_{A,\lambda}$ for the cut-off operator $\Phi_{A,\lambda}$, 
its modification $\tG$ and their~$\Phi_\lambda$-images} \label{f:Co}
\end{figure}

We would like to decrease $\lambda$, keeping a zooming out measure. For any compact interval $J\subset (0,+\infty)$ 
the strict inequality in~\eqref{eq:tnu-strict} between the continuous functions $\Phi_{\lambda}[\tG]$ and 
$\Phi_{\lambda}[G_{A,\lambda}]$ on~$J$ is preserved under a small perturbation: considering $J=[x_0,A']$, 
$x_0>0$, it is possible to find $\eps_0$ such that for any $\eps<\eps_0$ and $x_0\le x\le A'$,
\[
\Phi_{\lambda-\eps}[\tG](x)<\Phi_{\lambda}[G_{A,\lambda}](x)\le \tG(x).
\]
The inequality $\Phi_{\lambda-\eps}[\tG](x)\le\tG(x)$ also automatically holds for all $x\ge A'$, as 
$\tG(x)=1$ for any such~$x$. The only problem is thus to handle the neighbourhood of $x=0$.

To do so, let us modify the measure $\tnu$ by mixing it with a Dirac mass at zero: consider the 
family of measures $\tnu_p:=(1-p)\,\tnu+p\,\ddelta_0$. We claim that for all sufficiently small $p$, 
we still save the strict inequality between the partition functions of $\tnu_p$ and of its image 
$\Phi_{\lambda}[\tnu_p]$, and that, moreover, it extends to $x=0$. That is, that there exists 
$p_0$ such that for every $p<p_0$ and $x\in [0,+\infty)$,
\begin{equation}\label{eq:tnu-p}
F_{\Phi_{\lambda}[\tnu_p]}(x)<F_{\tnu_p}(x).
\end{equation}

\smallskip

Once~\eqref{eq:tnu-p} is established (and this can be done following the proof of Lemma~\ref{l:x-delta}), 
concluding the proof of the lemma is almost immediate. Indeed, observe that for any $p\in (0,p_{cr})$ 
both partition functions $F_{\tnu_p}$ and $F_{\Phi_{\lambda}[\tnu_p]}$ are continuous on $[0,+\infty)$ 
(in particular right-continuous at $x=0$), moreover the strict inequality 
$F_{\tnu_p}(0)=p>\theta(p)=F_{\Phi_{\lambda}[\tnu_p]}(0)$ between their values is satisfied. 
A strict inequality between continuous functions on a compact set is preserved by a small perturbation, 
hence for any sufficiently small $\eps>0$
\[
F_{\Phi_{\lambda-\eps}[\tnu_p]}(x)\le F_{\tnu_p}(x)\quad \text{for every }x\in [0,A'].
\]
Once again the inequality $F_{\tnu_p}(x)\ge F_{\Phi_{\lambda-\eps}[\tnu_p]}(x)$ is automatically 
satisfied as $x\ge A'$, thus implying the desired
\[\tnu_p\lo \Phi_{\lambda-\eps}[\tnu_p].\]

\smallskip

We focus now on deriving the inequality~\eqref{eq:tnu-p}. The equality can be obtained outside any 
arbitrarily small neighbourhood of $0$, using the arguments of the last paragraph: for any $x_0>0$, 
the strict inequality $F_{\tnu}>F_{\Phi_{\lambda}[\tnu]}$ between the two continuous functions on 
the compact interval $[x_0,A']$ implies the same inequality for $\tnu_p$, for any $p<p_1=p_1(x_0)$, 
while for $x>A'$ there is nothing to check.
 
In the same way as in~\eqref{eq:pcr}, we have
\[
\Phi_{\lambda}[\tnu_p]
=\theta(p)\,\ddelta_0
+(1-p)^4\,\Phi_{\lambda}[\tnu]
+(1-\theta(p)-(1-p)^4)\,\Phi'_{\lambda}[\tnu].
\]
and hence, as $p\to 0$, we have
\begin{align*}
F_{\Phi_{\lambda}[\tnu_p]}(x)=\,&\theta(p) + (1-p)^4 F_{\Phi_{\lambda}[\tnu]}(x) + 
(1-\theta(p)-(1-p)^4)\, F_{\Phi'_{\lambda}[\tnu]}(x)
\\
=\,&F_{\Phi_{\lambda}[\tnu]}(x) + 4p\, ( F_{\Phi'_{\lambda}[\tnu]}(x) - F_{\Phi_{\lambda}[\tnu]}(x) ) + O(p^2),
\end{align*}
where the $O(p^2)$ is uniform in $x\in [0,+\infty)$. Roughly speaking, the main correction term 
$4p\, ( F_{\Phi'_{\lambda}[\tnu]}(x) - F_{\Phi_{\lambda}[\tnu]}(x) )$ corresponds to the situation 
when one of the four edges becomes collapsed, and this happens with probability approximatively $4p$. 
Notice also that the coefficient $( F_{\Phi'_{\lambda}[\tnu]}(x) - F_{\Phi_{\lambda}[\tnu]}(x) )$ tends 
to zero as $x\to 0$. Indeed, this corresponds to the fact that even conditionally on that one of the edges 
is collapsed, we need at least one more edge to have a short $\In\Out$-distance.

At the same time,
\begin{align*}
F_{\tnu_p}(x) =\,& p\cdot 1 + (1-p)\cdot F_{\tnu}(x)\\
=\,&F_{\tnu}(x)+ p \cdot (1-F_{\tnu}(x)) \\
\ge\,& F_{\Phi_{\lambda}[\tnu]}(x) + p \cdot (1-F_{\tnu}(x)).
\end{align*}
Subtracting, we get 
\beqn{diff}{
F_{\tnu_p}(x) - F_{\Phi_{\lambda}[\tnu_p]}(x) \ge p \cdot \left( (1-F_{\tnu}(x))  - 
4 ( F_{\Phi'_{\lambda}[\tnu]}(x) - F_{\Phi_{\lambda}[\tnu]}(x) ) \right) + O(p^2).
}
As we have already mentioned, both $F_{\tnu}(x)$ and  $F_{\Phi_{\lambda}[\tnu]}(x)$ tend to $0$ as $x\to 0$.
Hence, there exists $x_0>0$ such that
\[
F_{\tnu}(x_0)< \tfrac{1}{6},\quad F_{\Phi_{\lambda}[\tnu]}(x_0) < \tfrac{1}{6}.
\]
Substituting it into the right hand side of~\eqref{eq:diff}, we get 
\[
(1-F_{\tnu}(x))  - 4 ( F_{\Phi'_{\lambda}[\tnu]}(x) - F_{\Phi_{\lambda}[\tnu]}(x) ) \ge 1-\tfrac{5}{6} = 
\tfrac{1}{6},
\]
which gives
\[
F_{\tnu_p}(x) - F_{\Phi_{\lambda}[\tnu_p]}(x) > \tfrac{1}{6} p + O(p^2).
\]
As $O(p^2)$ is uniform in $x$ (it corresponds to at least two edges being collapsed), this implies that 
for sufficiently small $p$ we have $F_{\tnu_p}(x) > F_{\Phi_{\lambda}[\tnu_p]}(x)$ for all $x\in [0,x_0]$.
The proof of the Key Lemma is now over.
\end{proof}

\section{Continuity of the critical parameter at $\ddelta_1$}

In this very short section we prove Proposition~\ref{prop:dirac-perturbed}. The idea is to describe the 
supercritical set $\Lambda$ with a quantitative version of Lemma~\ref{l:Lambda}. We start with the lower 
semi-continuity for it is easier to prove.

\begin{lem}[Lower bound]\label{l:lower}
For any $\eps>0$ there exists an open neighbourhood $U$ of $\ddelta_1$ in $\cP_0$ such that 
$\lcr(m)>\tfrac{1}{2}-\eps$ for any $m\in U$, for which $\lcr$ is defined.
\end{lem}

\begin{proof}
Set $\lambda=\frac{1}{2}-\eps$. Fix an arbitrary $q\in(p_{cr},1)$ and consider 
$\mu_-=q\cdot \ddelta_1+(1-q)\cdot \ddelta_\infty$. The measure $\mu_-$ 
(that we have already seen in the second half of the proof of Lemma~\ref{l:Lambda}) corresponds 
to the situation when the $\In\Out$-distance is equal to~$1$ with probability $q$ and is infinite 
otherwise. Glueing together four such independent ``edges'' we get an $\In\Out$-distance that is 
equal to~$2$ with probability~$\theta(q)$ and is infinite otherwise. Finally, the multiplicative 
convolution with $m$ and the rescaling by $\lambda$ give us the measure $\Phi_{\lambda}[\mu_-]$, 
whose distribution function is equal to
\[
\Phi_{\lambda}[F_{\mu_-}] (x) = \theta(q) \cdot F_m (\tfrac{x}{2\lambda}).
\]
As $F_{\mu_-}=q\cdot \mathbf{1}_{[1,\infty]},$ if $m$ is sufficiently close to $\ddelta_1$ so 
that $m((0,\tfrac{1}{2\lambda}))>\frac{q}{\theta(q)}$, we have 
\[
\Phi_{\lambda}[F_{\mu_-}](1)>\theta(q) \cdot \frac{q}{\theta(q)} = q= F_{\mu_-}(1),
\]
and hence
\[
\Phi_{\lambda}[\mu_-] \lo \mu_-.
\]
We then see that $\lambda$ is not supercritical, and hence $\lcr>\frac{1}{2}-\eps$, in the same way as in the 
proof of Lemma~\ref{l:Lambda}.
\end{proof}

The upper bound turns out to be trickier. In the proof of Lemma~\ref{l:lower}, glueing together distances that
are equal to $1$ or $\infty$ leads (before random rescaling) to distances that take only two values, $2$ 
and~$\infty$, and thus are of the same type. Meanwhile, glueing together distances that are equal to $0$ 
or~$1$ (as we would like to do to obtain an upper estimate) leads to distances that take three values, 
$0$, $1$ and $2$, and we would not get a direct way of $\lo$-comparing the measure with its image here. 
However the same ideas work with some modification.

\begin{lem}[Upper bound]\label{l:upper}
For any $\eps>0$ there exists an open neighbourhood $U$ of $\ddelta_1$ in $\cP_0$ such that
$\lcr(m)<\tfrac{1}{2}+\eps$ for any $m\in U$ for which $\lcr$ is defined.
\end{lem}

\begin{proof}
Similarly to the previous lemma, take an arbitrary $\lambda_0=\tfrac{1}{2}+\eps$. We are going to prove 
that there is a neighbourhood of $\ddelta_1$ in which there exists a $\lambda_0$-zooming out measure, 
thus proving that in this neighbourhood $\lcr<\lambda_0$. In fact, we will construct a measure $\mu$
that is $\lambda_0$-zooming out simultaneously for any measure $m$ in such a neighbourhood.

To ensure this, it suffices to verify the following properties:
\begin{enumerate}
\item the measure $\mu$ is supported on $[0,1]$;
\item the Euclidean image $\Phi_{\lambda_0}[\mu]$, associated to $m=\ddelta_{1}$, satisfies 
$\Phi_{\lambda_0}[\mu]\go\mu$;
\item\label{i:delta} there exists $\delta>0$ such that the above inequality can be strengthened 
to the following one: at any point $x\in (0,1]$, one has 
\begin{equation}\label{eq:p-zoom-ineq}
\Phi_{\lambda_0}[\mu]((0,x]) +\delta \le \mu((0,x)).
\end{equation}
\end{enumerate}
Indeed, the last two properties ensure that $\lambda_0$-zooming out of the measure $\mu$ will be preserved 
by a small perturbation of the measure~$m$. 

Choose now an arbitrary $\lambda$ in the interval $(1/2,\lambda_0)$; we will first consider the ``smaller'' 
random glueing image, corresponding to $\lambda$ instead of $\lambda_0$. We will be looking for the 
desired measure of the form $\mu_p=p\cdot\mu'+(1- p) \cdot\ddelta_1$, with a fixed probability measure $\mu'$, 
supported on $[0,1)$, and with very small value $p$ (tending to~$0$). Note that the $\Phi_{\lambda}$-image 
(with no random factor) of such a measure is composed of: 
\begin{itemize}
\item the atomic measure $\ddelta_{2\lambda}$, coming with the weight $(1-p)^4$ (glueing all the four 
intervals of length~$1$);
\item the image $f_* \mu'$ of the measure $\mu'$ under the map $f:x\mapsto \lambda(1+x)$, coming with 
the weight $4p-O(p^2)$; this comes from glueing three intervals of length~$1$ and one of (random) length chosen 
with respect to~$\mu'$;
\item a remaining part of total mass $O(p^2)$, corresponding to at least two intervals being picked with 
respect to~$\mu'$.
\end{itemize}

Now let the sequence $\{x_n\}$ be defined by 
\[
x_0=0, \quad x_{n+1}=f(x_n).
\]
This sequence tends to $\lim_{n\to \infty} x_n=\frac{\lambda}{1-\lambda}>1$, so there exists $N\in\N$ 
such that $x_N<1\le x_{N+1}$. Take 
\[
\mu':=\frac{1}{Z} \sum_{j=0}^N 4^j\, \ddelta_{x_j},
\]
where $Z:=\sum_{j=0}^N 4^j$ is the normalization constant.

Then, the image $\Phi_{\lambda}[\mu']$ (corresponding to $m=\ddelta_1$) consists up to $O(p^2)$, 
of $(1-4p)\,\ddelta_{2\lambda}$ and of 
\[
4 p\, f_* \mu' = 4p \sum_{j=0}^{N} 4\, \ddelta_{f(x_j)} = 
p\, \mu' - p( \tfrac{1}{Z} \ddelta_0 - \tfrac {1}{Z} 4^{N+1} \ddelta_{x_{N+1}})
\]
Hence, as $x_{N+1}\ge 1$, for any $x\in [0,1]$ we have 
$\Phi_{\lambda}[\mu']([0,x]) +\frac{p}{Z} \le \mu'([0,x])+O(p^2)$, with $O(p^2)$ being uniform in~$x$. 
In particular, we can choose and fix sufficiently small $p>0$ such that for any $x\in [0,1]$ 
\[
\Phi_{\lambda}[\mu']([0,x]) +\frac{p}{2Z} \le \mu'([0,x]).
\]
Finally, replacing $\lambda$ by $\lambda_0>\lambda$, we ensure that all the atoms are pushed even 
further away, thus obtaining the property~\eqref{eq:p-zoom-ineq}, with $\delta=\tfrac{p}{2Z}$. 
This concludes the proof.
\end{proof}

\section{Convergence in law and uniqueness of the stationary measure}\label{s:convergence}

The purpose of this section is to show that, under suitable assumptions, the 
operator~$\Pc$ defines a contraction on $\cP$. This is the content of 
Theorem~\ref{t:conv}.

Fix a stationary measure $\bm$ from the statement of Theorem~\ref{t:conv}, let $\Fbm$ be 
its distribution function. For any $\alpha\in(0,1)$ denote by $\kal$ the $\alpha$-quantile of 
the measure $\bm$: let $\kal:=F_{\bm}^{-1}(\alpha)$. Note that these quantiles are well-defined.
Indeed, $\bm$ is a stationary measure and hence~$\Fbm=\Pc[\Fbm]$; as 
the measure~$m$ is absolutely continuous with continuous positive density,~$\Fbm$ is 
a~$C^1$ function with positive derivative on~$\R_+$ (we operate a multiplicative 
convolution with $m$ as the last calculation needed for obtaining its $\Phi_\lambda$-image).

Throughout this section, we prefer to work with the space of distribution functions (cf.~Remark~\ref{r:distribution}).

\begin{dfn}
Denote by $\mF$ the space of monotonic non-decreasing right-continuous functions on the extended half-line 
$[0,+\infty]$, satisfying $0\le F\le 1$ and $F(+\infty)=1$. We will denote by $\mF_0$ the subspace of 
distribution functions of probability measures in $\cP_0$ (and so verifying $F(0)=0$ and 
$\lim_{x\rightarrow\infty}F(x)=1$).
\end{dfn}

\subsection{The class $\mC_{\alpha,\delta}$ and contraction}\label{ssc:classCad}

\begin{dfn}
A distribution function $F\in\mF$ is of  \emph{class $(\alpha,\delta)$} (where $\delta>0$ 
and $\alpha\in(0,1/2)$) if 
\begin{itemize}
\item $F(x)\le \Fbm(x)$ on $[\kal,\kone]$,
\item $F(x)\le \Fbm(x)+\delta$ on $[0,\kal)\cup (\kone,+\infty]$.
\end{itemize}
We denote the set of such functions by $\mC_{\alpha,\delta}$.
\end{dfn}

We also consider the family of rescaling operators on $\mF$ induced by the operators $\Upsilon_c$: 
\begin{dfn}
For any $r\in\R$ let 
\[
T_r: \mF \to \mF, \quad T_{r}[F](x)=F(e^{-r}x).
\]
In other words, if $F=F_\mu$ is the distribution function of a measure $\mu$, then $T_r[F]$ is 
the distribution function of the rescaled measure $\Upsilon_{c}[\mu]$, with $\log c=r$.
\end{dfn}

The main result in this part is the following key proposition which gives a good family of rescaling operators 
``improving'' the class of a given distribution function:
\begin{prop}\label{l:shift}
There exist constants $L>0$ and $\alpha\in(0,1/2)$ such that for every $\delta\in (0,\alpha]$ the operator 
$T_{L\delta}\Pc$ sends $\mC_{\alpha,\delta}$ in $\mC_{\alpha,\delta/2}$.
\end{prop}

\begin{rem}\label{r:class}
Although not every distribution function $F\in\mF$ belongs to some class 
$\mC_{\alpha,\delta}$, every distribution function in the smaller space $\mF_0$ can be 
rescaled by some $T_r$ to belong to a certain class. Moreover, for any such distribution 
function $F$ and for \emph{any} class $\mC_{\alpha,\delta}$ there exists a rescaling 
$T_{r}[F]$ that belongs to this class. It is also clear that there is no reason to consider 
$\delta>\alpha$: the class $\mC_{\alpha,\delta}$ coincides with 
$\mC_{\alpha,\alpha}$.
\end{rem}

Using Proposition~\ref{l:shift}, we fix corresponding constants $\alpha$ and $L$ (note that the 
value of $L$ does not depend on~$\delta$). The following remark gives us a way of 
obtaining upper bounds for the asymptotic behaviour of the iterates $\Pc^n[F]$, comparing 
them to a rescaling of the distribution function of a stationary measure.
\begin{prop}\label{c:iterate}
Let $F\in \mF$ be a function of class $(\alpha,\delta)$, then for any $n$ the function 
$T_{2L\delta} \Phi^n_{\lcr}[F]$ belongs to $\mC_{\alpha,\delta/2^n}$.
\end{prop}
\begin{proof}
We have 
\begin{align*}
T_{L\delta}\Pc[F] & \in \mC_{\alpha,\delta/2},\\
T_{L\delta/2}\Pc T_{L\delta}\Pc[F] & \in \mC_{\alpha,\delta/4},\\
&\textrm{etc.,}\\
T_{L\delta/2^n}\Pc \dots T_{L\delta}\Pc[F] & \in \mC_{\alpha,\delta/2^n}.
\end{align*}
Since the operator $\Pc$ commutes with the rescaling ones we have 
\[
T_{\frac{L\delta}{2^n}+ \dots + \frac{L\delta}{2} + L\delta} \, \Pc^n[F]\in 
\mC_{\alpha,\delta/2^n}
\]
and the result easily follows, for $T_{\frac{L\delta}{2^n}+ \dots + \frac{L\delta}{2} + 
L\delta} \, \Pc^n[F]\ge T_{2L\delta} \, \Pc^n[F]$.
\end{proof}
Described in a different way, the previous proposition reads
\begin{prop}\label{c:upper}
In the assumptions of Proposition~\ref{c:iterate}, for any $x\in \R_+$ and $n\ge 1$ we have
\[\Pc^n[F](x)\le T_{-2L\delta}[\Fbm](x)+\delta/2^n.\]
\end{prop}

The rest of this section is devoted to the proof of Proposition~\ref{l:shift}.

\subsubsection{The function $\Fbmd$}
 
A useful remark is that any class $\mC_{\alpha,\delta}$ has a greatest element $\Fbmd$ 
(corresponding to the smallest probability measure with respect to the stochastic domination) defined by
\[\Fbmd(x)=\begin{cases}
\Fbm(x)+\delta &\text{on }[0,\kappa_{(\alpha-\delta)})\cup [\kone,\kappa_{(1-\delta)}), \\
\alpha & \text{on }[\kappa_{(\alpha-\delta)},\kal), \\
\Fbm(x) & \text{on }[\kal,\kone), \\
1 & \text{on }[\kappa_{(1-\delta)},+\infty).
\end{cases}
\]
That is, a function $F\in \mF$ verifies $F\le\Fbmd$ if and only if 
it belongs to $\mC_{\alpha,\delta}$.
\begin{figure}[ht]
\[
\includegraphics[width=.7\textwidth]{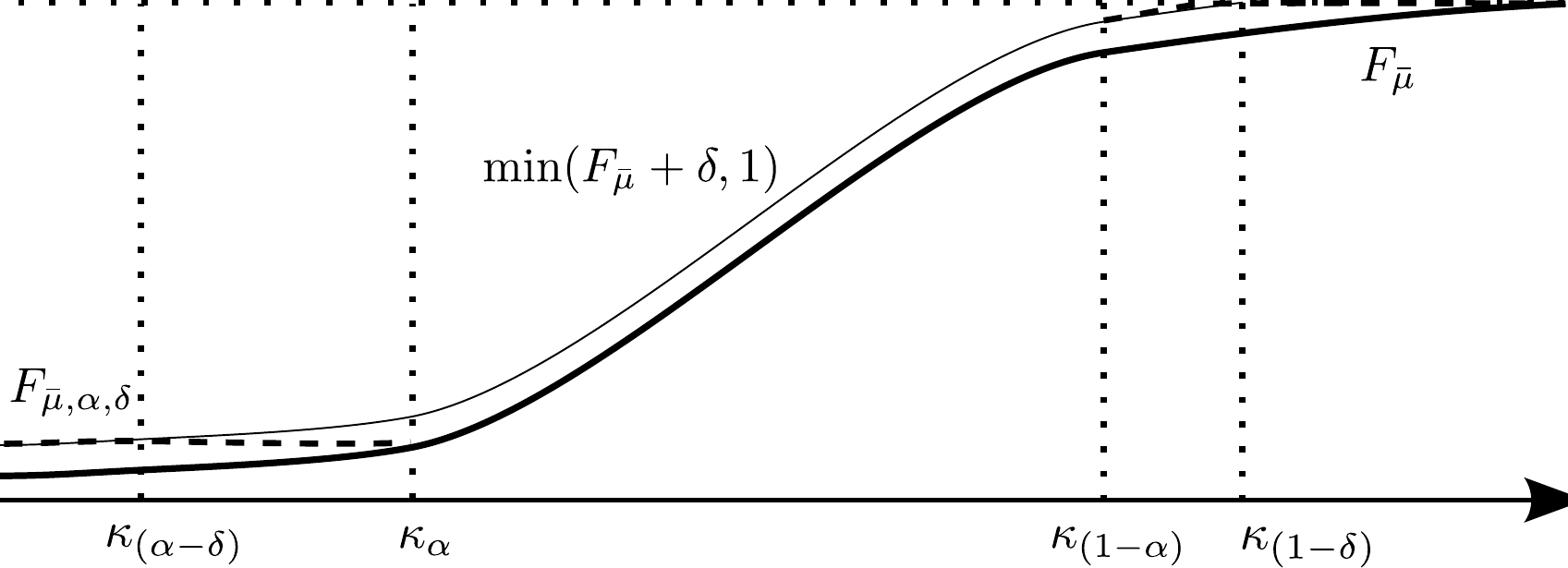}
\]
\caption{The distribution functions $\Fbm$ and $\Fbmd$.}
\end{figure}

\smallskip

By monotonicity of the operator $\Pc$, it is sufficient to prove Proposition~\ref{l:shift} for the 
function~$\Fbmd$: we have to show the existence of a constant $L$ such that the image $T_{L\delta}\Pc[\Fbmd]$ 
is of class $\mC_{\alpha,\delta/2}$. This requires an accurate description of how far the image $\Pc[\Fbmd]$ 
is from $\Fbm$. We will first give a \emph{global} upper bound and then
study the neighbourhoods of $0$ and $+\infty$.

\subsubsection{Coupling $\Fbm$ and $\Fbmd$}

Once more, it will be more convenient to work with a
good coupling between the laws of $\Fbm$ and $\Fbmd$.

So for the rest of this section, we fix a probability space $(\Omega,\mathbf{P})$. Given a random 
variable $X$ whose law is $\bm$, we define a new 
random variable $X^{\alpha,\delta}$ by
\begin{equation}\label{eq:couplingdelta}
X^{\alpha,\delta}(\omega)=
\begin{cases}
0 & \textrm{if }X(\omega)\in[\kappa_{(\alpha-\delta)},\kal ),\\
\kone & \textrm{if }X(\omega)\in[\kappa_{(1-\delta)},+\infty), \\
X(\omega) & \textrm{otherwise}.
\end{cases}
\end{equation}
It is easy to remark that the distribution function of $X^{\alpha,\delta}$ is exactly the function 
$\Fbmd$: the $\delta$-jumps of $\Fbmd$ at $0$ and $\kone$ correspond to the first two lines in 
\eqref{eq:couplingdelta}. Moreover the event $\{X^{\alpha,\delta}\neq X\}$ is given by the first 
two possibilities in~\eqref{eq:couplingdelta} and its probability equals $2\delta$. 

Let us apply the operator $\Pc$ to the distribution function $\Fbmd$. To do so, draw four 
independent random variables $X_1,\ldots,X_4$ distributed with respect to
the law $\bm$ and define 
four further random variables $X_1^{\alpha,\delta},\ldots,X_4^{\alpha,\delta}$ according 
to~\eqref{eq:couplingdelta}. Take also
a random variable $\xi$ of law~$m$ and independent of the eight previous ones. We set
\begin{align*}
Y=&\,R_{\lcr}(X_1,\ldots,X_4;\xi), \\
Y^{\alpha,\delta}=&\,R_{\lcr}(X_1^{\alpha,\delta},\ldots,X_4^{\alpha,\delta};\xi).
\end{align*}
Since the law $\bm$ is stationary, the law of~$Y$ also coincides with~$\bm$. On the other 
hand the probability of $Y^{\alpha,\delta}\neq Y$ is upper bounded by the sum of the probabilities 
of four events $\{X_j\neq X_j^{\alpha,\delta}\}$ and does not exceed $4\cdot 2\delta=8\delta$.

Then for any $\alpha\in(0,1/2)$ and $\delta\in (0,\alpha]$ we have the inequality
\begin{equation}\label{eq:4delta}
\Phi_{\lcr}[\Fbmd]\le \min(\Fbm+8\delta,1).
\end{equation}

\smallskip

However this bound is quite rough. For our purposes we need to be a little 
more careful: we know that the function $\Pc[\Fbmd]$ takes the value 
$\theta(\delta)=O(\delta^2)$ at zero, which is pretty better than~$8\delta$. Henceforth, we will give a bound 
on the values $s_0$ and $t_0$
for which the distance between $\Fbm$ and~$\Pc[\Fbmd]$ is less than $\delta/2$ 
outside the interval $[s_0,t_0]$.

\begin{lem}\label{l:delta2}
There exist two points $s_0$, $t_0\in \R_+$ such that for every $\alpha\in (0,1/80)$ the following estimates 
hold. For any $\delta\in(0,\alpha]$, $s\in [0,s_0)$ and $t\in (t_0,+\infty)$ we have
\begin{align}\label{eq:delta2a}
\mathbf{P}(Y^{\alpha,\delta}\le s<Y)&\le \delta/2,\\
\label{eq:delta2b}
\mathbf{P}(Y^{\alpha,\delta}< t<Y)&\le \delta/2.
\end{align}
\end{lem}

\begin{proof}
We only prove the estimate \eqref{eq:delta2a}, since the estimate \eqref{eq:delta2b} can be obtained in a 
similar manner.

\smallskip

The condition $Y^{\alpha,\delta}<Y$ implies that at least one of the $X_{i}^{\alpha,\delta}$'s is different 
from $X_i$, and that this edge is not short-cut by its parallel of length~$X_j$. By symmetry, we have
\begin{equation}\label{eq:YsY}
\mathbf{P}(Y^{\alpha,\delta}\le s<Y)\le 4\cdot \mathbf{P}(Y^{\alpha,\delta}\le s
<Y,\, X^{\alpha,\delta}_4<\min(X_3,X_4)).
\end{equation}
The event $\{X^{\alpha,\delta}_4<\min(X_3,X_4)\}$ is covered by the two events
\[E_1:=\{X_4\in [\kappa_{(\alpha-\delta)},\kal)\}
\text{ and }
E_2:=\{X_4\ge \kappa_{(1-\delta)},X_3>\kappa_{(1-\alpha)}\},\]
whose probabilities are respectively $\mathbf{P}(E_1)=\delta$ and $\mathbf{P}(E_2)=\alpha\delta$.
Then we can write from~\eqref{eq:YsY}:
\begin{equation}\label{eq:contribution}
\mathbf{P}(Y^{\alpha,\delta}\le s<Y)\le
 4\cdot \mathbf{P}(Y^{\alpha,\delta}\le s 
<Y\mid E_1)\,\mathbf{P}(E_1)+4\cdot \mathbf{P}(Y^{\alpha,\delta}\le s <Y\mid E_2)\,\mathbf{P}(E_2)
\end{equation}
The rightmost summand in~\eqref{eq:contribution} (the one involving $E_2$) does not 
exceed~$4\cdot \frac{1}{80} \delta=\frac{\delta}{20}$. Roughly speaking, in our coupling the probability 
that one of the $X_i$'s is different from the corresponding~$X_i^{\alpha,\delta}$ is approximately $8\delta$, 
but if $X_i$ is large, most probably this edge will be short-cut by a parallel one, and this leaves us 
only with at most $\frac{\delta}{20}$ of contribution.

Unfortunately the same estimate for the summand in~\eqref{eq:contribution} involving $E_1$ does 
not work: it  leads to the 
contribution of $4\delta$, that is much larger than the total $\frac{\delta}{2}$ that we should obtain.
There is a good reason for that: if $X_i^{\alpha,\delta}=0$ no parallel short-cutting edge can change 
the difference. To handle the contribution of this set, we recall that we are asking for the total distance 
$Y^{\alpha,\delta}$ not only to be different from $Y$, but also to be less than~$s\in [0,s_0]$; choosing 
$s_0$ sufficiently small will allow us to impose an additional restriction, reducing the probability as desired. 

Namely, since the measures $m$ and $\bm$ have no atoms, we can choose $s_0$ such that 
\[
\mathbf{P}(\lcr\xi\min(X_1,X_2)\le s_0)\le  \frac{1}{20}.
\]
We claim that this choice for $s_0$ is fine enough.

\smallskip

Indeed, when $E_1$ holds, we have
\[
Y^{\alpha,\delta}= \lcr\xi\min(X^{\alpha,\delta}_1,X^{\alpha,\delta}_2).
\]
Hence, conditionally on $E_1$, the probability of $\{Y^{\alpha,\delta}\le s\}$ does not exceed the sum 
of two probabilities, the one of 
$\{\lcr\xi\min(X_1,X_2) \le s\}$ and the one of the event
\[B=\left \{\lcr\xi\min(X^{\alpha,\delta}_1,X^{\alpha,\delta}_2)\neq \lcr\xi\min(X_1,X_2)\,\right \}\]
Using the same argument as before, we observe that the probability of $B$ is at most $4\delta$. We 
then obtain the estimate
\begin{align*}
\mathbf{P}(Y^{\alpha,\delta}\le s <Y\mid E_1)\le \,&
\,\mathbf{P}(B)+\mathbf{P}(\lcr\xi\min(X_1,X_2) \le s) \\
\le \,& 4\delta+\frac{1}{20} \le \frac{1}{10}.
\end{align*}
Plugging in this last estimate into~\eqref{eq:contribution}, we can now conclude:
\begin{align*}
\mathbf{P}(Y^{\alpha,\delta}\le s<Y)\le\, &
4\cdot \mathbf{P}(Y^{\alpha,\delta}\le s <Y,\, X^{\alpha,\delta}_4<\min(X_3,X_4)) \\
\le \,& 4\cdot \mathbf{P}(Y^{\alpha,\delta}\le s 
<Y\mid E_1)\,\mathbf{P}(E_1)+4\cdot \mathbf{P}(Y^{\alpha,\delta}\le s <Y\mid E_2)\,\mathbf{P}(E_2) \\
\le \,&4\delta\cdot \frac{1}{10}+\frac{\delta}{20}< \frac{\delta}{2}.
\end{align*}
\end{proof}

The previous lemma implies that for any $z\in [0,s_0]\cup[t_0,+\infty)$, the functions $\Pc[\Fbmd]$ and 
$\Fbm$ differ at most by $\delta/2$ at the point $z$. The key fact is that the points $s_0$ and $t_0$
do not depend on $\alpha$. In particular we can choose $\alpha$ sufficiently small in such a way that 
$[s_0,t_0]\subset[\kappa_{\alpha},\kone]$. We use this result to end the proof of Proposition~\ref{l:shift},
which follows directly from the next easy lemma.

\begin{lem}
For any $\alpha<1/80$ such that $[s_0,t_0]\subset[\kappa_{\alpha},\kone]$ 
there exists a constant $L_0$ with the following property.  
Take $\delta\in(0,\alpha]$, and assume that the distribution function $F$ satisfies
\[
F(x)\le \begin{cases}
\Fbm(x)+\delta/2 & \text{on } [0,s_0)\cup (t_0,+\infty),\\
\Fbm(x)+8\delta & \text{on } [s_0,t_0].
\end{cases}
\]
Then $T_{8L_0\delta}[F]\in\mC_{\alpha,\delta/2}$.
\end{lem}
\begin{proof}
We have already observed that the function $\Fbm$ is $C^1$ regular with positive derivative on
$(0,+\infty)$. Hence, the function $g:=\log \Fbm^{-1}$ is well defined and locally Lipschitz on $(0,1)$.
Let $L_0$ be the Lipschitz constant of $g$ on the interval $[\kappa_{\alpha/2}, \kappa_{1-\alpha}]$. We shall show that
this choice is fine enough.

Namely, for $x\notin [\kappa_{\alpha},\kappa_{1-\alpha}]\supset [s_0,t_0]$ we have 
\[
T_{8L_0\delta}[F](x)\le F(x)\le \Fbm(x)+\delta/2,
\]
so we have to check only the inequality $T_{8L_0\delta}[F](x)\le \Fbm(x)$ for all $x\in [\kappa_{\alpha},\kappa_{1-\alpha}]$. 

To do so, for any such $x$ consider the point $x'=e^{-8L_0\delta}x$. Assume first that $x'\ge \kappa_{\alpha/2}$. Then, the values $y:=F_{\bm}(x)$ and $y':=F_{\bm}(x')$ both belong to $[\alpha/2,1-\alpha]$, and the Lipschitz condition for the function $g$ on this interval gives us 
\[
8L_0\delta= \log x- \log x'=\log F_{\bm}^{-1}(y)- F_{\bm}^{-1}(y')\le L_0 (y-y') = L_0 (F_{\bm}(x)-F_{\bm}(x')).
\]
Hence, $F_{\bm}(x')\le F_{\bm}(x)-8\delta$ and 
\[
T_{8L_0\delta}[F](x)=F(x')\le F_{\bm}(x')+8\delta \le F_{\bm}(x)-8\delta+8\delta=F_{\bm}(x).
\]
Finally, if $x'\le \kappa_{\alpha/2}$, we have 
\[
T_{8L_0\delta}[F](x)=F(x') \le F(\kappa_{\alpha/2})\le F_{\bm}(\kappa_{\alpha/2})+\delta/2 \le \alpha =F_{\bm}(\kappa_{\alpha}) \le F_{\bm}(x).
\]
\end{proof}

\subsection{Asymptotic upper bounds and convergence}

With Proposition~\ref{c:upper} at our disposal, we shall now detect the good rescaled measure to which $\mu$ 
(or~$F$) converges under the iterations of $\Pc$. We first make a guess and then prove that it is correct.

\begin{dfn}
The rescaling $T_r[\Fbm]$ \emph{asymptotically upper bounds $F$} if 
for every $\eps>0$ it is possible to find $n_0\in \N$ such that for any $n\ge n_0$ and $x\in \R_+$ we have
\[ 
T_r[\Fbm](x) +\eps \ge \Pc^n[F](x).
\]
In the same way, the rescaling $T_r[\Fbm]$ \emph{asymptotically lower bounds $F$} if 
for every $\eps>0$ it is possible to find $n_0\in \N$ such that for any $n\ge n_0$ and $x\in \R_+$ we have
\[
T_r[\Fbm](x) -\eps \le \Pc^n[F](x).
\]
\end{dfn}
\begin{dfn}
Given any function $F\in \mF_0$ we define the sets
\begin{align*}
R_+=R_+(F)&=\{r\in \R\mid T_r[F_{\bm}] \text{ asymptotically upper bounds } F \},
\\
R_-=R_-(F)&=\{r\in \R\mid T_r[F_{\bm}] \text{ asymptotically lower bounds } F \}.
\end{align*}
\end{dfn}

We have already observed in Remark~\ref{r:class} that for any $F\in \mF_0$ 
there exists a rescaling of $F$ that belongs to the class $(\alpha,\alpha)$. So
Proposition~\ref{c:upper} implies that~$R_+(F)$ is nonempty. 
We can repeat the arguments in Section~\ref{ssc:classCad} with the class of functions $F\in \mF$ 
verifying the reversed conditions:
\begin{enumerate}
\item $F(x)\ge \Fbm(x)$ on $[\kal,\kone]$,
\item $F(x)\ge \Fbm(x)-\delta$ on $[0,\kal)\cup (\kone,+\infty]$.
\end{enumerate}
Then we see that the 
set~$R_-(F)$ is nonempty as well. Finally, for any $r_+\in R_+$ and $r_-\in R_-$ it is easy to see 
that~$r_+\le r_-$, hence the set $R_+$ is nonempty and right bounded. 

\smallskip

This allows to consider the value 
\[
r_0=r_0(F):=\sup R_+(F).
\]
The proof of Theorem~\ref{t:conv} will be concluded once we prove the following
\begin{prop}\label{p:limit}
For any $F\in \mF_0$ the iterates $\Pc^n[F]$ converge to $T_{r_0(F)}[\Fbm]$ as $n$ tends to $\infty$.
\end{prop}
\begin{proof}
The upper bound is easy. Indeed for any $r<r_0$, the definition of $r_0$ itself gives
\[
\limsup_{n\to\infty} \Pc^n[F](x) \le T_{r}[\Fbm](x),
\]
and passing to the limit as $r\uparrow r_0$, we obtain the desired
\begin{equation}\label{eq:asympt-upper}
\limsup_{n\to\infty} \Pc^n[F](x) \le T_{r_0}[\Fbm](x).
\end{equation}
From~\eqref{eq:asympt-upper} we automatically have 
\begin{equation}\label{eq:asympt}
\liminf_{n\to\infty} \Pc^n[F](x) \le T_{r_0}[\Fbm](x).
\end{equation}
For the desired equality $\lim_{n\to\infty} \Pc^n[F](x) = T_{r_0}[\Fbm](x)$ we have to show that 
it is impossible to have a \emph{strict} inequality in~\eqref{eq:asympt} at any point~$x\in \R_+$. 
For this we shall need the following:
\begin{lem}\label{l:lower2}
For any $\eps>0$ there exist $\delta$, $r>0$ such that if $F\in\mF$ is of class 
$(\alpha,\delta)$ and for some~$x\in \R_+$ one has $F(x)<F_{\bm}(x)-\eps$, then 
$T_{r}[F_{\bm}]$ asymptotically upper bounds~$F$.
\end{lem}

Indeed, if a strict inequality takes place in~\eqref{eq:asympt} at some point $x$, 
set 
\[
\eps:=\frac{1}{2} \left(T_{r_0}[\Fbm](x)-\liminf_{n\to\infty}\Pc^n[F](x)\right).
\]  
Then we can apply Lemma~\ref{l:lower2} and find the corresponding $r$ and $\delta$; the 
inequality~\eqref{eq:asympt-upper} then implies that for every sufficiently large $n$ the functions 
$T_{-r_0}\Pc^n[F]$ are of class $(\alpha,\delta)$. Taking $n$ such that 
\[
\Pc^n[F](x)<T_{r_0}[\Fbm](x)- \eps,
\]
we have from the conclusion of Lemma~\ref{l:lower2} that $T_{r}[F_{\bm}]$ asymptotically upper
bounds $T_{-r_0}\Pc^n[F]$, and hence $r_0+r$ belongs to $R_+(F)$.

\smallskip

The above arguments are driven by the following idea: if the sequence of iterations $\Pc^n[F](x)$ 
was asymptotically upper bounded by $T_{r_0}[\Fbm](x)$, but with a strict inequality somewhere, the 
glueing iteration would then ``disperse'' this inequality everywhere, allowing to reduce~$r_0$ further. 
With Lemma~\ref{l:lower2} we formalize this rough statement.

\begin{proof}[Proof of Lemma~\ref{l:lower2}]
Let us first reduce (at the cost of replacing $\eps$ by $\eps/2$) the possible set of values of 
$x$ that we have to consider from $\R_+=(0,\infty)$ to some interval bounded away from $0$ and 
from~$\infty$. Indeed if $x>\kappa_{1-{\eps}/{2}}$ we have
\[
F(\kappa_{1-{\eps}/{2}})<F(x)<F_{\bm}(x)-\eps<1-\eps=
(1-{\eps}/{2})-{\eps}/{2}= F_{\bm}(\kappa_{1-{\eps}/{2}})-{\eps}/{2},
\]
so the assumptions of the lemma are also satisfied for $\eps':={\eps}/{2}$ and 
$x':=\kappa_{(1-\eps')}$.

On the other hand, $F(x)<F_{\bm}(x)-\eps$ implies that $F_{\bm}(x)\ge \eps$ and hence 
$x\ge \kappa_{\eps}$. It is therefore sufficient to prove the lemma under the additional 
assumption that $x$ stays in the interval $[\kappa_{\eps},\kappa_{(1-\eps)}]$.

\medskip

Before pursuing the proof in full generality, we establish a weaker statement. Namely, 
let us assume that the map $F$ satisfies $F(x)\le \Fbm(x)$ for all $x$ (in other words, that $\delta=0$)
and let us show the existence of a positive $r=r(F)$ for which $T_r[\Fbm]$ asymptotically upper bounds~$F$. 
This statement is much weaker than the desired one, not only because it corresponds to 
$\delta=0$, but also because we establish an existence of $r>0$ for each $F$ instead of finding a uniform 
$r$ for all functions~$F$ satisfying our assumptions; nevertheless, it clarifies the arguments that we 
shall use later (that would be otherwise too technical).

\smallskip

As $\Pc[\Fbm]=\Fbm$ and the measure $m$ is of full support on $\R_+$, after applying $\Pc$ once we get 
(from the convolution operation) the strict inequality 
\begin{equation}\label{eq:strict}
\Pc[F](x)<\Fbm(x)\quad \text{for any }x\in \R_+\,.
\end{equation}
Due to the strict inequality~\eqref{eq:strict} it is possible to find, for any compact 
interval $J\subset \R_+$ (and in particular for $[\kappa_{\eps},\kappa_{(1-\eps)}]$), a sufficiently 
small $r'>0$ such that
\[
T_{-r'}\Pc[F](x) < \Fbm (x)\quad\textrm{for every }x\in J.
\]
Hence, for such an $r'$
the function $T_{-r'}\Pc[F]$ belongs to some class $(\alpha,\delta_1)$, 
with $\delta_1=\delta_1(r')<r'$ and $\alpha$ depending on $J$ (for example $\alpha<\eps$ if 
$J=[\kappa_{\eps},\kappa_{(1-\eps)}]$).
 
Using Proposition~\ref{c:upper}, this implies that $T_{-r'}[F]$ is 
asymptotically upper bounded by $T_{-2L\delta_1}[\Fbm]$ and hence $F$ is 
asymptotically upper bounded by
$T_{r'-2L\delta_1}[\Fbm]$.

\smallskip

The final part of 
the argument, that will conclude the consideration of this particular case, is that it is possible 
to make a choice for $\delta_1$ such that
\begin{equation}\label{eq:small}
\delta_1(r')=o(r') \quad \text{as } r'\to 0. 
\end{equation}
Indeed, once this estimate is established, we deduce that for a sufficiently small $r'$ the value 
$r:=r'-2L\delta_1(r')$ is strictly positive, which is exactly what we required.

To establish~\eqref{eq:small}, let us pass to the logarithmic scale, considering the coordinate 
$y=\log x$. After this change of variable, we have that for the function 
$\varphi(y):=\Fbm(e^y)$, the derivative tends to $0$ at $\pm\infty$. Indeed, $\bm$ is 
equal to the multiplicative convolution of $m$ with the law of $R_{\lcr}(X_1,\dots,X_4;1)$, 
where the $X_j$'s are i.i.d.~with law~$\bm$. Hence the distribution function of 
$\log_* \bm$ on $\R$ is a convolution of the distribution function of the measure $\log_* m$ 
with the measure $\nu=\law(\log R_{\lcr}(X_1,\dots,X_4;1))$. By the assumptions on $m$, the distribution function of 
the measure $\log_* m$ is a function of class $C^1$ with  derivative that tends to zero at 
$\pm\infty$, hence the same holds for its convolution with an arbitrary measure~$\nu$.

\smallskip

Now, the slope of the function $\varphi$ tends to 0 at $\pm\infty$. Hence for an arbitrarily small $s>0$ 
we can find a compact interval $[y_1,y_2]\subset \R$, 
such that for any point $y\notin [y_1,y_2]$ and any shift parameter $0<r'<1$, the increment 
$\varphi(y+r')-\varphi(y)$ does not exceed $sr'$. In other words, we have 
\[
\varphi(y+r')<\varphi(y)+sr'.
\]
Coming back to the initial coordinates 
we get that outside the interval $J:=[e^{y_1},e^{y_2}]$ one has 
\[
T_{-r'}\Pc[F]\le T_{-r'}[\Fbm]<\Fbm+sr'
\]
for all sufficiently small~$r'>0$. Since for the interval $J$, for all sufficiently small $r'>0$ one has the strict 
inequality $T_{-r'}\Pc[F]|_J<\Fbm|_J$, we eventually conclude that $T_{-r'}\Pc[F]\in \mC_{\alpha,sr'}$ for 
some $\alpha$. As $s>0$ was arbitrary, we obtain the desired $\delta_1(r')=o(r')$ as $r'\to 0$.

\medskip

Let us now modify the above arguments to return to the full general assumptions of the lemma. 
We set $s:=1/6L$; due to the above arguments, there exists an interval $[y_1,y_2]\subset\R$  such that for any 
$y\notin [y_1,y_2]$ and any $r'\in (0,1)$ we have $\varphi(y+r')<\varphi(y)+s r'$. 
Furthermore, without loss of generality we can assume $y_1<\log\kal$ and $y_2>\log\kone$.

Observe that if for a function~$F\in \mC_{\alpha,\delta}$ it is possible to find a point 
$x_0\in [\kappa_{\eps},\kappa_{(1-\eps)}]$ such that $F(x_0) < \Fbm(x_0) -\eps$, then $F$ is bounded from 
above (see~Figure~\ref{f:lemma-9}) by the function 
\[
F_{\alpha,\delta,\eps,x_0,\bm}(x):=
\min (F_{\bm,\alpha,\delta}, (\Fbm(x_0)-\eps)\cdot \mathbf{1}_{x\le x_0} + \mathbf{1}_{x>x_0}).
\]
\begin{figure}[ht]
\[
\includegraphics[width=.7\textwidth]{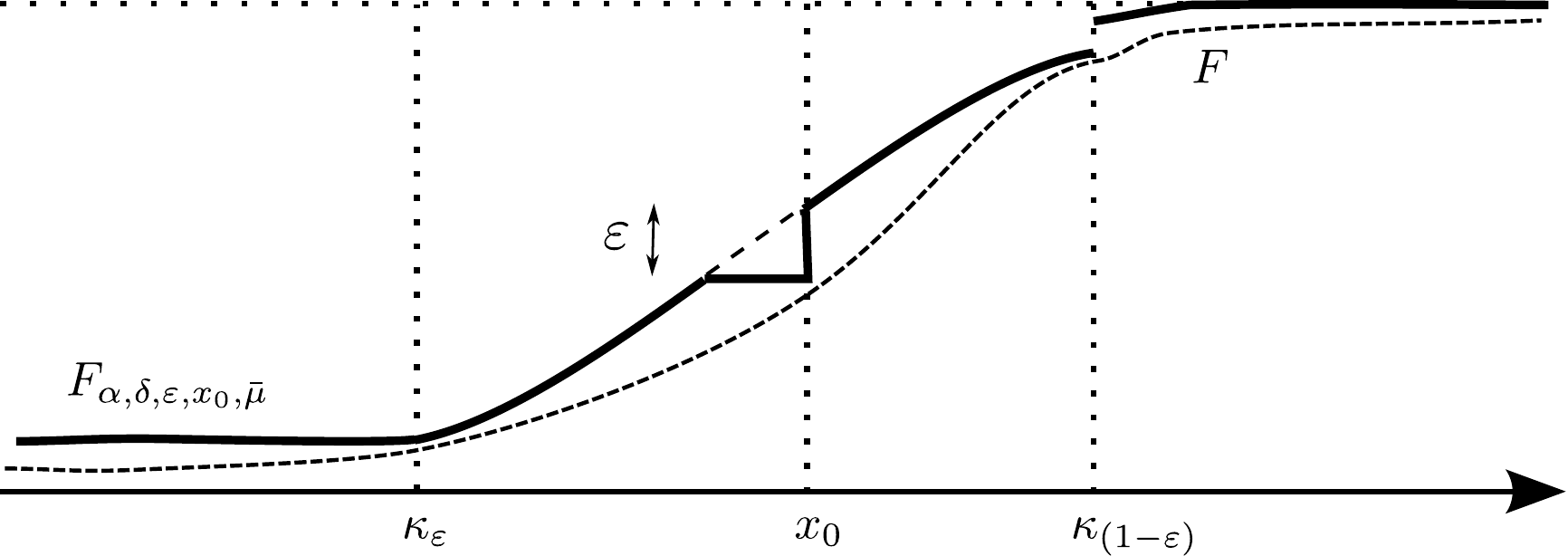}
\]
\caption{The distribution functions $F$ and $F_{\alpha,\delta,\eps,x_0,\bm}$.}\label{f:lemma-9}
\end{figure}

Note that this family of functions defines a family of measures on $[0,+\infty]$, depending 
continuously on $\delta$ and $x_0$. Hence, the family of functions 
$\Pc[F_{\alpha,\delta,\eps,x_0,\bm}]$ depends $C^0$-continuously on $\delta$ and $x_0$ 
(due to the multiplicative convolution with~$m$). Then the function
\[
H(x,x_0,\delta,r')=T_{-r'}\Pc[F_{\alpha,\delta,\eps,x_0,\bm}](x)-\Fbm(x)
\]
considered on $[e^{y_1},e^{y_2}]\times [\kappa_{\eps},\kappa_{1-\eps}]\times [0,\alpha] \times [0,1]$
is also continuous. It is strictly negative when $\delta$ and $r'$ are $0$ and so, due to the compactness 
of $[e^{y_1},e^{y_2}]\times [\kappa_{\eps},\kappa_{1-\eps}]$, the same holds for sufficiently 
small $\delta$ and $r'$. Thus, there exist $\delta_2>0$ and $r'_2>0$ such that for every $\delta<\delta_2$
and $r'<r'_2$ one has
\[
H(x,x_0,\delta,r')<0 \, \text{ for any } x\in [e^{y_1},e^{y_2}], \, x_0\in [\kappa_{\eps},\kappa_{1-\eps}].
\]
For any such $\delta$ and  $r'$, using the inequality between $F_{\alpha,\delta,\eps,x_0,\bm}$ 
and $F$, we have 
\begin{equation}\label{eq:2s}
T_{-r'}\Pc[F](x)<\Fbm(x) \quad \textrm{for any } x\in [e^{y_1},e^{y_2}].
\end{equation}
On the other hand $F$ is of class $(\alpha,\delta)$ and so the inequality~\eqref{eq:4delta} holds:
\[
\Pc[F]\le \Fbm+8\delta.
\]
Take any $r'<\min(1,r'_2)$ and  $\delta<\min(\delta_2,r'/48L)$. Then for such 
$r'$ and $\delta$ and any function $F$ of class $\mC_{\alpha,\delta}$ we have the inequality
\[
T_{-r'}\Pc[F]\le \Fbm \quad\textrm{on } [e^{y_1},e^{y_2}]\supset [\kal,\kone]
\]
due to~\eqref{eq:2s}, while on the complement $[0,e^{y_1}]\cup [e^{y_2},\infty]$,
\[
T_{-r'}\Pc[F]\le T_{-r'}[\Fbm + 8\delta] \le \Fbm + \frac{r'}{6L} + 8 \,\frac{1}{48L} r' \le \Fbm + \frac{r'}{3L},
\]
giving that $T_{-r'}\Pc[F]$ is a function of class $(\alpha,r'/3L)$. Using Proposition~\ref{c:upper}, 
we obtain that $T_{-r'}[F]$ is asymptotically upper bounded by $T_{-2r'/3} [\Fbm]$. 
Eventually we apply $T_{r'}$: then $F$ is asymptotically upper bounded by $T_{r'}T_{-2r'/3}[\Fbm]$ 
and this proves the lemma with $r={r'}/{3}$.
\end{proof}
We have established Lemma~\ref{l:lower2}, and this concludes the proof of Proposition~\ref{p:limit}
due to the above arguments, and thus ends the proof of Theorem~\ref{t:conv}.
\end{proof}

\section{Studying the random metric on the limit object}\label{s:metric}

\subsection{First properties}
In order to study the constructed random metric space~$\mX$, it is crucial to have tail estimates of the stationary
measures. Actually, it is not conceptually difficult to have good asymptotics of tail decreasing, but it is
however rather technical and their discussion here would be perhaps confusing. The reader will find the 
statements and their proofs in \S\ref{app:tails}.

\begin{proof}[Proof of Theorem~\ref{t:metric}] ~

We shall keep the notations introduced in \S\ref{ss:random-metric}, as well as the spirit of the interpretation by means of the invariant RTP.

\paragraph{Subcritical case: $\gamma_{\BRW}+\log\lcr<0$. }
We want to prove that the random metric space $\mX$ is homeomorphic to $\Gamma_\infty$. 
The two inclusions of $V_\infty$ into the random metric space $\mX$ and into $\Gamma_\infty$ define a random 
``identity'' map $\iota=\iota(\mX):V_\infty\subset \Gamma_\infty\to V_\infty\subset\mX$. 
The subspace $V_\infty$ is dense in both spaces and $\Gamma_\infty$ is a compact metric space: if we prove that
the map $\iota$ is continuous, it will then follow that $\iota$ extends to a homeomorphism between $\Gamma_\infty$ 
and $\mX$.

It is enough to look at the images of level $n$ copies in $\Gamma$, for $n$ sufficiently large: we want to show 
that their sizes in $\mX$ are (uniformly) small. More precisely, for any $t\in \T$ we denote by $\mX_t$ the closure 
of $\iota(V_\infty\cap \Gamma_t)$ as a subspace of $\mX$, and for any $n\in \N$ we set
\[
D_n:=\max_{t: \, \| t \|=n} \diam(\mX_t).
\]
We claim that the limit of $D_n$ as $n$ goes to $\infty$ is $0$ almost surely. 

Recall, from the notations introduced in \S\ref{ss:random-metric}, that the random variables $Y_t$'s
defined by \eqref{eq:Y-rec}
\[Y_t=X_t\cdot \prod_{j=1}^{\|n\|} (\lcr\xi_{p^j(t)}),\]
give the $\In\Out$-distance inside $\mX_t$. We write
\[
D'_n:=\max_{t: \, \| t \|=n} Y_t.
\]
Cutting the geodesic paths in $\mX_t$ at dyadic points, we have the bound $D_n\le 2\sum_{j=n}^{\infty} D_j'$. 
Hence, it is enough to show that the series $\sum_{n\in\N}D'_n$ is summable: in turn, we shall show that, 
almost surely, the sequence $D'_n$ tends to zero exponentially fast.

Using the definition \eqref{eq:Y-rec}, we have the following upper bound for the maximum $D_n'$:
\begin{equation}\label{eq:Dnprime}
D_n'=\max_{t: \, \| t \|=n} Y_t =\max_{t: \, \| t \|=n} X_t\cdot \prod_{j=1}^{n} (\lcr \xi_{p^j(t)}) \le  
\max_{t: \, \| t \|=n} X_t\cdot \max_{t: \, \| t \|=n} \prod_{j=1}^{n} (\lcr \xi_{p^j(t)}).
\end{equation}
We keep the notation of the introduction and define
\[
M_n=\max_{t:\,\|t\|=n}\sum_{j=1}^n\log\xi_{p^j(t)}.
\]
Recall (see~\eqref{eq:BRWconstant}) that the Hammersley-Kingman-Biggins Theorem gives a constant 
$\gamma_{\BRW}$ such that $M_n=\gamma_{\BRW}\,n+o(n)$ almost surely. Finally for the stationary measure
$\bm$, as  we will see in \S\ref{app:tails}, we have exponential tail bounds. In particular, Proposition \ref{p:ln} 
there claims:
 \[L_n= \max_{t:\,\|t\|=n} \log X_t=o(n).\]
Taking the logarithm on both 
sides of \eqref{eq:Dnprime}, we thus obtain
\beqn{estdiam}{
\log D_n'\le (\gamma_{\BRW}+\log\lcr)\,n+o(n).
}
As $\gamma_{\BRW}+\log\lcr<0$ by hypothesis, we have that $D'_n$ decreases exponentially fast, as wanted.

\smallskip

As a by-product of the proof just given, we can obtain a rough bound on the Hausdorff dimension of the random 
metric space $\mX$ in this subcritical case. Indeed, the arguments that we have just explained show that 
the extended random identity map $\iota:\Gamma_{\infty}\to\mX$ is not only continuous, but \emph{H\"older} 
continuous of any exponent 
\begin{equation}\label{eq:holderexp}
\alpha<\frac{|\gamma_{\BRW}+\log\lcr|}{\log 2}.
\end{equation}
Hence, since the Hausdorff dimension of $\Gamma_\infty$ is $2$, the Haudorff dimension of 
$\iota(\Gamma_{\infty})=\mX$ is at most~$2/\alpha$ for any $\alpha$ satisfying~\eqref{eq:holderexp}. 
Choosing $\alpha$ closer and closer to $\frac{|\gamma_{\BRW}+\log\lcr|}{\log 2}$, we have 
$\frac{2\log 2}{|\gamma_{\BRW}+\log\lcr|}$ as an upper bound of the Hausdorff dimension of $\mX$.

Let us prove~\eqref{eq:holderexp}. First, from~\eqref{eq:estdiam}, for any $\beta<|\gamma_{\BRW}+\log\lcr|$ we have:
\[
\log D_n'\le -\beta n \quad \text{for all sufficiently large $n$.}
\]
We have then an upper bound for all such~$n$: 
\[D_n \le 2\sum_{j=n}^{\infty} D_j' \le 2 \sum_{j=n}^{\infty} \exp(-\beta j) = c \cdot \exp(-\beta n),
\]
where $c=\tfrac{2}{1-\exp(-\beta)}$ is a constant.

Taking into account that the diameter of any small copy $\Gamma_t\subset \Gamma_{\infty}$ is
\[\diam(\Gamma_t)=2^{-\|t\|},\]
we have, for all such $n$:
\begin{equation}\label{eq:max}
\max_{\|t\|=n}\frac{\diam\left (\iota(\Gamma_t)\right )}{\diam(\Gamma_t)^\alpha}=\frac{D_n}{2^{-\alpha n}}
\le \frac{c \cdot \exp(-\beta n)}{2^{-\alpha n}}=c\cdot \exp((\alpha \log 2- \beta) n).
\end{equation}
For any $\alpha$ satisfying~\eqref{eq:holderexp} we can take 
$\beta\in (\alpha\cdot \log 2, |\gamma_{\BRW}+\log\lcr|)$, ensuring that the maximum in~\eqref{eq:max}
stays bounded uniformly in~$n$. This guarantees that the inclusion map $\iota$ is H\"older continuous
of exponent~$\alpha$.

\paragraph{Supercritical case: $\gamma_{\BRW}+\log\lcr>0$. } We want to prove that the diameter of $\mX$
is almost surely unbounded.
Clearly the diameter $\diam(\mX)$ admits any of $\diam(\mX_t)$, $t\in\T$, as a lower bound. In particular,
looking at the distance between dyadic points of first depth in $\mX_t$, we have the bound
\[
\diam(\mX)\ge \min(Y_{t_1},\dots,Y_{t_4})\quad\text{for every }t\in\T.
\]
Using the definition \eqref{eq:Y-rec}, we have (with the abuse of notation $p^0=id$)
\begin{equation}\label{eq:BIO}
\min(Y_{t_1},\dots,Y_{t_4}) = \left( \prod_{j=0}^{\|t\|} \lcr \xi_{p^j(t)} \right) \cdot   \min(X_{t_1},\dots,X_{t_4}).
\end{equation}
For any $t\in \T$ the two marked factors in \eqref{eq:BIO} are independent.
In particular we can choose the (random) vertex $t=t^n_{\max}$, among the vertices of depth $n$, that maximizes the first factor:
\[
\prod_{j=0}^{n} \lcr \xi_{p^j(t_{\max}^n)} = \max_{t: \, \|t\|=n} \prod_{j=0}^{n} \lcr \xi_{p^j(t)}.
\]
The logarithm of this factor represents the maximum of the $\BRW$ with increments of law $\mN(\log\lcr,\sigma^2)$, 
as in the previous case, but has now a positive drift, and so this factor diverges (at an exponential rate). 
Since the second factor
\[
\min(X_{t_1},\dots,X_{t_4})
\]
has a law that does not depend on the vertex $t$ (and on $n$, in particular), this divergence guarantees that 
the diameter of $\mX$ is almost surely infinite, as wanted.
\end{proof}

We focus now our attention on the supercritical case, when $\gamma_{\BRW}+\log \lcr>0$. In this situation, as 
the Theorem~\ref{t:metric} states, the diameter of $\mX$ is almost surely infinite; moreover, the space $\mX$ 
must be of unbounded diameter at any scale, in the sense that any $\mX_t$ is of infinite diameter almost surely. 
However, Theorem~\ref{t:geodesic} claims that the space $\mX$ is connected. Before passing to its proof, we 
describe how it should look like with the help of a toy example.

\subsection{A toy model: percolation with replacement}
Let us analyse the following percolation-type problem: we shall discuss its relation to our model just afterwards.
\begin{ex}\label{ex:percolation}
Considering the the figure eight-graph $(\Gamma,\In,\Out)$ we construct a recursive family of \emph{random} 
graphs $\{(G_n,\alpha,\omega)\}_{n\in\N}$. 

We fix a parameter $p\in[0,1]$. The starting graph $G_0$ has two vertices $\{\alpha,\omega\}$ and one 
edge connecting them. Suppose that $(G_n,\alpha,\omega)$ has been constructed, then we obtain 
$(G_{n+1},\alpha,\omega)$ by running independently, for any edge of $G_n$ the following (random) operation:
\begin{itemize}
\item either, with probability $p$, replacing it by $\Gamma$, attaching the vertices $\In$ and $\Out$ 
at the endpoints of the edge,
\item or removing it, with probability $(1-p)$.
\end{itemize}
The marked vertices $\alpha$ and $\omega$ remain the same.
The construction is stopped if at some random step $N$ the vertices $\alpha$ and $\omega$ are no more 
connected by a path in $G_N$. In case it never happens, we set $N=\infty$.  

For any $n\in\N$ set $q_n=q_n(p)=\P(N>n)$ and $q_\infty=\P(N=\infty)$. Using the percolation function 
$\theta$ introduced in Definition~\ref{dfn:psi}, we have the recursive relation
\[
\begin{cases}
q_0=1, \\
q_{n+1}=\psi_p(q_n),
\end{cases}
\]
with $\psi_p(q)=p\cdot \theta(q)$.
Indeed, for $n=1,\ldots,N$, we must have $G_1=\Gamma$ (contributing with the factor $q_1=p$), and then 
we shall look for a $\alpha\omega$-path in $G_{n+1}$, knowing that any of the four edges in the figure 
eight is open with probability $q_n$.

In particular, $q_\infty$ must verify
\[q_\infty=p\cdot\theta(q_\infty).\]
Let us find $q_{\infty}$ (almost) explicitly. Namely, note that $\psi_p:[0,1]\to[0,p]$ is a homeomorphism; 
hence the point 
\[q_{\infty}=\lim_{n\to\infty}q_n=\lim_{n\to\infty}\psi_p^n(1)\]
is the largest among the points fixed by $\psi_p$. A fixed point of $\psi_p$ is either $0$, or a point 
$x\in(0,1]$ such that
\[\frac{\psi_p(x)}{x}=\frac{1}{p}.\]
The function $x\mapsto \dfrac{\psi_p(x)}{x}$ takes its maximum value $\tfrac{32}{27}$ on $(0,1]$ at 
the point $x_*=\tfrac{2}{3}$. Hence, $q_{\infty}(p)=0$ if 
$p<p_*=\tfrac{27}{32}$ whereas $q_{\infty}(p)$
is the unique solution in $[\tfrac{2}{3},1]$ of the equation $\tfrac{\theta(x)}{x}=\tfrac{1}{p}$ if $p\ge p_*$.
Note that $x_*$ is larger than $p_{cr}$ (see Figure~\ref{f:theta}).
\begin{figure}[ht]
\[
\includegraphics[width=.3\textwidth]{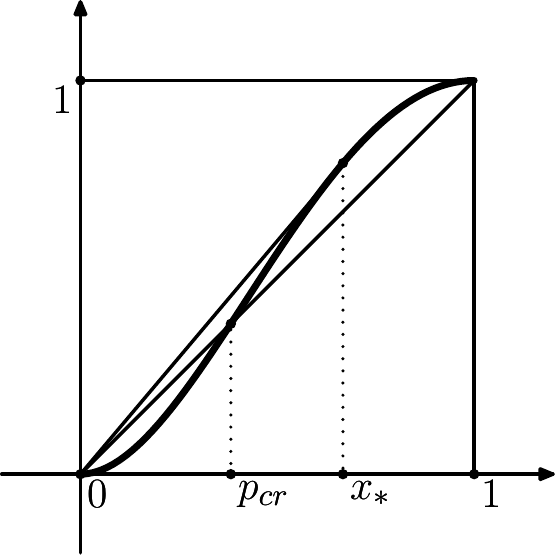}
\]
\caption{The function $\theta$, the percolation critical probability $p_{cr}=\varphi^{-2}$ and the limit probability 
$x_*=\tfrac23$ for the replacement 
procedure corresponding to the critical probability $p_*=\tfrac{27}{32}$.}\label{f:theta}
\end{figure}
\end{ex}

The relation between the previous problem and our setting is rather straightforward: when considering the law 
$m$ of $\xi_t$ to be $m=p\,\ddelta_{1/2}+ (1-p)\,\ddelta_{\infty}$, namely
\[
\xi_t=\begin{cases}
1/2 & \textrm{with probability } p,\\
\infty & \textrm{with probability } 1-p
\end{cases}
\]
(here the factor $\infty$ means that the passage by this edge becomes forbidden), the probability $q_\infty$ 
corresponds to the probability that there exists a path between $\In$ and $\Out$. We modify slightly this 
model in the following example.

\begin{ex}\label{ex:geodesics}
We consider now the measure $m= p_*\,\ddelta_{1/2} + (1-p_*)\, \ddelta_{10^3}$, in other words, let the 
random factors be
\[
\xi_t=\begin{cases}
1/2 & \textrm{with probability } p_*, \\
10^3 & \textrm{with probability }1-p_*.
\end{cases}
\]
Then taking $\lambda=1$, from Example~\ref{ex:percolation} we realize that with probability $x_*=\frac{2}{3}>0$ 
there exists an $\In\Out$-path along which the factor~$10^3$ never appears. Moreover as~$x_*>p_{cr}$, there
exists almost surely a path along which only a finite number of multiplications by~$10^3$ intervenes. Indeed,
the probability that there is a path along which there has been no multiplication by~$10^3$ after the first~$n$
steps is at least~$\theta^n(x_*)$, and this probability tends to~$1$ as~$n$ goes to~$\infty$.

When $\lambda=1$, the sequence of $\In\Out$-distances $d_n$ after $n$ steps is monotone (the replacement
never decreases the length of an edge), and its limit is finite almost surely  due to the above arguments.
We can then apply Proposition~\ref{p:metric} and define the associated random metric space (even though the
measure $m$ does not satisfy the assumptions of Theorem~\ref{t:stat}). In this case the random metric
space will be of infinite diameter almost surely: not only the corresponding 
$\gamma_{\BRW}+\log\lambda=\gamma_{\BRW}$ is positive, but even for any infinite branch in $\T$, the associated 
infinite product diverges to $\infty$: we have 
\[\mathbf{E}\log \xi=\log\left(1000^{1-p_*}\cdot (1/2)^{p_*}\right) >0.\]
Summarizing, the random metric space $\mX$ admits the following (almost sure) description. First, \emph{there is} 
an $\In\Out$-geodesic, that is, an isometric embedding of some interval in~$\mX$, such that its endpoints 
are mapped to $\In$ and $\Out$ respectively. Second, there are four ``local geodesics'' joining the vertices 
from $V_1$ inside the four level $1$ copies of $\mX$ (with two of these local geodesics that are the halves
of the $\In\Out$-geodesic). Next, there are sixteen ``local geodesics'' joining the vertices from $V_2$ 
inside the sixteen level $2$ copies of $\mX$ (eight of those local geodesics are the halves of the four 
local geodesics of the previous step). And so on, while we observe that the \emph{maximal length} of
the local geodesic in level~$n$ copies increases (exponentially) with~$n$.
\end{ex}

\begin{figure}[ht]
\[\includegraphics[width=.6\textwidth]{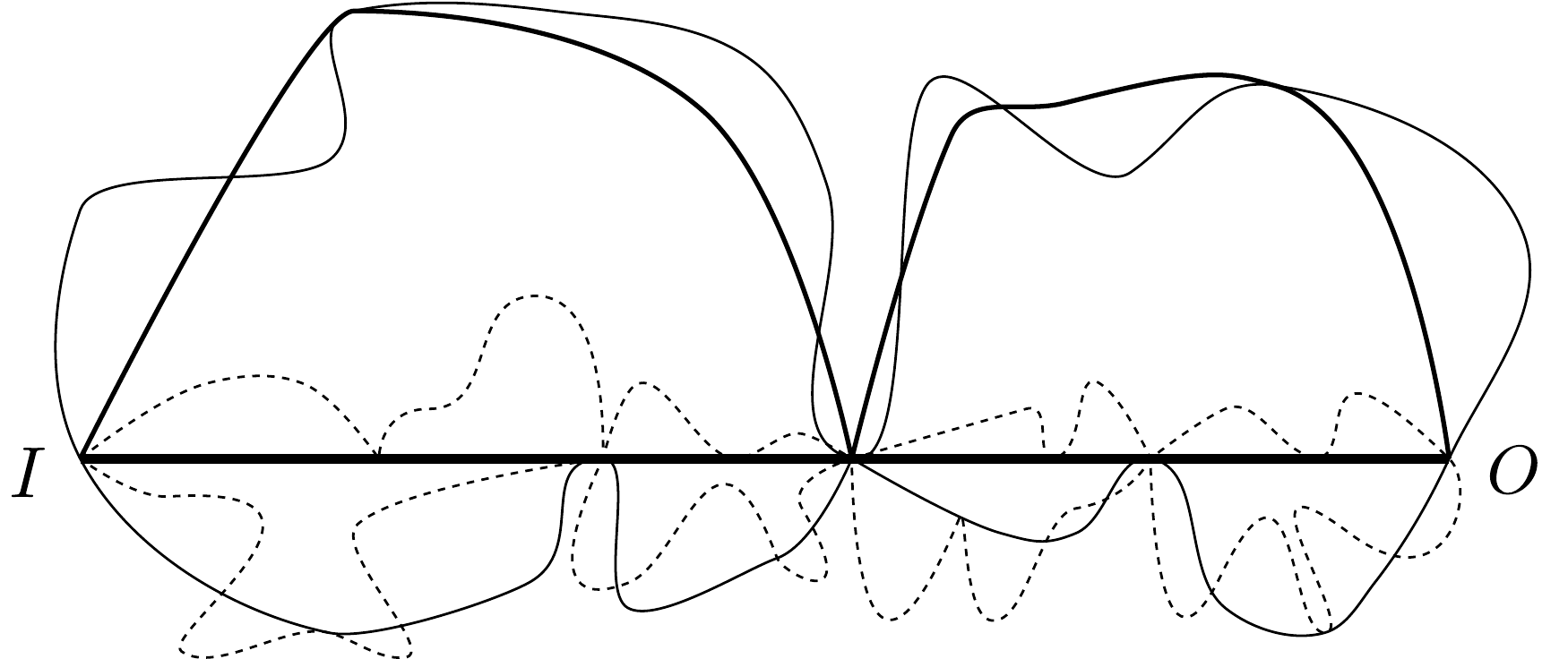}\]
\caption{The $\In\Out$-geodesic and some local geodesics.}
\end{figure}

\subsection{Existence of geodesic paths}\label{ssc:existence-geodesic}
Theorem~\ref{t:geodesic} claims that the random metric space $\mX$ is almost surely connected. The main 
reason for this to hold, as it will emerge during the proof, is that the situation analysed in 
Example~\ref{ex:geodesics} persists in the more complicated framework. In particular, even though we are
in the supercritical case of Theorem~\ref{t:metric}, the space will result to be the closure of the union 
of (longer and longer) local geodesic paths.

Note that in order to establish such a description it suffices to show that almost surely there exists
a geodesic path between $\In$ and $\Out$. Indeed, the stationarity of the random metric space implies 
the existence of local geodesics corresponding to edges of arbitrary depth. The union of all such geodesics
is a path connected space, containing $V_{\infty}$. This implies that the space $\mX$ is 
almost surely connected: the closure of a path connected space is connected.

\medskip

The distances between the points in $V_{\infty}$ -- that is, the RTP 
$\left \{\left (X_t,\xi_t\right )\right \}_{t\in\T}$ -- can be generated by the 
following ``top-to-bottom'' Markovian procedure.

We start by sampling an $\In\Out$-distance $X_{\bullet}$ with respect to the measure $\bm$. Then 
we sample the factor $\xi_{\bullet}$ and the four glued first-level distances $X_1,\ldots,X_4$
conditionally on the given value of $X_{\bullet}$. Then, conditionally on the four lengths 
$X_1,\ldots,X_4$, we sample the corresponding four factors $\xi_1,\ldots,\xi_4$ and sixteen
distances $X_{11},\ldots,X_{44}$, and so on.

\smallskip

Now, let $\mathcal D$ be a rooted dyadic infinite tree and consider the two-branching process
\[\left \{\left (\tX_s,\xi_s \right)\right \}_{s\in\mathcal{D}}=
\left \{\left (X_{t_s},\xi_{t_s}\right )\right \}_{s\in\mathcal D}\]
that selects at each level the distances corresponding to the two halves of the shortest path: we
start with $\tX_\bullet=X_{t_\bullet}=X_{\bullet}$, then
\[\tX_1=X_{t_1}=\min\left (X_1,X_2\right ),\quad \tX_2=X_{t_2}=\min\left (X_3,X_4\right ),\]
where $t_1$ and $t_2$ correspond to the two (random) vertices providing this respective minima; then we define
\[
\tX_{11}=X_{t_{11}}=\min\left (X_{t_11},X_{t_12}\right ),\quad \tX_{12}=X_{t_{12}}=\min\left (X_{t_13},X_{t_14}\right ),
\]
and so on. Remark that the collection of vertices $\left \{t_s\right \}_{s\in\mathcal D}$ form a (random) dyadic 
subtree of the quaternary tree~$\T$. 

By construction the values 
\[\tilde{Y}_s:=\prod_{j=1}^n\xi_{p^j(t_s)}\cdot \tX_s\]
are the lengths of the edges forming the $\In\Out$-geodesic after a given number of glueings. In particular,
\[\sum_{\|s\|=n}\tilde{Y}_s=\tilde{Y}_{\bullet}=d(\In,\Out)=X_{\bullet}.\]

\smallskip

Take the (random) interval $[0,X_{\bullet}]$ and for any $n$ divide it with $2^n-1$ points into $2^n$ 
intervals of consecutive lengths $\left \{\tilde{Y}_s\right \}_{\|s\|=n}$: they correspond to the
$\In\Out$-path passing through the vertices of $V_n$. The next $(n+1)^{th}$ partition is then a 
subpartition of the $n^{th}$ one (as $V_n\subset V_{n+1}$).

The existence of an $\In\Out$-geodesic is now equivalent to the fact that the limit of such
partitions is dense in $[0,X_{\bullet}]$.
If it is not the case, then there exists an infinite branch $\{s_n\}_{n\in\N}$ in the dyadic
tree~$\mathcal{D}$, along which $\tilde Y_s$ does not tend to zero. In particular, as for any 
$s$ we have $\tilde Y_s=\tilde Y_{s1}+\tilde Y_{s2}$, along this branch
\beqn{branch}{
\lim_{n\rightarrow\infty}
\frac{\max\left (\tilde Y_{s_n1},\tilde Y_{s_n2}\right )}{\min\left (\tilde Y_{s_n1},\tilde Y_{s_n2}\right )}=
\lim_{n\rightarrow\infty}\frac{\max\left ( \tX_{s_n1}, \tX_{s_n2}\right )}{\min\left ( \tX_{s_n1}, \tX_{s_n2}\right )}=
\infty,}
and starting from some vertex $s_{n_0}$, $\tilde{Y}_{s_{n+1}}$ is the largest among $\tilde Y_{s_n1}$ and 
$\tilde Y_{s_n2}$, and the quotient of the largest of the two by the smallest is at least $2$.

This finally leads us to the following \emph{non-branching} process, indexed by a (random) infinite branch 
in $\mathcal{D}$ and hence in $\T$: from a vertex $s$ and the associated $\tX_s$, we select $\tX_{s1}$ or $\tX_{s2}$ 
depending on which is the largest (note that $\tX_{s1}>\tX_{s2}$ if and only if $\tilde Y_{s1}>\tilde Y_{s2}$, as 
the forehead factor composed of $\xi$'s is the same for both).

The values $Z_{\|s\|}=\tX_s$, associated to the selected vertices, form a \emph{Markov process}. Note that this 
process is stationary; given $Z_n$, we find $Z_{n+1}$ by sampling $\left (\xi,X_1^{(n)},\ldots,X_4^{(n)}\right )$
conditionally on 
\[Z_n=\lcr \xi\left (\min\left (X_1^{(n)},X_{2}^{(n)}\right )+\min\left (X_3^{(n)},X_4^{(n)}\right )\right )\] 
and taking 
\[Z_{n+1}=\max\left (\min\left (X_1^{(n)},X_{2}^{(n)}\right ),\min\left (X_3^{(n)},X_4^{(n)}\right )\right  ).\]

The scenario arising when $\left \{\tilde{Y}_{s_n}\right  \}_{n\in\N}$ does not tend to zero can be realized 
with positive probability only if we have \eqref{eq:branch}. Notice that the law of the quotient
$\frac{\max(\tilde Y_{s_n1},\tilde Y_{s_n2})}{\min(\tilde Y_{s_n1},\tilde Y_{s_n2})}$, 
conditionally on the past, is completely determined by $Z_n$ (the rescaling factors do not change it).

At the same time, for any compact interval $J\subset\R_+$,
\[
\min_{z\in J}\mathbf{P}\left (\frac{\max(\tilde Y_{s_n1},\tilde Y_{s_n2})}{\min(\tilde Y_{s_n1},\tilde Y_{s_n2})}\le
2\,\middle\vert\, Z_n=z\right )=:p_J>0
\]
as the density of the law of $\frac{\max(\tilde Y_{s_n1},\tilde Y_{s_n2})}{\min(\tilde Y_{s_n1},\tilde Y_{s_n2})}$ 
conditionally on $Z_n=z$ is given by an explicit formula, is positive and continuous, and hence bounded away from
zero once $z$ belongs to $J$. So as soon as we know there exists an interval $J$ such that $Z_n$ visits $J$ 
almost surely infinitely many times, the associated $\tilde Y_{s}$ has infinitely many attempts (with 
probability at least $p_J$ at every time) to be split into two subintervals, whose lengths are in proportion 
less than $2:1$, and hence almost surely goes to zero. Finding such an interval $J$ constitutes the final step
and the most delicate part of the proof of Theorem~\ref{t:geodesic}: we concentrate it in 
Proposition~\ref{p:markov} of~\S\ref{app:tails}.

\subsection{Tail estimates for the stationary measures}\label{app:tails}

The stationarity of the law $\bm$ imposes a strong behaviour on the tails of its distribution, especially 
when the reference measure $m$ has a distribution of fast decrease, as it is the case
with~$m=\exp_*\mN(0,\sigma^2)$ (as all along this section). The main reason is that we can write the 
property that~$\bm$ is a fixed point, $\Pc[\bm]=\bm$, as
\begin{equation}\label{eq:stat-conv1}
\log_*\bm=\nu*\mN(\log\lcr,\sigma^2),
\end{equation}
where $\nu$ is the law of the random variable $\log R_{1}(X_1,\dots,X_4;1)$, with the variables $X_i$'s 
independent and distributed as $\bm$. In other terms, $\nu=(R_1)_*(\bm^{\otimes 4}\otimes\ddelta_1)$.

Roughly speaking, a passage from a measure~$\mu$ to $(R_1)_*(\mu^{\otimes 4}\otimes\ddelta_1)$ makes 
the tails of a distribution decrease faster: 
for $\log R_{1}(X_1,\dots,X_4;1)$ to be large we need at least \emph{two} (``parallel'') large $X_i$'s,
as well as for it to be small we need at least \emph{two} (``consecutive'') small ones. Finally, the 
normal law has tails that decrease sufficiently fast, so the convolution with it does not slow down 
the decrease of $(R_1)_*(\mu^{\otimes 4}\otimes\ddelta_1)$ too strongly. 

All the technical lemmas presented within this section are based on this leading idea, at different levels
of depth.

\subsubsection{Notations}
As we have already explained at the beginning of Section~\ref{s:convergence}, as a consequence of 
\eqref{eq:stat-conv1} the measure~$\log_*\bm$ has a continuous and everywhere positive density with 
respect to the Lebesgue measure, since~$\mN(\log\lcr,\sigma^2)$ has. Keeping the notation from the 
statement of Theorem \ref{t:conv}, we denote by $\rho$ the density of $\log_*m$,
\[\rho(s)=\frac{1}{\sqrt{2\pi\sigma^2}}e^{-\frac{1}{2\sigma^2} s^2},\]
and by $\rho_1$ the density of $\log_*\bm$. The stationarity equation \eqref{eq:stat-conv1} gives the relation
\begin{equation}\label{eq:rho-convolution}
\rho_1(s)=(\nu*\rho)(s-\log\lcr)=\int_{-\infty}^\infty \rho(s-\log\lcr-t)\,d\nu(t).
\end{equation}
We shall first need to work with the tail distribution function of $\log_*\bm$, so we set
\begin{equation}\label{eq:f(s)}
f(s)= \mathbf{P}\left( \log X>s\right)=\int_s^\infty \rho_1(t)\,dt,
\end{equation}
where $X$ is a random variable of law $\bm$. For the same purpose, we define $g$ to be the tail distribution 
function of the random variable $\log R_{1}(X_1,\dots,X_4;1)$, where, as before, the variables $X_i$'s are 
independent and distributed as $\bm$.
The stationarity \eqref{eq:stat-conv1} now reads as
\begin{equation}\label{eq:conv-stat}
f=g*\mN(\log\lcr, \sigma^2).
\end{equation}

\subsubsection{Fast decrease of the tail distribution function}
The first technical result was needed in the proof of Theorem~\ref{t:metric}:
\begin{prop}\label{p:ln}
Let $m=\exp_*\mN(0,\sigma^2)$ be a $\log$-normal distribution and $\{(X_t,\xi_t)\}_{t\in\T}$ the associated 
invariant RTP. Then the random variable $L_n:=\max_{t:\,\|t\|=n}\log X_t$ is almost surely sublinear, 
namely~$L_n=o(n)$.
\end{prop}

In order to prove Proposition~\ref{p:ln}, first we must verify that the random variables $\log X_t$ are not 
too dispersed. Formally, this will be done in the following lemma:

\begin{lem}\label{l:exp-decr}
Let $m=\exp_*\mN(0,\sigma^2)$ be a $\log$-normal distribution, $\lcr$ the associated normalizing constant 
and $\bm$ any non-trivial $\Phi_{\lcr}$-stationary measure. Given a random variable $X$ of law $\bm$, we have 
the superexponential bound
\begin{equation}\label{eq:log-tail}
\mathbf{P}\left( |\log X|>s\right)=o\left(e^{-As}\right) \quad \text{as }s\to\infty,\text{ for any }A>0.
\end{equation}
\end{lem}

\begin{proof}[Proof of Lemma~\ref{l:exp-decr}]
For simplicity, we check the asymptotics \eqref{eq:log-tail} only for the function
$f$ defined by \eqref{eq:f(s)}, the case of $ \mathbf{P}\left( \log X<-s\right)$ being analogous. In fact, 
in view of Proposition~\ref{p:ln} we shall only need to estimate the tail distribution function $f$, though the
other one will be used for the proof of Proposition~\ref{p:markov}.

\smallskip

As explained at the beginning of this part (\S\ref{app:tails}) we use the passage from $\bm$ to $\nu=(R_1)_*(\bm^{\otimes 4}\otimes\ddelta_1)$ 
in order to have a good estimate on the decrease of the tail distribution function $g$: it is easy to notice that 
there exists $c>0$ such that 
\begin{equation}\label{eq:g(s)}
g(s+\log 2)\le 1-\theta(1-f(s))\le c f(s)^2.
\end{equation}
Indeed it suffices to have an $\In\Out$-path composed of edges of length $\le e^s$ to ensure that the 
$\In\Out$-distance is no larger than $2e^s$, so that we obtain the first inequality (this argument is 
essentially the same as in Lemma~\ref{l:psi}). The second inequality in \eqref{eq:g(s)}, can be deduced 
expanding the polynomial function~$\theta$ at~$1$ (up to order $2$): recall that $1$ is an attracting fixed 
point for the function $\theta$ with $\theta'(1)=0$.

\smallskip

The next step is to pass from $g$ to $f$, using the convolution by the Gaussian density.
Let us fix some (large) $C>0$ and compare $f(s)$ to $f(s+C)$: by \eqref{eq:conv-stat} we have
\[
f(s)=\int_{-\infty}^{\infty} g(s-t) \rho(t-\log\lcr) \, dt, \quad f(s+C)= 
\int_{-\infty}^{\infty} g(s-t) \rho(t-\log\lcr+C) \, dt,
\]
Decomposing the integral for $f(s+C)$ at the point $t=-\log 2$, we get the following estimate:
\begin{align*}
f(s+C)&\,=\int_{-\infty}^{\infty} g(s-t) \rho(t-\log\lcr+C) \, dt \\
&\, = \int_{-\infty}^{-\log 2} g(s-t) \rho(t -\log\lcr +C) \, dt 
+ \int_{-\log 2}^{\infty} g(s-t) \rho(t -\log\lcr +C) \, dt  \\ 
&\,\le g(s+\log 2)  \int_{-\infty}^{-\log 2} \rho(t -\log\lcr +C) \, dt +
\int_{-\log 2}^{\infty} g(s-t) \rho(t) \, dt  \sup_{t\ge -\log (2\lcr)} \frac{\rho(t +C)}{\rho(t)}  \\
&\,\le c f(s)^2 + f(s) \sup_{t\ge -\log (2\lcr)} \frac{\rho(t+C)}{\rho(t)} \\
&\,=f(s) \cdot \left(cf(s) + \sup_{t\ge -\log (2\lcr)} \frac{\rho(t+C)}{\rho(t)}\right),\numberthis\label{eq:f(s+C)}
\end{align*}
where in the first inequality we have used that $g$ is non-increasing and in the second one that inequality 
\eqref{eq:g(s)} holds. The second summand in the last factor of \eqref{eq:f(s+C)} can be explicitly computed:
\begin{align*}
\sup_{t\ge -\log (2\lcr)} \frac{\rho(t+C)}{\rho(t)} = 
&\, \sup_{t\ge -\log (2\lcr)} \exp\left(-\frac{(t+C)^2 - t^2}{2\sigma^2} \right)  \\ 
=&\, \sup_{t\ge -\log (2\lcr)} \exp\left( -\frac{2tC + C^2}{2\sigma^2} \right) \\
=&\, \exp\left( -\frac{C-2\log (2\lcr) }{2\sigma^2} \cdot C \right).
\end{align*}
Choose and fix $C$ such that $\frac{C-2\log (2\lcr) }{2\sigma^2}>A$. Then
\[
\sup_{t\ge -\log (2\lcr)} \frac{\rho(t+C)}{\rho(t)} < e^{-AC}.
\]
Since $f(s)\to 0$ as $s\to\infty$ and the function $f$ is non-increasing, there exists $s_0$ such that for any~$s>s_0$
\[
cf(s)+ \sup_{t\ge -\log (2\lcr)} \frac{\rho(t+C)}{\rho(t)} < e^{-AC}.
\]
Then for any $s\ge s_0$ we have $f(s+C)\le e^{-AC} f(s)$, which implies that $f(s)=O\left(e^{-As}\right)$ and 
actually, as $A$ was arbitrary, $f(s)=o\left(e^{-As}\right)$, providing us with the desired~\eqref{eq:log-tail}.
\end{proof}

\begin{proof}[Proof of Proposition \ref{p:ln}]
Fix $\eps>0$, then for any $n$ we have
\[
\mathbf{P}\left(L_n>\eps n\right)\le 4^n\,\mathbf{P}\left( \log X>\eps n\right).
\]
The tail estimate \eqref{eq:log-tail} gives the bound
\[
\mathbf{P}\left(L_n>\eps n\right)\le 4^ne^{-A\eps n}\quad \text{for any }A>0.
\]
In particular, choosing $A>\frac{\log 4}{\eps}$, we have an exponential decrease and applying the 
Borel-Cantelli lemma, we have that, almost surely, $L_n<\eps n$ for $n$ sufficiently large. 
Since $\eps>0$ was arbitrary, we have $L_n=o(n)$, as wanted. 
\end{proof}

\subsubsection{Gaussian control for the decay of the density}
If on the one hand the convolution by the Gaussian measure does not slow down the decrease of the tail 
distribution function $f$, on the other it gives a Gaussian lower bound on the speed, as the following easy
lemma explains:

\begin{lem}\label{l:rho-low}
Let $m=\exp_*\mN(0,\sigma^2)$ be a $\log$-normal distribution, $\lcr$ the associated normalizing 
constant and $\bm$ any non-trivial $\Pc$-stationary measure. Let $\rho_1$ be the density function of 
$\log_*\bm$, then we have the Gaussian lower bound
\[
\rho_1(s)=\Omega\left(e^{-\frac{1}{\sigma^2}s^2}\right) \quad \text{ as } |s|\to\infty.
\]
\end{lem}

\begin{proof}
From the stationary relation \eqref{eq:rho-convolution}, we have the estimate
\[
\rho_1(s)\ge \int_{-1-\log\lcr}^{1-\log\lcr}\rho(s-\log\lcr-t)d\nu(t)
\ge \rho(|s|+1)\cdot\nu([-1-\log\lcr,1-\log\lcr])\]
and the last expression is of the order of $e^{-\frac{1}{2\sigma^2}(|s|+1)^2}$, giving us the desired lower bound.
\end{proof}

\subsubsection{Fast decrease of the density}
Going beyond the asymptotics obtained in Lemma \ref{l:exp-decr} for the tails of the distribution 
function of $\log_*\bm$, we are going to obtain the same kind of estimate for the associated density 
function $\rho_1$.

\begin{lem}\label{l:rho-exp}
Let $m=\exp_*\mN(0,\sigma^2)$ be a $\log$-normal distribution, $\lcr$ the associated normalizing constant 
and $\bm$ any non-trivial $\Pc$-stationary measure. Let $\rho_1$ be the density function of $\log_*\bm$, 
then we have the superexponential bound
\[
\rho_1(s)=o\left(e^{-A|s|}\right) \quad \text{ as } |s|\to\infty, \text{ for any }A>0.
\]
\end{lem}

\begin{proof}
Note that the Gaussian density $\rho(s)$ has the following property of \emph{quasi-convexity}:
\begin{equation}\label{eq:qconv}
\text{for any }s\in\R,\quad \int_{s-\sigma}^{s+\sigma}\rho(u)\,du\ge \frac{\sigma}{10}\,\rho(s).
\end{equation}
Indeed, if $|s|\ge \sigma$, on one of the two half intervals $[{s-\sigma},{s}]$ and $[{s},{s+\sigma}]$, 
we have $\rho\ge \rho(s)$; otherwise, the integral $\int_{s-\sigma}^{s+\sigma}\rho$ is lower bounded by
$\int_{0}^{\sigma}\rho(u) du \ge \sigma \rho(\sigma) =\frac{e^{-1/2}}{\sqrt{2\pi}} $, while 
$\frac{\sigma}{10}\rho(s)$ does not exceed~$\frac{1}{10\sqrt{2\pi}}$.

\smallskip

Now, \eqref{eq:qconv} automatically holds also for a convolution of $\rho$ with any measure, in particular 
for $\rho_1=\nu*\rho(\,\cdot\,-\log\lcr)$. Hence, for any $|s|\ge \sigma$ we have 
\[\rho_1(s)\le \frac{10}{\sigma}\int_{s-\sigma}^{s+\sigma}\rho_1(u)\,du\le 
\frac{10}{\sigma}\,\mathbf P\left (|\log X|>|s|-\sigma\right ).\]
Applying, for any $A>0$, the result of Lemma~\ref{l:exp-decr}, we have, for all $s\in\R$
\[\rho_1(s)\le \frac{10}{\sigma}
\cdot o\left (e^{-A(|s|-\sigma)}\right )=o\left (e^{-A|s|}\right ),\]
hence $\rho_1(s)=o\left (e^{-A|s|}\right )$ for any $A>0$ as $|s|\to\infty$.
\end{proof}

\subsubsection{Drift to the origin}
In order to conclude the proof of Theorem~\ref{t:geodesic}, as we explained at 
the very end of \S\ref{ssc:existence-geodesic}, we need to show that the stationary Markov process 
$(\log Z_n)_{n\in\N}$  visits some compact interval $J$ infinitely many times. We will see in the next paragraph
that the key argument is quite tricky and it is based on a technical estimate that we shall prove now:

\begin{prop}\label{l:exponentialZ}
There exists a compact interval $J\subset \R_+$
such that on any value $z_0$ of $Z_0$ with $z_0\notin J$, the following conditional probability verifies
\begin{equation}\label{eq:Zn}
\mathbf{P}\left (\rho_1(\log Z_{1})\le e\cdot\rho_1(\log z_0)\,\big\vert\,Z_0=z_0\right )\le 
\exp\left (-\sqrt{|\log \rho_1(\log z_0)|}\right ).
\end{equation}
\end{prop}

Before passing to its proof, let us try to explain why such an estimate should hold. Given a compact interval 
$J\subset\R_+$, we shall informally say that $z_0$ is ``large'' if $z_0$ belongs to the connected component of 
$\R_+\setminus J$ containing arbitrary large numbers. If $z_0\in\R_+\setminus J$ belongs to the other connected
component, then we say that $z_0$ is ``small''.

If the value of
\[Z_0=\lcr\xi\left(\min(X_1,X_2)+\min(X_3,X_4)\right)\]
is ``large'', it is very likely that this is due to the fact that the factor $\xi$ is relatively large,
whereas the quantity
\[Z_1=\max\left(\min(X_1,X_2),\min(X_3,X_4)\right)\]
is very likely to be considerably smaller: this is expressed quantitatively by~\eqref{eq:Zn}.

In the same way, if $Z_0$ is ``small'', it is very likely that this is due to the fact that the factor~$\xi$ is 
relatively small, and $Z_1$ is very likely to be considerably larger.

\begin{rem}\label{r:statMarkov}
By the stationarity of the Markov process $(\log Z_n)_{n\in\N}$, we can replace the couple
$Z_0$ and $Z_1$ in the statement with $Z_n$ and $Z_{n+1}$ respectively.
\end{rem}

\begin{proof}[Proof of Proposition~\ref{l:exponentialZ}]
For this time it is convenient to work also with the density of $\bm$ in the original coordinates, and we shall denote it by $\rho_{\bm}$. It is related to the density~$\rho_1$ by a logarithmic change of coordinates: $\rho_{\bm}(z)=\dfrac{1}{z}\rho_1(\log z)$, for any $z\in\R_+$.

Once again we shall use the stationarity relation $\Phi_{\lcr}[\bm]=\bm$, under the form~\eqref{eq:stat-conv1}:
\[
\log_*\bm=\log_*\nu*\mathcal{N}(\log\lcr,\sigma^2).
\]

We first work with the probability measure $\nu=(R_1)_* (\bm^{\otimes 4} \otimes \ddelta_1)$ which is the law of 
the non-rescaled distance $R_1(X_1,\dots,X_4;1)$, 
where the~$X_i$'s are i.i.d.~random variables, distributed with respect to~$\bm$. Using the work done previously, 
we can easily find \emph{global} good estimates when dealing with its density~$\hat{\rho}$. Then we can upgrade 
the bounds on the density~$\hat{\rho}$ to bounds on the density~$\bm$, since we pass from one to the other 
by convolution in the logarithmic coordinates $s=\log z$: the relation~\eqref{eq:stat-conv1} writes 
\begin{equation}\label{eq:mult-conv}
\rho_1  = (e^s \hat{\rho}(e^s) ) * \rho(\cdot - \log \lcr).
\end{equation}

Before proceeding further, let us remark that the density~$\hat{\rho}$ can be expressed in a simple form using 
the density~$\rho_{\bm}$:
\begin{align*}
\hat{\rho}(z)&=8 \int_{x_1+x_3=z, \, x_1\ge x_3} \rho_{\bm}(x_1) (1-F_{\bm}(x_1)) \cdot \rho_{\bm}(x_3) (1-F_{\bm}(x_3)) \, dx_1\\
&=\int_{z/2}^z\rho_{\bm}(x_1) (1-F_{\bm}(x_1)) \cdot \rho_{\bm}(z-x_1) (1-F_{\bm}(z-x_1)) \, dx_1,
\end{align*}
where the expression under the integral corresponds to the particular configuration 
\[
X_1\ge X_3,\quad X_1\le X_2,\quad X_3\le X_4,
\]
the factor~$8$ to the fact that there are $8$ possible configurations, symmetric  each other, and the factors 
$(1-F_{\bm}(x_1))$ and $(1-F_{\bm}(x_3))$ correspond to the inequalities $X_1\le X_2$ and $X_3\le X_4$ 
respectively. Remark that on this particular configuration we have $Z_1=X_1$.

For any $z_0\in\R_+$, let us write~$\delta_0$ for $e\cdot \rho_1(\log z_0)$
and~$\beta_0$ for $\exp\left( -\sqrt{|\log \rho_1(\log z_0)|}\right )$. The quantity
$\frac1e \delta_0=\rho_1(\log z_0)$ is the $\bm$-density in the logarithmic
coordinates of the event $\{Z_0=z_0\}$, on which we are taking the conditional
probability in~\eqref{eq:Zn}. Using the convolution relation~\eqref{eq:mult-conv},
the inequality~\eqref{eq:Zn} is satisfied as soon as we have:
\[
zI_{\delta_0}(z)\le \frac{1}{e}\delta_0\,\beta_0,\quad\text{for any }z\in\R_+,
\]
where the function $I_{\delta_0}$ is defined by
\begin{equation}\label{eq:Iz}
I_{\delta_0}(z):=8\int_{x_1+x_3=z, \, x_1\ge x_3 \atop \rho_1(\log x_1)\le \delta_0}\rho_{\bm}(x_1) 
(1-F_{\bm}(x_1)) \cdot \rho_{\bm}(x_3) (1-F_{\bm}(x_3)) \, dx_1.
\end{equation}
This integral corresponds indeed to the $\nu$-density (in the original coordinates) of the event
\[
\{\min(X_1,X_2)+\min(X_3,X_4)=z,\,\rho_1(\log Z_1)\le \delta_0\}.
\]
Therefore the proof of Proposition~\ref{l:exponentialZ} will be over with the following:

\begin{lem}
With the notations introduced above, there exists a compact interval $J\subset \R_+$ such that for any 
$z_0\notin J$, we have
\[
zI_{\delta_0}(z)\le \frac{1}{e}\delta_0\,\beta_0,\quad\text{for any }z\in\R_+.
\]
\end{lem}

\begin{proof}
Within the set 
\[\mathbf A_z:=\{z/2\le x_1\le z,\,\rho_1(\log x_1)\le \delta_0\},\]
over which we compute the integral defining $I_{\delta_0}$ in~\eqref{eq:Iz}, the variable~$x_1$ is 
linearly comparable to~$z$.

This implies, when~$\delta_0$ is sufficiently small, that the set~$\mathbf A_z$ is empty for relatively small 
values of~$|\log z|$, and thus for such values one has $zI_z(\delta_0)=0$.
This is a plain consequence of the fact that the density function~$\rho_1$ is continuous and tends to 
zero at~$\pm\infty$: for sufficiently small $\delta_0$ and small $|\log z|$, we cannot have simultaneously
\[z/2\le x_1\le z\quad\text{and}\quad\rho_1(\log x_1)\le \delta_0.\]
A key fact is that we can quantify this statement using the Gaussian control for the decay of~$\rho_1$: 
by Lemma~\ref{l:rho-low} and the fact that $z$ and $x_1$ are linearly comparable, we have the lower bound
\[
\rho_1(\log x_1)=\Omega\left (e^{-\frac{2}{\sigma^2}(\log z)^2}\right )\quad\text{as }|\log z|\to\infty.
\]
This implies that there exists some constant $c>0$, which does not depend on~$\delta_0$, such that if 
$z\in\R_+$ verifies
\begin{equation}\label{eq:lowdelta}
c\cdot e^{-\frac{2}{\sigma^2}(\log z)^2}> \delta_0,
\end{equation}
then $\mathbf{A}_z$ is empty.

\smallskip

In the other case, when~\eqref{eq:lowdelta} does not hold, we need a different strategy:
we need to use the bounds on~$f$ and~$\rho_1$ previously obtained,
in order to study the asymptotic behaviour of $\frac{z}{\delta_0}I_{\delta_0}(z)$ 
as~$|\log z|\to\infty$. We claim that it has superpolynomial decay, uniformly on~$\delta_0>0$: for any $A>0$ we have
\begin{equation}\label{eq:superpoldecay}
\frac{z}{\delta_0}I_{\delta_0}(z)=o\left (e^{-A|\log z|}\right )\quad\text{as }|\log z|\to\infty,
\text{ for any }\delta_0>0.
\end{equation}
This estimate is enough to conclude the proof of this lemma. Indeed, from
\[
\beta_0=\exp\left (-\sqrt{|\log \delta_0-1|}\right ),
\]
when $|\log z|$ is sufficiently large, so that~\eqref{eq:lowdelta} does not hold, we have
\begin{equation}
\label{eq:boundbeta}
\beta_0>\exp\left (-\frac{k}{\sigma}|\log z|\right ),
\end{equation}
where $k>0$ is a constant which is bounded from above as $\delta_0$ goes to $0$ (for instance, for any $\delta_0<1$ we have $k<\sqrt{2/\log c}$).

If we had $\frac{z_n}{\delta_0}I_{\delta_0}(z_n)>\frac{1}{e}\beta_0$ for some sequence $\{z_n\}_n$ such that $\lim_{n\to\infty}|\log z_n|=\infty$, then for every sufficiently large~$n$ we would have
\[
\frac{z_n}{\delta_0}I_{\delta_0}(z_n)>\frac{1}{e}\exp\left (-\frac{k}{\sigma}|\log z_n|\right )
\]
because of~\eqref{eq:boundbeta}, and this contradicts the superpolynomial decay~\eqref{eq:superpoldecay} of the left hand side. Therefore there exists a compact interval $K\subset\R_+$ which does not depend on~$\delta_0$ such that 
\[
zI_{\delta_0}(z)\le \frac{1}{e}\delta_0\beta_0\quad\text{for every }z\notin K.
\]
We can choose~$\delta_1>0$ such that for every $\delta_0<\delta_1$ the inequality~\eqref{eq:lowdelta} holds for every~$z\in K$, and therefore for all these values of~$\delta_0$ we have
\[I_{\delta_0}(z)=0\quad\text{on }K.\]
For fixed $\delta_1>0$, there exists a compact interval $J\subset\R_+$ such that the condition $\delta_0=e\cdot\rho_1(\log z_0)< \delta_1$ is satisfied for all $z_0\notin J$. This is exactly the compact interval we wanted.

\medskip

It remains to verify the claim~\eqref{eq:superpoldecay}. Let us make a preliminary observation: within the set~$\mathbf A_z$, 
we have $\rho_1(\log x_1)\le \delta_0$, or equivalently $\rho_{\bm}(x_1)\le \frac{\delta_0}{x_1}$. 
Moreover on $\mathbf{A}_z$ we have $x_1\ge z/2$, and so 
\begin{equation}
\label{eq:boundrho}
\rho_{\bm}(x_1)\le 2\frac{\delta_0}{z}.
\end{equation}
With this we easily obtain a first bound:
\begin{equation}\label{eq:Iz2}
\frac{z}{\delta_0}I_{\delta_0}(z)\le 16\int_{z/2}^z(1-\Fbm(x_1))\cdot \rho_{\bm}(x_3)\,dx_1,
\end{equation}
where the factor $(1-\Fbm(x_3))$ has been upper bounded by $1$ and $\rho_{\bm}(x_1)$ by~\eqref{eq:boundrho}.
We now proceed in two different ways, depending on whether $\log z$ is positive or negative.

As $\log z\to+\infty$, the factor $(1-\Fbm(x_1))=f(\log x_1)$, appearing in the integral~\eqref{eq:Iz2}, 
has superpolynomial decay, after Lemma~\ref{l:exp-decr}.
On the other hand, as $\log z\to-\infty$, it is the density function~$\rho_{\bm}(x_3)$ that has 
superpolynomial decay, after Lemma~\ref{l:rho-exp} and because of the inequality~$x_3\le z$.
In either case, the decay~\eqref{eq:superpoldecay} is proved, and so the lemma.
\end{proof}
This concludes the proof of Proposition~\ref{l:exponentialZ}, as well.
\end{proof}

\subsubsection{End of the proof of Theorem~\ref{t:geodesic}}

\begin{prop}\label{p:markov}
There exists an interval $J\subset\R_+$ such that for any initial value, the process $(Z_n)_{n\in\N}$ 
almost surely visits $J$ infinitely many times.
\end{prop}

\begin{proof}
The process $(Z_n)_{n\in\N}$ is a stationary Markov process and as we have already observed with Remark~\ref{r:statMarkov},
the statement of Proposition~\ref{l:exponentialZ} 
automatically implies the estimate~\eqref{eq:Zn}, when we look at~$Z_{n+1}$ conditionally on $Z_n=z$, with no dependence on~$n$.

\smallskip

The series $\sum_{k=1}^{\infty}e^{-\sqrt{k}}$ converges, so we can choose $k_0\in\N$ such that the tail 
$\sum_{k=k_0}^{\infty}e^{-\sqrt{k}}$ is less than $\tfrac12$.

At the price of extending the interval $J$ provided by Proposition~\ref{l:exponentialZ}, we can assume the bound 
$|\log\rho_1(s)|\ge n_0$ outside $\log J$ (recall that, after Lemma~\ref{l:rho-exp}, $\rho_1(s)$ tends 
to zero as $s$ goes to $\pm\infty$).

\smallskip

We claim the following estimate: conditionally on any given $Z_0=z$, the probability of hitting $J$ is at least
\begin{equation}\label{eq:hit}
1-\sum_{k=k_0}^{[|\log\rho_1(\log z)|]}e^{-\sqrt{k}}.
\end{equation}
We prove this claim by induction on $K=[|\log\rho_1(\log z)|]$. If $K=k_0-1$, there is nothing to prove: 
we immediately have $Z_0\in J$.  Suppose that we have already proved the claim up to $K$. Then, 
assuming $[|\log\rho_1(\log Z_0)|]=K+1$ we have $\rho_1(\log Z_{1})\ge e\cdot\rho_1(\log Z_0)$ unless for 
an event of probability at most $\exp(-\sqrt{K+1})$. Thus, $[|\log\rho_1(\log Z_1)|]\le K$ and the induction 
assumption applies.

\smallskip

For any initial value $Z_0\not\in J$ consider the segment $Z_0,\ldots,Z_{n_1}$, cut at the first time 
$n_1$ when either 
\begin{enumerate}
\item $\rho_1(\log Z_{n_1+1})\le e\cdot \rho_1(\log Z_{n_1})$, or
\item $Z_{n_1}\in J$.
\end{enumerate}
The second outcome has probability at least $\tfrac12$ after~\eqref{eq:hit} and due to the choice of~$k_0$.

On the other hand, if the first outcome has taken place, we then consider the next segment
$Z_{n_1+1},\ldots, Z_{n_2}$. Then, if $Z_{n_2}\not\in J$, the segment $Z_{n_2+1},\ldots, Z_{n_3}$, and so on. 
At each step we have probability at least $\tfrac12$ of hitting $J$, therefore the process $(Z_n)_{n\in\N}$ 
almost surely hits $J$, sooner or later.

This concludes the proof of Proposition~\ref{p:markov} and thus of Theorem~\ref{t:geodesic}.
\end{proof}

\section{Hierarchical graphs}
\label{s:hierarchical}

Hierarchical graphs have been studied so long due to their nice self-similar structure that it is impossible 
for us to give an exhaustive list of references on them (see \cite{HK} as a recent example, some historical 
references can be found therein).

It should be clear that what we have done for the figure eight-graph, naturally 
generalizes to hierarchical graphs built with different ``bricks''. We will provide details on how to proceed, 
putting more care in those cases -- the pivotal graphs -- for which some of the arguments need to be modified,
and the hypotheses
on the measure~$m$ strengthened.

\subsection{Notations}\label{ssc:def_hier}
We consider a finite oriented graph $\G=(V,E)$ with two marked 
distinct points $I$ and $O$. We define the corresponding hierarchical graph in the following way:

\begin{dfn}[Hierarchical graphs]
\label{dfn:hierarchical}
Define the sequence of graphs $(\G_n,\In,\Out)$ inductively. We start with $\G_0$ which is simply an interval 
with endpoints $\In$ and $\Out$, oriented from $\In$ to $\Out$. For the next step, when $\G_{n-1}$ is already 
constructed, we replace each oriented edge $e=(x,y)$ 
with a copy of the graph $\G$ in such a way that $x$ coincides with $I$ and $y$ with $O$. 

As for the hierarchical figure eight-graph, the vertices $V_{n-1}$ of $\G_{n-1}$ are naturally included into 
the vertices $V_n$ of $\G_n$. 
The \emph{generation} of a vertex $x$ is the minimum $i$ such that $x\in V_i$.
The \emph{in} and \emph{out} vertices of $\G_n$ are the vertices $I$ and $O$ of the $0^{th}$ generation 
respectively.
\end{dfn}

When the graph distance $r=d_\G(I,O)>1$, there is a non-trivial (``Euclidean'') geometry 
on~$\G_\infty$ which is the limit of the $(\G_i,\frac{1}{r^n}d_{\G_n})$'s in the Gromov-Hausdorff 
topology. Indeed, rescaling the distance function $d_{\G_i}$ by the factor 
$r^{-n}$ guarantees that the rescaled $\In\Out$-distance stays equal to~$1$ at every step. We want to show 
that even for 
\emph{random metrics}, we are able to define such a limit in law, in the same way that we 
managed for the figure eight-graph.

\subsection{$(\min,+)$-type recursive distributional equations}

Given any finite graph $(\G,I,O)$ as in the previous section, we can associate to it a function~$\rho_\G$ 
of~$\# E$ variables $x_1,\ldots,x_{\# E}$ involving only $\min$ and $+$ 
operations in the way that we explain hereafter.

\begin{dfn}
A path $\pi$ in $\G$ from $I$ to $O$ is a \emph{simple path} if it crosses every edge at most 
once.

Let us fix a labelling $\{1,\ldots,\# E\}$ for the edges, then to any simple path $\pi$ we 
associate the linear function 
$\ell(\pi)(x_1,\ldots,x_{\# E})=\sum_{i=1}^{\# E} \mathbf{1}_{i\in \pi}\, x_i$.

Finally we define the function $\rho_\Gamma$ as 
$\rho_\G=\min_{\pi\textrm{ simple path}} \ell(\pi)$.
\end{dfn}

Sometimes the function $\rho_\G$ can have a nicer expression, for instance the reader may 
think of the function $\rho$ given by the figure eight-graph: the definition of $\rho$ 
simplifies to~$\rho(x_1,x_2,x_3,x_4)=\min(x_1,x_2)+\min(x_3,x_4)$.

Given any probability measure $m$ on $(0,+\infty)$, we define 
the \emph{recursive distributional equation}
\beqn{rde}{
Y=\xi\cdot \rho_\G (X_i),
}
where $\xi$ is a random variable of law $m$, the $X_i$'s are i.i.d.~random 
variables whose law $\mathcal{L}$ is the same as that of $Y$ and is the unknown of the problem.

We shall suggest the reader to compare \eqref{eq:rde} with equations 
\eqref{eq:Phi}-\eqref{eq:iter}.

\begin{rem}
Conversely, it is sometimes possible to associate a ``brick'' graph to $(\min,+)$-type RDEs: for instance if
\[
Y=\xi\cdot \rho (X_i)
\]
is the RDE we want to study, then a sufficient condition is that the function $\rho$ makes 
the variables~$X_i$'s appear only once (see Figure~\ref{f:RDE-example}). 
\end{rem}

\begin{figure}[ht]
\[
\includegraphics[width=.18\textwidth]{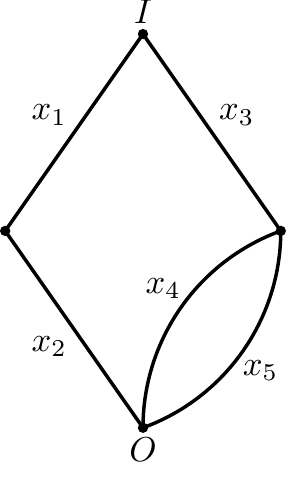}
\]
\caption{The ``brick graph'' associated to $\rho(x_i)=\min(x_1+x_2,x_3+\min(x_4,x_5))$. The orientation 
for the edges can be chosen arbitrarily.}\label{f:RDE-example}
\end{figure}

\subsection{Bond percolation on hierarchical graphs and stationary random distances}\label{ssc:pivotal}

\subsubsection{General setting}

In the proof of Theorem~\ref{t:stat}, and hence of Theorem~\ref{p:metric}, the percolation 
function~$\theta$ introduced in Definition~\ref{dfn:psi} played a crucial r\^ole. 
Such a function can be associated to any ``brick graph'' $(\G,\In,\Out)$:
\begin{dfn}
Given any $p\in [0,1]$, we consider the Bernoulli percolation of parameter $p$ on~$\G$. 
This defines the following function $\theta_\G$ of variable $p$
\[
\theta_{\G}(p)=\P(I\to O),
\]
where $I\to O$ denotes the event that $I$ is connected to $O$.
\end{dfn}
It is easy to observe that $\theta_{\G}$ defines a monotone continuous function, fixing $0$ and 
$1$: $\theta_{\G}$ is a homeomorphism of the closed interval $[0,1]$ (defined by a polynomial 
expression in $p$).

We have already remarked that a key feature of the function $\theta_{\G}$ associated to the figure eight-graph 
is the nature of the fixed points of $\theta_{\G}$: $0$ and $1$ are super-attracting and there is only another 
one in~$(0,1)$. Indeed, this is the only property of $\theta_{\G}$ that we really needed for the proofs of 
Theorem~\ref{t:stat} and hence of Theorem~\ref{p:metric}. 

Though, such a behaviour is not a special feature of the figure eight-graph, but it is peculiar for graphs 
that do not possess \emph{pivotal edges}:
\begin{dfn}
An edge in the graph $(\Gamma,\In,\Out)$ is \emph{pivotal} if it is a straight $\In\Out$-edge (\emph{shortcut}), or if its 
removal disconnects $\In$ from $\Out$ (\emph{bridge}). Otherwise it is called \emph{non-pivotal}. A graph 
is \emph{pivotal} if it possesses at least one pivotal edge, and \emph{non-pivotal} otherwise.
\end{dfn}
Indeed, it is easy to see that for any non-pivotal graph $\G$ 
the points~$0$ and $1$ are super-attracting for the 
percolation function $\theta_{\G}$:
\begin{prop}
Let $(\G,I,O)$ be any finite graph with two distinct marked points and no pivotal edges. 
Then $0$ and $1$ are super-attracting points for the dynamics induced by the map 
$\theta_{\G}$ associated to~$\G$ (\emph{i.e.}~the first derivative of $\theta_{\G}$ at $0$ and at $1$ is equal $0$). 
This infers that $\theta$ has at least \emph{three} fixed points on $[0,1]$.
\end{prop}

\begin{proof}
The weight of every subgraph $\Delta\in \{0,1\}^E$ is $p^{\#_1\Delta}(1-p)^{\#_0\Delta}$. 
A path connecting $I$ to $O$ must pass at least through $d_\G(I,O)$ edges: if 
$\Delta\in (\io)$ then $\#_1\Delta\ge d_\G(I,O)$. Hence, the order of the zero of $\theta_{\G}(p)$ at $p=0$ 
is $d_\G(I,O)$.

Furthermore, if in the configuration $\Delta$ there is no $\In\Out$-path, 
we must have $\#_0\Delta\ge 2$, for the absence of bridges implies the existence in $\G$ of at least 
two disjoint possible $\In\Out$-paths, each of which should be cut. Hence, the order of the zero of 
$1-\theta_{\G}(p)$ at $p=1$ is at least two.
\end{proof}

\begin{ex}
The smallest non-pivotal graphs are the figure eight-graph and its dual, the diamond-graph (see 
Figure~\ref{f:8diamond}).
\end{ex}

For simplicity reasons we will assume that any edge of $\G$ belongs to at least one simple $\In\Out$-path 
(otherwise it can be removed), and we assume that the graph $\G$ is not a single $\In\Out$-edge.

A theorem of Moore and Shannon~\cite{MS} says that after the change of variable 
$s=\log \frac{p}{1-p}\in\R$ the map~$\theta_{\G}$ becomes expanding and hence possesses a 
unique repelling fixed point:

\begin{thm*}[Moore~--~Shannon, \cite{MS}]\label{thm:Moore-Shannon}
With the previous notations, if there are no pivotal edges, the function $\theta_{\G}$ has exactly 
three fixed points on $[0,1]$.
\end{thm*}

Let us summarize the different behaviours of the function $\theta_{\G}$ according to the properties of the graph $\G$:

\begin{enumerate}
\item The graph $\G$ has no pivotal edges, then $\theta_{\G}$ has three fixed points in the interval~$[0,1]$:
one repelling inside, and the endpoints $0$ and $1$ that are super-attracting.
\item The graph $\G$ has a shortcut, that is, an $\In\Out$-edge. Then we have the inequality $\theta_{\G}(p)>p$ 
everywhere inside $(0,1)$, and the only fixed points of~$\theta_{\G}$ are the endpoints $0$ and $1$, with $1$ 
super-attracting, and $0$ (topologically) repelling.
\item The graph $\G$ has a bridge, that is, an edge that any $\In\Out$-path is obliged to cross. Then, we have 
the inequality $\theta_{\G}(p)<p$ everywhere inside $(0,1)$, and the only fixed points of~$\theta_{\G}$ are
the endpoints $0$ and $1$, with $0$ super-attracting, and $1$ (topologically) repelling.
\end{enumerate}

\subsubsection{Non-pivotal graphs}

The arguments from Sections~\ref{s:existence} through~\ref{s:metric} generalize verbatim to the case 
of non-pivotal graph $\Gamma$, providing us with analogues of 
Theorems \ref{p:metric} through~\ref{t:conv}. 
For instance, Theorems \ref{t:stat} and \ref{t:conv} translate as follows: 
\begin{thm}\label{t:non-pivotal}
Given any non-pivotal finite graph $(\G,I,O)$ as in \S\ref{ssc:def_hier}, for 
any non-atomic, fully supported probability distribution $m$ on $\R_+$ there exists a 
normalizing constant $\lcr\in\R_+$ and a non-atomic probability distribution $\bm$ on 
$\R_+$ such that $\bm$ is a solution of the modified RDE \eqref{eq:rde}:
\[
Y=\lcr\,\xi\cdot\,\rho_\G (X_i).
\]
This also implies the existence of a $\mPhi_{\lcr;\G}$-stationary random metric on the hierarchical graph 
associated to $\G$, where $\mPhi_{\lcr;\G}$ is the glueing operator defined for the graph $\G$.

In addition, if $m$ is absolutely continuous with the density as in Theorem~\ref{t:conv}, then $\bm$ 
is  unique up to rescaling. Moreover, for any probability measure $\mu$ on $\R_+$ there exists a
constant $c>0$ such that the iterations of this measure converge to the $c$-rescaled
measure~$\Upsilon_c[\bm]$:
\[
\Phi_{\lcr;\G}^n(\mu)\to \Upsilon_c[\bm] \quad \text{as } n\to \infty,
\]
where $\Phi_{\lcr;\G}$ is the glueing operator corresponding to the graph~$\G$.
\end{thm}

Indeed, as we have already mentioned, the proofs of the theorems regarding the hierarchical figure eight-graph, 
use in fact the behaviour of the associated function $\theta$ (or, what is the same, the fact that one needs at 
least two very long -- parallel -- edges to form a very long $\In\Out$-distance, and at least two very short 
ones for it to be very short).

\subsubsection{Graphs with bridge edges}

When the graph $\G$ is pivotal, we cannot obtain such results under the same general hypotheses on the 
measure~$m$, and we shall impose additional assumptions. In this work we restrict 
ourselves to the analogues of~Theorems~\ref{p:metric} and \ref{t:stat}: even if it is possible to study 
the problem of convergence to the stationary measure, the analysis required for this is more involved than 
what we did throughout Section~\ref{s:convergence}.

We start with the case of a graph with a bridge; this case (as we will see in Section~\ref{s:sierpinski}) is
highly similar to the one for the Sierpi\'nski Gasket and for this reason it takes some priority here.

\begin{thm}\label{t:pivotal-bridge}
Given any pivotal finite graph $(\G,I,O)$ with a bridge edge (as in \S\ref{ssc:def_hier}), for 
any non-atomic, fully supported probability distribution $m$ on $\R_+$ with finite first moment, there exists a 
normalizing constant $\lcr\in\R_+$ and a non-atomic probability distribution $\bm$ on 
$\R_+$ such that $\bm$ is a solution of the modified RDE \eqref{eq:rde}:
\[
Y=\lcr\,\xi\cdot\,\rho_\G (X_i).
\]
This also implies the existence of a $\mPhi_{\lcr;\G}$-stationary random metric on the hierarchical graph 
associated to $\G$, where $\mPhi_{\lcr;\G}$ is the glueing operator defined for the graph $\G$.
\end{thm}

Before passing to its proof, we consider the most basic example of such graphs:
\begin{ex}\label{ex:Galton-Watson}
Let $(\G,I,O)$ be the interval formed by $d$ edges in a row, then the associated RDE is associated to a 
Galton-Watson process (it should be thought also as a 1D Mandelbrot Multiplicative Cascade)
\[
R_{\lambda}(X_1,\ldots,X_d;\xi)=\lambda\xi\sum_{i=1}^dX_i.
\]
When $\xi$ has finite $\alpha$-moment, with $\alpha>1$, it is classical that there is $\lcr>0$ such that 
there is a unique (up to rescaling) stationary probability measure and the convergence is exponential 
(the contraction method applies~\cite{RDE}).
\end{ex}

\begin{proof}[Sketch of the proof of Theorem~\ref{t:pivotal-bridge}]
In this case, the function $\theta_{\G}$ has $1$ as  (topologically) repelling fixed point, and $0$ as  
super-attracting fixed point. We  can still launch the cut-off procedure, defining the supercritical set 
$\Lambda$ in the same way, and the arguments of the non-emptiness of $\Lambda$ used in Lemma~\ref{l:Lambda}
still work. 

However, the non-emptiness of $(0,+\infty)\setminus\Lambda$ cannot  be ensured any more by the same arguments: 
the point~$1$ is not attracting for $\theta_{\G}$. Instead we will use some analogy with the Galton-Watson 
process considered in Example~\ref{ex:Galton-Watson}. More precisely, let us write with abuse of notation 
$\E[m]$ for the first moment of a random variable of law~$m$, we will show that if 
$\lambda<\frac{1}{d_{\G}(\In,\Out) \cdot\E [m]}$, then $\lambda$ cannot be supercritical.

Indeed, consider the associated cut-off operator $\Phi_{A,\lambda;\G}$ and the sequence of measures 
$\mu^{A,\lambda}_n$ defined analogously to~\eqref{eq:def-G}. Then $\mu^{A,\lambda}_1=\ddelta_{A}$ and 
all the measures $\mu^{A,\lambda}_n$'s with $n\ge 1$ are supported on~$[0,A]$. Using one $\In\Out$-path 
for an upper bound, we obtain
\[
\E[ \mu^{A,\lambda}_{n+1}]\le d_{\G}(\In,\Out)\cdot \lambda \E[ m] \cdot \E [\mu^{A,\lambda}_{n}],
\]
and hence 
\[
\E[ \mu^{A,\lambda}_{n+1}]\le A \cdot (d_{\G}(\In,\Out)\cdot \lambda \E [m])^n \to 0 \quad\text{as } n\to \infty.
\]
Thus, $\nu_{A,\lambda}=\lim_{n\to\infty} \mu^{A,\lambda}_{n}$ is the Dirac measure concentrated at~$0$, 
and $\lambda$ is not supercritical.

The proof of the Key Lemma still works in the same way (the point $0$ is super-attracting for $\theta_{\G}$), 
and we still define $\lcr=\inf \Lambda$. From the openness of~$\Lambda$ 
one constructs the stationary measure~$\bm$ as a ``diagonal'' limit in the same way as in 
Lemma~\ref{l:subseq-diagonal}. Finally, since the map $\theta_{\G}$ does not have fixed points 
inside $(0,1)$, we get immediately from the stationarity (and non-triviality) of $\bm$ that $\bm$ has 
no atoms, neither at~$0$ nor at infinity.
\end{proof}

\subsubsection{Graphs with shortcut edges}

The case when there is a ``shortcut'' edge linking $\In$ to $\Out$ directly is slightly different. A first remark is that, as we will see, we have to vary our cut-off procedure:
we replace the upper cut-off by the lower one. However the analogue of Theorem~\ref{t:stat} remains valid in this setting:
\begin{thm}\label{t:pivotal-shortcut}
Given any pivotal finite graph $(\G,I,O)$ with a shortcut $\In\Out$-edge (as in \S\ref{ssc:def_hier}),
for any non-atomic, fully supported probability distribution $m$ on $\R_+$ with finite first negative moment, 
there exists a normalizing constant $\lcr\in\R_+$ and a non-atomic probability distribution $\bm$ on 
$\R_+$ such that $\bm$ is a solution of the modified RDE \eqref{eq:rde}:
\[
Y=\lcr\,\xi\cdot\,\rho_\G (X_i).
\]
This also implies the existence of a $\mPhi_{\lcr;\G}$-stationary random metric on the hierarchical graph 
associated to $\G$, where $\mPhi_{\lcr;\G}$ is the glueing operator defined for the graph $\G$.
\end{thm}
\begin{rem}\label{r:geometry-shortcut}
This implies the analogue of Theorem~\ref{p:metric}, in the same way as before, by applying Proposition~\ref{p:equiv}.
However the stationary
random metric space has no chance of being homeomorphic to the ``Euclidean'' version (in contrast to what happens in the 
subcritical case of Theorem~\ref{t:metric} for non-pivotal graphs). Indeed, in the ``Euclidean''
limit there still is an $\In\Out$-shortcut edge, which is obtained by the sequence of replacements by shortcut edges. However, to
this shortcut it is associated an infinite product of i.i.d.~ non-constant random variables. 
There is no chance for this product to converge to a finite non-zero limit. Hence the ``length'' of this edge should 
be either zero, or infinite. But if it was zero, the $\In\Out$-distance would be collapsed, 
and hence by stationarity all the distances would be collapsed, too. Therefore the shortcut 
edge should be of infinite length in the sense of the limit metric, and so this edge should be ``cut'' (in the topological sense).
\end{rem}

Again, before passing to the proof, consider the most basic example: it is the case of the extrema of a branching 
random walk.
\begin{ex}\label{ex:BRW}
When $(\G,I,O)$ consists of two parallel edges, the RDE is associated (using logarithmic coordinates) to the
extrema of a deterministic $\BRW$:
\[R_{\lambda}(X_1,X_2;\xi)=\lambda\xi\,\min(X_1,X_2).\]
When the random variable $\xi$ has some finite positive moment, it is classical~\cite{hammersley}
that a critical parameter~$\lcr$ exists and it is actually given by $\exp(-\gamma_{\BRW})$, where the constant~$\gamma_{\BRW}$ is
the one defined in~\eqref{eq:BRWconstant}, also corresponding to the unique solution $\lambda$ of
\[
\inf_{\theta\ge 0}\E[(\lambda\xi)^\theta]=\tfrac12
\]
(cf.~\cite[equation (24)]{RDE}). In order to ensure the existence of a stationary measure, and possibly the uniqueness (up to scalar
multiplication), stronger conditions on the random variable~$\xi$ are needed and this, for general $m$,
still constitutes an active subject of the current research. Also, the convergence to the stationary measure 
can be established under some assumptions, for example if the distribution function of $\log \xi$ has 
superexponential decay~\cite{BZ}; in any case, the convergence is not exponential (cf.~Remark~\ref{r:critical}).
For further reading, see also~\cite{BRW,aid}.
\end{ex}

\begin{rem}\label{r:conj-pivotal}
Commenting further on the previous example, we expect that the relevant difference between these two examples is 
the presence of the ``$+$'' operation in the first RDE, which intertwines the core part with the tails of the 
distribution. We conjecture that the convergence to the stationary measure is exponential when $d_\G(I,O)\ge 2$ 
(and when there is some reasonable moment condition on~$\xi$).
\end{rem}

\begin{proof}[Sketch of the proof of Theorem~\ref{t:pivotal-shortcut}]
The presence of a shortcut $\In\Out$-edge in $\Gamma$ makes $0$ a topologically repelling fixed point for the 
map~$\theta_{\G}$. Due to this, instead of considering the \emph{upper cut-off}, that is, considering the 
law of $\min(A, R_{\lambda}(X_i;\xi))$, we consider the \emph{lower cut-off}, defining the operator 
$\hat{\Phi}_{a,\lambda;\G}$ that sends a measure $\mu$ to the law of $\max(a, R_{\lambda}(X_i;\xi))$, 
where the $X_i$'s are distributed with respect to~$\mu$. Likewise, we change the definition of the measures 
$\mu^{A,\lambda}_n$, considering the sequence $\left\{\hat{\mu}^{a,\lambda}_n\right\}$, starting with 
$\hat{\mu}^{a,\lambda}_0=\ddelta_0$ and then recursively
\[
\hat{\mu}^{a,\lambda}_n = \hat{\Phi}_{a,\lambda;\G} [\hat{\mu}^{a,\lambda}_{n-1}]
\]
(cf.~\eqref{eq:def-G}). Passing to the (monotone) limit, we are able to define 
$\hat{\nu}_{a,\lambda}=\lim_{n\to\infty} \hat{\mu}^{a,\lambda}_n$ (as in Definition~\ref{def:nu}).

\medskip

We consider then the \emph{subcritical} set $\hat{\Lambda}$ defined by 
\[
\hat{\Lambda}=\{\lambda>0 \mid \forall\,a\in\R_+\text{ one has }\hat{\nu}_{a,\lambda}\neq \ddelta_{\infty}\}.
\]
Adapting the arguments given in the proof of Lemma~\ref{l:Lambda}, it is not difficult to show that $\hat{\Lambda}$ 
is nonempty (as the point $1$ is super-attracting for~$\theta_{\G}$). On the other hand, as the point~$0$ is 
(topologically) repelling, we need a different strategy to show the non-emptiness of 
$(0,+\infty)\setminus\hat{\Lambda}$. Similarly to the case of graphs with bridge edges, we shall compare the 
dynamics defined by $\hat{\Phi}_{a,\lambda;\G}$ with the one of the simplest example, that is, the extrema of 
a branching random walk (Example~\ref{ex:BRW}).

\smallskip

Indeed, let $d$ be the degree of the vertex $\In$ in $\G$, that is, the number of edges having $\In$ as one 
of the endpoints. Given a measure $\mu$, we denote by $X_1,\ldots, X_d, X_{d+1}, \ldots, X_{\# E}$ a 
collection of independent random lengths for the edges of $\G$, of law $\mu$. Without loss of generality, we can
suppose that the first $d$ lengths $X_i$'s are associated to the edges attached to $\In$. Since any 
$\In\Out$-path must contain one of these $d$ edges, we have the inequality
\begin{equation}\label{eq:compBRW}
\rho_\G(X_i)\ge \min(X_1,\ldots, X_d).
\end{equation}
Consider now $\lambda>d\cdot \E[\frac1m]$ (again, with abuse of notation, $\E[\frac1m]$ denotes the first negative 
moment of a random variable of law~$m$); we want to show that $\lambda$ is supercritical, that is
$\lambda\not\in \hat\Lambda$. For this purpose, note that if $Y_1,\ldots,Y_d$ are i.i.d.~positive 
random variables with finite first moment, the plain inequality $\max(Y_1,\dots,Y_d)\le Y_1+\dots+Y_d$ implies that
\[\E[\max(Y_1,\ldots,Y_d)]\le d\cdot\E[Y_1].\]
Hence, looking at the inverses of the $ \hat{\mu}^{a,\lambda}_n$-random $\In\Out$-distances, using the 
inequality \eqref{eq:compBRW}, we have
\[\E\left[\frac{1}{\hat{\mu}^{a,\lambda}_{n+1}}\right]\le
\frac{d}{\lambda}\cdot\E\left[\frac1m\right]\cdot \E\left[\frac{1}{\hat{\mu}^{a,\lambda}_{n}}\right]\]
and therefore
\[
\E\left[\frac{1}{\hat{\mu}^{a,\lambda}_{n+1}}\right]\le
\frac1a \cdot\left(\frac{d}{\lambda} \cdot \E\left[\frac1m\right]\right)^n\to 0 \quad\text{as }  n\to\infty.
\]
This implies that the limit $\hat{\nu}_{a,\lambda}=\lim_{n\to\infty}\hat\mu^{a,\lambda}_n$ is the Dirac 
measure concentrated at $\infty$, and $\lambda$ is supercritical.

\medskip

The statement of the Key Lemma still holds in this case (the proof is absolutely analogous, with some 
natural changes of relations: choosing $a'<a$ instead of $A'>A$, and so on). Defining $\lcr:=\sup \hat{\Lambda}$ 
allows us to find the stationary measure $\bm$ as a diagonal (subsequential) limit in the same way as in 
Lemma~\ref{l:subseq-diagonal} (again, with the natural changes of relations: this time $\lambda\nearrow \lcr$,
which implies $a(\lambda)\to 0$).

Finally, as the map $\theta_{\G}$ does not have fixed points inside $(0,1)$, we get immediately from 
the stationarity (and non-triviality) of $\bm$ that $\bm$ has no atoms neither at~$0$, nor at infinity.
\end{proof}

\subsubsection{Additional remarks}

\begin{rem}
It is also possible to study hierarchical graphs built out of 
\emph{more than one kind of ``brick''}, say $\G^1,\G^2,\ldots,\G^d$. For example, at 
every step in the construction, one can choose randomly to insert one graph among the 
$\G^i$'s. These different processes are still defined recursively and the stationary 
solutions satisfy some RDEs that can be treated with the tools introduced throughout this work.
\end{rem}

\begin{rem}
Let us suppose that the graph $\G$ is planar and 
that its embedding into the plane is chosen and fixed. 
Then we can define the dual graph $(\G^*,I^*,O^*)$: the 
graph $\G^*$ is the dual graph of $\G$, but we must be a little cautious defining $I^*$ and 
$O^*$.
Let $e_{IO}$ be an additional edge between $I$ and $O$, in such a way that in the chosen 
embedding it stays outside the graph: just to fix notations we make this edge pass 
through the point at infinity. Then there are well defined left and right hand 
sides of the graph. We set them to correspond respectively to $I^*$ and $O^*$.

In the planar case, the construction of the hierarchical graph preserves the duality between 
graphs:~$(\G_i)^*=(\G^*)_i$. Remark that the percolation function
of a graph $\G$ is conjugated to the one associated to the dual graph $\G^*$:
\[\theta_{\G^*}(p)=1-\theta_\G(1-p).\] 
Interestingly, there is no apparent relation between solutions of the RDEs defined for a planar graph 
$(\G,I,O)$ and $(\G^*,I^*,O^*)$. 
\end{rem}

\begin{figure}[ht]
\[
\includegraphics[width=.4\textwidth]{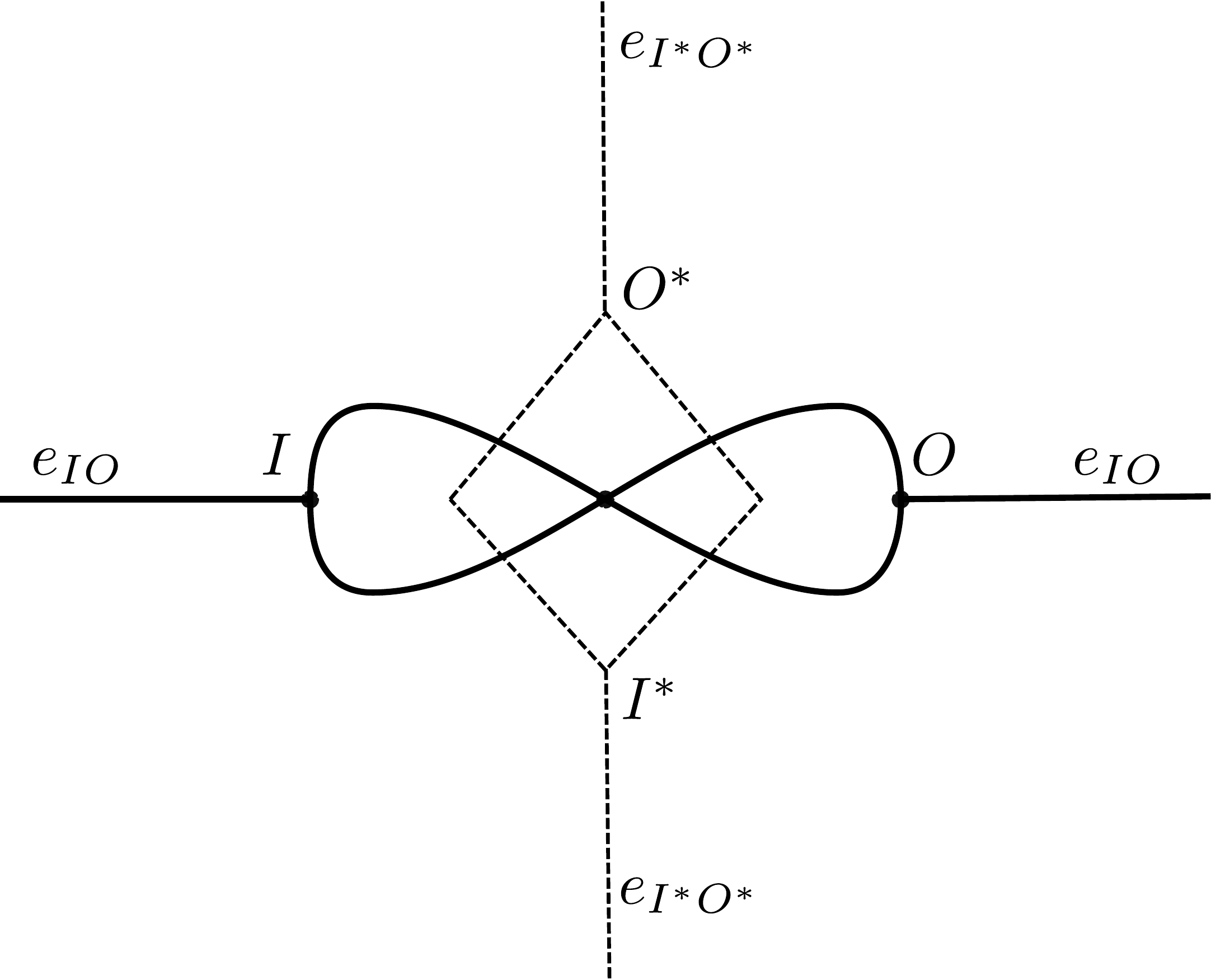}
\]
\caption{The figure eight-graph and its dual graph, the diamond.}\label{f:8diamond}
\end{figure}

\section{Stationary random metrics on the Sierpi\'nski Gasket}\label{s:sierpinski}

Although generalizing our main results to hierarchical graphs is rather immediate, some care is needed 
when studying further classes of self-similar space, which include for example the well-known \emph{Sierpi\'nski Gasket}.

To this extent, we recall (see \cite{BBI}) that a metric space $(X,d)$ is a \emph{length space} if the distance
between any two points is equal to the infimum of the lengths of the paths joining them. We call $d$ a
\emph{length metric} on $X$ if $(X,d)$ is a length space.

\begin{dfn}
A compact metric length space $(X,d)$ is a \emph{self-similar length space} if there exist finitely many 
scalar contractions $\phi_i:(X,d)\to (X,d)$, $i=1,\ldots,N$ such that 
\begin{enumerate}
\item $X=\bigcup_{i=1}^N\phi_i(X)$,
\item there exists an open dense set $O\subset X$ such that $O\supset\bigcup_{i=1}^N \phi_i(O)$, and this 
union is disjoint.
\end{enumerate}
\end{dfn}

Formally, we will be working with the very peculiar self-similar spaces $(X,d_0)$ for which the intersections 
$\phi_i(X)\cap \phi_j(X)$ is at most one point (hierarchical graphs satisfy this hypothesis). Without getting 
into the broadest possible setting, we will consider the illustrative example of the \emph{Sierpi\'nski Gasket} 
$\Sigma$. In the recursive construction we have the three distinguished vertices $\vs_1$, $\vs_2$ and $\vs_3$.
\begin{figure}[ht]
\[
\includegraphics[width=.2\textwidth]{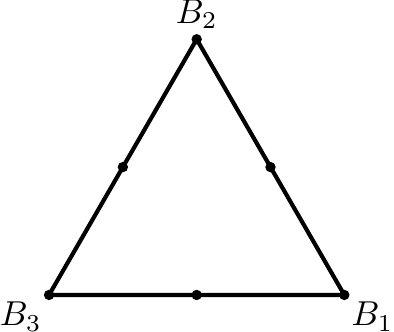}
\qquad
\includegraphics[width=.2\textwidth]{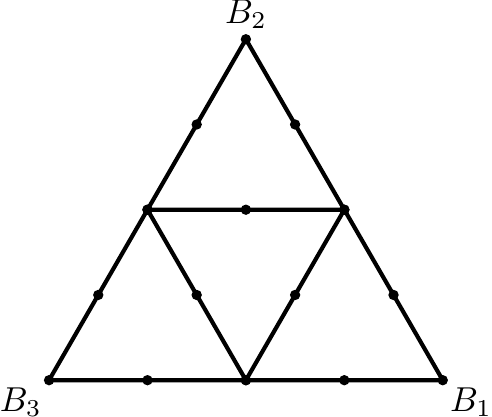}
\qquad
\includegraphics[width=.2\textwidth]{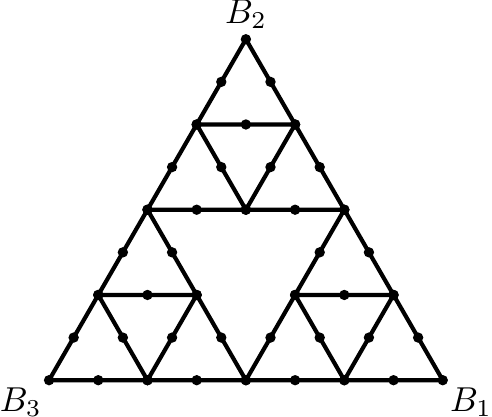}
\]
\caption{The Sierpi\'nski Gasket}\label{fig:Sierpinski}
\end{figure}

As in the case of hierarchical graphs,
instead of studying directly random infinite-dimensional objects (\emph{i.e.}~\emph{random metrics}~$d$),
it is simpler to deal first with random ``marginal'' vectors of dimension~$3$: triples of 
\emph{random distances} $d({\vs_1,\vs_2})$, $d({\vs_2,\vs_3})$, $d({\vs_3,\vs_1})$ between any two of the
three vertices. 

Studying the dynamics for these marginal vectors and finding a stationary measure in the way that will 
be described below, we then can return to the construction of a stationary random metric in 
the same way as it was done in Proposition~\ref{p:equiv}. 

\subsection{Notations}

We shall keep most of the notations introduced for the figure eight-hierarchical graph. That is, let $\Sigma_0$ be the usual triangle (Figure~\ref{fig:Sierpinski}, left), $\Sigma_1$ the three triangles glued by their vertices (the first step of the construction of the Sierpi\'nski Gasket, Figure~\ref{fig:Sierpinski}, centre), and for any $n$ let $\Sigma_{n+1}$ be obtained from $\Sigma_n$ by replacing each small triangle (which is a copy of~$\Sigma_0$) by a copy of $\Sigma_1$. Denoting by $V_n$ the set of vertices
of $\Sigma_n$, we have a natural inclusion $V_n\hookrightarrow V_{n+1}$. The set~\mbox{$V_\infty=\bigcup_{n}V_n$} in
the Sierpi\'nski Gasket $\Sigma=\Sigma_{\infty}=\lim_{n\to\infty} \Sigma_n$ is a dense subset, which corresponds to the set of
``dyadic'' points for the interval or hierarchical figure-eight graph. 

Then we define~$\MG$ to be the set of complete metric spaces $(\mX,d)$ that 
contain~$V_\infty$ as a dense subset and for any $\lambda>0$ we have the map
\[\mathcal{R}_\lambda\,:\,\MG^3\times \R_+\to \MG\]  
which defines a new metric via the glueing.

Given a factor $\xi>0$ and the marginal vectors $X_1, X_2, X_3\in\R_+^3$, associated to the three 
metrics $d_1, d_2, d_3\in \MG$, we denote by $R_\lambda(X_1,X_2,X_3;\xi)$ the marginal vector 
of $\mathcal{R}_\lambda(d_1,d_2,d_3;\xi)$. It is not complicated to obtain an explicit definition 
for the map $R_\lambda$ (see Figure~\ref{fig:sierpinski-paths}). Writing 
$X_i=(d_i(\vs_1,\vs_2),d_i(\vs_2,\vs_3),d_i(\vs_3,\vs_1))=(x_i,y_i,z_i)$, we have
\begin{equation}\label{eq:R3}
R_\lambda(X_1,X_2,X_3;\xi)=\lambda\xi\cdot\begin{pmatrix}
\min(z_2+x_1+y_3, x_2+x_3) \\
\min(x_3+y_2+z_1, y_3+y_1) \\
\min(y_1+z_3+x_2, z_1+z_2)
\end{pmatrix}.
\end{equation}

\begin{figure}[ht]
\[
\includegraphics[width=.28\textwidth]{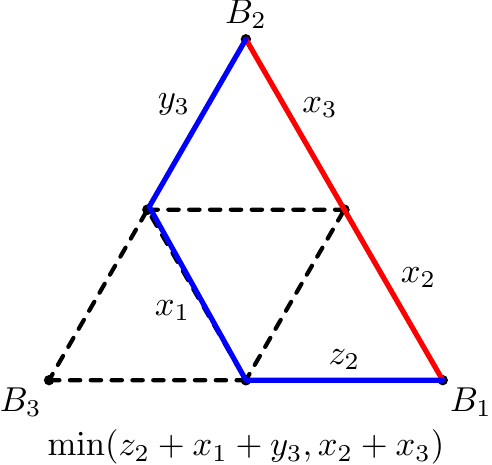}
\qquad
\includegraphics[width=.28\textwidth]{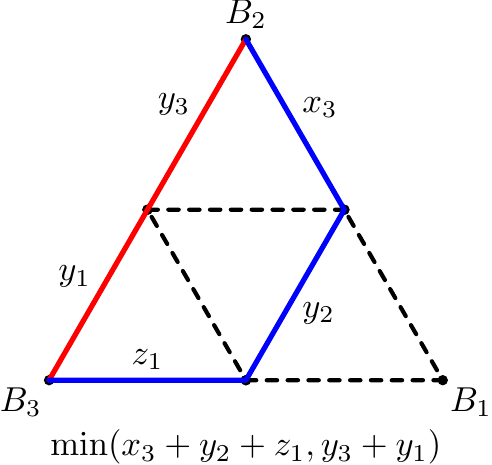}
\qquad
\includegraphics[width=.28\textwidth]{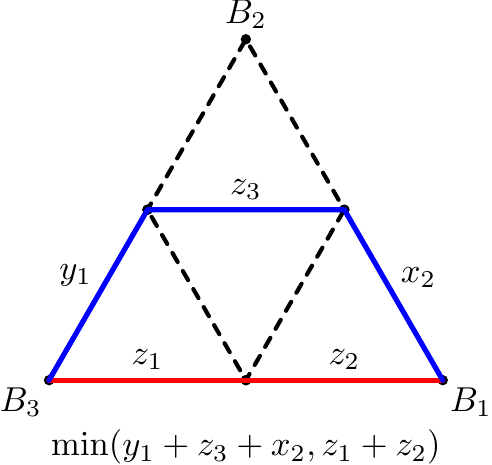}
\]
\caption{The combinatorics giving an expression for the function $R_\lambda$.}
\label{fig:sierpinski-paths}
\end{figure}

As for the hierarchical graphs case, it will be useful to consider distances that can be zero or infinite; the definition~\eqref{eq:R3} naturally extends to the map $R_{\lambda}$ from $([0,+\infty]^3)^3\times \R_+$ to~$[0,+\infty]^3$.

Actually, the three coordinates of the vectors $X_1$, $X_2$, $X_3$ and $R_\lambda(X_1,X_2,X_3;\xi)$
must verify the triangular inequality. Thus it is natural to introduce the subspace 
$\Delta\subset \R_+^3$ of triples of points $(x,y,z)$ such that
\[
x+y\ge z,\quad y+z\ge x,\quad z+x\ge y,
\]
as well as its closure $\bDe\subset [0,+\infty]^3$. For this reason we will sometimes 
consider $R_\lambda$ as a function from $\Delta^3\times \R_+$ to~$\Delta$ or from $\bDe^3\times \R_+$ to~$\bDe$.

Denoting by $\cP$ the space of Radon probability measures on $\bDe$, by $\cP_0$ its subset consisting of measures that do not charge $0$- or $\infty$-faces of $[0,+\infty]^3$. Then, the transformation
on the space of Radon probability measure on $\MG$ 
\[
\mathbf{\Phi}_{\lambda;\Sigma}\,:\,\mathbf{m}\in\cP(\MG)\mapsto 
(\mathcal{R}_\lambda)_*\left(\mathbf{m}^{\otimes 3}\otimes m\right)\in\cP(\MG)
\]
induces an operator $\Phi_{\lambda;\Sigma}$ of $\cP_0$ which is given by
\[
\Phi_{\lambda;\Sigma}[\mu]=\left (R_\lambda\right )_*\left (\mu^{\otimes 3}\otimes m\right ),
\]
and that extends naturally to an operator on~$\cP$.

More explicitly, $\Phi_{\lambda;\Sigma}[\mu]$ is the law of $R_{\lambda}(X_1,X_2,X_3;\xi)$, where the 
$X_i$'s and $\xi$ are independent, distributed with respect to $\mu$ and $m$ respectively.

Note that the glueing process for the Sierpi\'nski Gasket is analogous to the one for a hierarchical graph with 
pivotal bridge edge. Namely, for one of the vertices 
to be very far from both the other ones in the glued triangle, it suffices that it is far away from both 
in only \emph{one} small triangle (the one corresponding to this vertex, see Figure~\ref{f:pivotal}).

\begin{figure}[ht]
\[
\includegraphics[height=6cm]{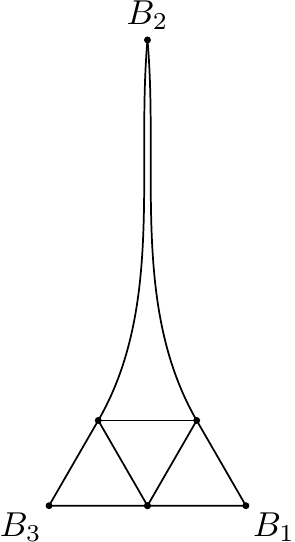} \qquad \includegraphics[height=6cm]{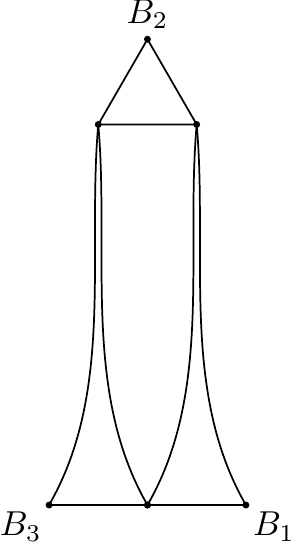}
\qquad \qquad
\includegraphics[height=6cm]{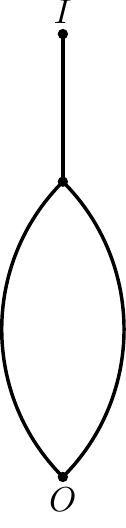}
\]
\caption{The pivotal behaviour for the Sierpi\'nski Gasket, compared to the racket graph}\label{f:pivotal}
\end{figure}

As in Theorem \ref{t:pivotal-bridge}, we claim the existence of a non-trivial self-similar random distance vector.
\begin{thm}\label{t:stat_Sierpinski}
For any non-atomic, fully supported probability distribution $m$ on $\R_+$ with 
finite first moment, there exists a normalizing constant $\lcr>0$ and a non-atomic probability 
distribution~$\bm$ on~$\Delta$ such that~$\bm$ is a fixed point for the operator~$\Phi_{\lcr;\Sigma}$.
\end{thm}

The remaining and concluding part of this work is devoted to the proof of this result.

\subsubsection{Coupling and stochastic order on~$\cP$} 

\begin{dfn}
We equip $\cP$ (and hence $\cP_0$) with the stochastic order $\lo$ 
induced by the coordinate-wise partial order on $\R_+^3$:
\[(x,y,z)\le (x',y',z')\quad\textrm{if }x\le x',\, y\le y',\, z\le z'.\]
That is, we write $\mu\lo \nu$ if there is a coupling $(X,Y)$ between them (namely, 
$\law(X)=\mu$, $\law(Y)=\nu$), such that almost surely $X\le Y$ (in particular they are almost surely comparable).
\end{dfn}

We already had the opportunity to mention (\S\ref{sc:morenotations}) that there is a functional
interpretation of the stochastic domination: the relation $\mu\lo\nu$ is equivalent to the 
Strassen's condition that for any increasing bounded real valued-function~$f$ on~$[0,+\infty]^3$,
\[
\int_{[0,+\infty]^3} f\,d\mu\le \int_{[0,+\infty]^3} f\,d\nu
\]
(a function $f$ is increasing if for any $(x,y,z)\le (x',y',z')$, then $f(x,y,z)\le f(x',y',z')$).
\begin{rem}
 Quite often the Strassen's condition is interpreted as the continuous version of
 Hall's matching lemma \cite{hall} (see the sketch of 
 the argument below, or for instance \cite{dudley,kamae}).
\end{rem}

In fact, in the same way as for the distributions on the real line, we can restrict ourselves and 
compare measures of sufficiently simple sets (or what is the same, their indicator functions).
For this we introduce the following class of sets:
\begin{dfn}
A set $M\subset [0,+\infty]^3$ is \emph{monotone (decreasing)} if for any $(x,y,z)\in M$, the 
subset $\{(x',y',z')\lo (x,y,z)\}$ is in~$M$. We denote by $\mathcal{M}$ the collection of all Borel 
monotone subsets of~$[0,+\infty]^3$.
\end{dfn}
It is evident that the condition $\mu(M)\ge\nu(M)$ for any $M\in\mcM$ is a necessary condition for~$\mu\lo\nu$. 
It turns out that this condition is also sufficient. Moreover assuming that the measures $\mu$, $\nu$ do 
not charge the $\infty$-faces (namely they are supported on $[0,+\infty)^3$), it suffices to consider 
only the sets that are open inside~$[0,+\infty)^3$. This is easy to deduce, for instance from the same 
Hall's lemma: considering the sets formed by a finite union of cubes of edge length~$\eps>0$, we 
obtain a coupling between ``$\eps$-discretizations'' of the initial measures; it suffices then to pass to 
the weak limit of these couplings.

It will be useful to work with the smaller class of open monotone sets:
\begin{dfn}
We introduce $\mcM_0$ to be the collection of monotone sets $M\in\mcM$ that are proper and relatively open subsets of $[0,+\infty)^3$.
In addition, for any $M\in\mcM_0$ we write $\partial_+ M:=\partial ([0,+\infty)^3\setminus M)$.
\end{dfn}

The ``boundary'' $\partial_+M$ uniquely defines $M\in\mcM_0$. Rotating the system of coordinates we can consider $\partial_+M$
as the graph of a continuous map $f_M$ from the plane $\pi=\{(x,y,z) \mid x+y+z=0\}\subset \R^3$ to its orthogonal complement, the 
diagonal line $D=\{(s,s,s)\}_{s\in\R}$. Vice versa, when such a function defines a monotone set $M\in\mcM_0$, we denote the corresponding set by~$M_f$.

We have the following immediate lemma, which will be used to adapt the compactness arguments that we used in the one dimensional setting:
\begin{lem}\label{l:compact}
For any $M\in\mcM_0$, the function $f_M$ is $3$-Lipschitz. Moreover if the functions $f_{M_j}$'s converge to $f_M$ uniformly on compact subsets of~$\pi$,
then one has 
\[
M\subset \liminf M_j \subset \limsup M_j \subset M \cup \partial_+M.
\]
\end{lem}
Denote by $\bF$ the space of functions $f:\pi\to D$ that correspond to at least one $M\in\mcM_0$, and let us equip this space
with the metric of the uniform convergence on the compact sets. Then $\bF$ is 
a complete metric space, and any subset $\bF_{A}=\{f\in\bF \mid f(0)\le A\}$ is compact for any $A\in\R_+$ (by the Ascoli--Arzel\`a theorem).

From Lemma \ref{l:compact} we have the following:
\begin{prop}
Take $\mu\in\cP_0$ and assume that for any $M\in\mcM_0$ we have $\mu(\partial_+M)=0$. Then the function $f\mapsto \mu(M_f)$ is continuous on $\bF$. 
\end{prop}

With the toolkit ready, let us go back to our problem. We have the hierarchical glueing transformation acting on~$\cP$: 
the image of a measure $\mu$ is the measure $\Phi_{\lambda;\Sigma}[\mu]$ which is the 
law of the three-dimensional random variable
\begin{equation}
\label{eq:recursionPhi}
Y=R_\lambda(X_1,X_2,X_3;\xi)=\lambda\xi\cdot\begin{pmatrix}
\min(z_2+x_1+y_3, x_2+x_3) \\
\min(x_3+y_2+z_1, y_3+y_1) \\
\min(y_1+z_3+x_2, z_1+z_2)
\end{pmatrix}
\end{equation}
where the variables $X_i=(x_i,y_i,z_i)$, $i=1$, $2$, $3$ are i.i.d.~with respect to $\mu$ and $\xi$ 
is independent of the previous variables and follows the law of $m$. In the same way as in
Lemma~\ref{lem:order}, the operator $\Phi_{\lambda;\Sigma}$ preserves the partial order~$\lo$: glueing 
the shortest (random) distances, one obtains the shortest distances.

\smallskip

For further use, we state the following easy fact:

\begin{lem}\label{l:continuous}
For any measure $\mu\in \cP_0$, any parameter $\lambda>0$ and any set $M\in\mcM_0$, the image~$\Phi_{\lambda;\Sigma}[\mu]$
does not charge~$\partial_+ M$.
\end{lem}

\subsection{Cut-off process}
Given $A\in \R_+$, we introduce the operator $\Phi_{A,\lambda;\Sigma}$ which shortcuts every 
distance with a path of length $A$: as $\Phi_{\lambda;\Sigma}$ was defined by 
\eqref{eq:recursionPhi}, we define the image $\Phi_{A,\lambda;\Sigma}[\mu]$ of a given $\mu\in \cP$,
to be the law of the random variable
\[
Y^{A}=R_{A,\lambda}(X_1,X_2,X_3;\xi)=\begin{pmatrix}
\min(\lambda\xi\cdot \min(z_2+x_1+y_3, x_2+x_3),A) \\
\min(\lambda\xi\cdot\min(x_3+y_2+z_1, y_3+y_1),A) \\
\min(\lambda\xi\cdot\min(y_1+z_3+x_2, z_1+z_2),A)
\end{pmatrix} 
\]
where the variables $X_i=(x_i,y_i,z_i)$, $i=1,2,3$ are i.i.d.~with respect to $\mu$ and $\xi$ is 
independent of the previous variables and follows the law~$m$.

It is geometrically evident that the operator $\Phi_{A,\lambda;\Sigma}$ is monotone for any $A$.
Hence, if we start from the deterministic degenerate metric $\ddelta_{(\infty,\infty,\infty)}$, 
the sequence $\mu_n^{A,\lambda}:=\Phi_{A,\lambda;\Sigma}^{n}[\ddelta_{(\infty,\infty,\infty)}]$ is $\preccurlyeq$-decreasing
and so converges to a probability measure $\nu_{A,\lambda}$.

\subsection{Existence of a stationary law}

Following the strategy for the proof of Theorem~\ref{t:pivotal-bridge}, we define the \emph{supercritical set} 
$\Lambda$ of factors $\lambda$ such that the limit measure $\nu_{A,\lambda}$ is non-trivial for any 
(equivalently, some) $A$:
\[\Lambda=\{\lambda>0 \mid \forall\,A\in\R_+\text{ one has }\nu_{A,\lambda}\ne \ddelta_{(0,0,0)}\}.\]
We claim that this set $\Lambda$ is a left-bounded half-line, whose left extremity is the factor $\lcr$ which is 
the candidate parameter for finding a non-trivial $\Phi_{\lcr;\Sigma}$-stationary probability measure.

\smallskip

Indeed, note first that $\Lambda$ is nonempty. This can be shown using the same arguments as 
in Lemma~\ref{l:Lambda}: to ensure that $\lambda\in\Lambda$, it suffices to find a compactly supported measure 
$\mu\neq\ddelta_{(0,0,0)}$ such that $\Phi_{\lambda,\Sigma}[\mu]\go \mu$ (as earlier, we can call such a measure \emph{$\lambda$-zooming out}).
The proof of Lemma~\ref{l:out} (claiming that the existence of a $\lambda$-zooming out measure implies $\lambda\in\Lambda$) can be followed nearly verbatim: 
for sufficiently large $A$ we have $\Phi_{A;\lambda,\Sigma}(\mu)\go \mu$ and hence by induction 
\[
\mu\lo \Phi_{A,\lambda;\Sigma}^{n}[\ddelta_{(\infty,\infty,\infty)}]
\]
for all $n$. The latter implies $\mu\lo \nu_{A,\lambda}$ and hence $\nu_{A,\lambda}\neq \ddelta_{(0,0,0)}$.

In order to find such a parameter $\lambda$, we consider $\mu=p\,\ddelta_{(0,0,0)}+ (1-p)\,\ddelta_{(1,1,1)}$, with $p>0$. 
The measure $\Phi_{1,\Sigma}[\mu]$ has an atom of weight $p^3$ at the origin and charges the $0$-faces (the positive 
quarters of the coordinate planes) of $\R_+^3$ with total weight~$p'=p^2(3-2p)$. In particular, for sufficiently small $p$ we 
have $p'<p$; we fix one such~$p$.

Any point in $\R_+^3$ not belonging to the coordinate planes can be rescaled so that each of its coordinates exceeds~$1$. This gives that for sufficiently large $\lambda$ we have 
\[
\Phi_{\lambda,\Sigma}[\mu]((1,+\infty)^3)>1-p,
\]
thus implying $\Phi_{\lambda,\Sigma}[\mu]\go \mu$.

\smallskip

Next, let us check that $\lcr>0$, that is, we want to prove that $(0,+\infty)\setminus\Lambda$ is nonempty. 
To do so, we compare any image $\Phi_{\lambda;\Sigma}[\mu]$ with the image of~$\mu$ under the operator 
$\Phi_{\lambda;\Sigma}^+$ associated to the map
\[
R^+_\lambda(X_1,X_2,X_3;\xi)=\lambda\xi\cdot\begin{pmatrix}
x_2+x_3 \\
y_3+y_1 \\
z_1+z_2
\end{pmatrix}.
\]
The natural coupling gives $\Phi_{\lambda;\Sigma}[\mu]\lo \Phi_{\lambda;\Sigma}^+[\mu]$, so any 
subcritical parameter for~$\Phi_{\lambda;\Sigma}^+$ is also subcritical for~$\Phi_{\lambda;\Sigma}$. 
Following the very same argument given in the proof of Theorem~\ref{t:pivotal-bridge}, any 
$\lambda<\frac{1}{2\cdot\E[m]}$ must be subcritical.

\medskip

Once again, the crucial step in the proof of Theorem \ref{t:stat_Sierpinski} is to prove that $\lcr$ 
does not belong to $\Lambda$. 
\begin{lem}[Key Lemma]\label{l:key}
The supercritical set $\Lambda$ is open: $\lcr\not\in\Lambda$.
\end{lem}

\begin{proof}[Sketch of the proof of Theorem~\ref{t:stat_Sierpinski}]
For every $\lambda>\lcr$ sufficiently close to $\lcr$, consider the value~$y(\lambda)$ such that
$\nu_{1,\lambda}\left([0,y(\lambda)]^3 \right)= 1/2$. Remark that $y(\lambda)>0$ and  
$y(\lambda)$ goes to $0$ as $\lambda\searrow\lcr$ (this is a consequence of the Key Lemma).
Defining $A(\lambda)=\frac{1}{y(\lambda)}$, we have
\begin{equation}\label{eq:norm-sierp}
\nu_{A(\lambda),\lambda}\left([0,1]^3\right)=\tfrac{1}{2}
\end{equation}
and $A(\lambda)$ goes to $\infty$ as $\lambda\searrow\lcr$. The limit $\bm$ of a convergent 
subsequence $\left\{\nu_{A(\lambda_j),\lambda_j}\right\}_{j\in \N}$ is then~$\Phi_{\lcr;\Sigma}$-stationary.

We remark that the last operation in the operator $\Phi_{\lcr;\Sigma}$ is the multiplicative convolution. Hence 
any measure that belongs to its range, in particular $\bm=\Phi_{\lcr;\Sigma}[\bm]$, does not charge 
$[0,1]^3 \setminus [0,1)^3$: within the rescalings of any $(x,y,z)$, there is at most one that belongs 
to this part of the boundary. Thus passing to the limit in \eqref{eq:norm-sierp} we get
\begin{equation}\label{eq:limit12}
\bm([0,1]^3)=\tfrac{1}{2},
\end{equation}
so that the measure $\bm$ is non-trivial.

\smallskip

Now, let us check that the measure $\bm$ does not charge $[0,+\infty]^3\setminus (0,+\infty)^3$ 
(that is, any of the three distances $d(B_i,B_j)$'s is almost surely positive and finite). Indeed, for any measure 
$\mu$ on $[0,+\infty]^3$ let $S_1(\mu)=\mu(\{\infty\}\times [0,+\infty] \times \{\infty\})$ be the probability 
that $B_1$ is at infinite distance from both~$B_2$ and~$B_3$. Then comparing the Sierpi\'nski Gasket 
with the racket graph $\Gamma$ (see Figure~\ref{f:pivotal}), it is easy to see that for any measure $\mu$ one has
\[
S_1(\Phi_{\lcr;\Sigma}[\mu])\ge 1- \theta_{\Gamma}(1-S_1(\mu)).
\]
As for any $p\in(0,1)$ one has $\theta_{\Gamma}(p)<p$, if for the stationary measure $\bm=\Phi_{\lcr;\Sigma}[\bm]$ we had $S_1(\bm)=p\in(0,1)$, this would imply 
\begin{equation}\label{eq:S1}
p=S_1(\Phi_{\lcr;\Sigma}[\bm])\ge 1-\theta_{\Gamma}(1-p)>p,
\end{equation}
leading to a contradiction. Hence the only possible values for $S_1(\bm)$ are~$0$ and~$1$. Due to the symmetry of $\bm$ 
under the permutation of the three points $B_i$'s, if $S_1(\bm)$ was equal to~$1$, we would have $\bm=\ddelta_{(\infty,\infty,\infty)}$, 
contradicting~\eqref{eq:limit12}. Hence all the three distances are finite almost surely and $\bm$ is concentrated on $[0,+\infty)^3$.

We repeat the same arguments with the function $S_1'(\mu)=\mu((0,+\infty]\times [0,+\infty] \times (0,+\infty])$, 
expressing the probability that $B_1$ is at positive distance from both $B_2$ and $B_3$. The 
analogue of~\eqref{eq:S1} still holds, and hence $S_1'(\bm)$ is also equal to $0$ or $1$. Now if we had $S_1'(\bm)=0$, 
the symmetry would imply that $\bm=\ddelta_{(0,0,0)}$, again contradicting~\eqref{eq:limit12}. Thus $S_1'(\bm)=1$, 
and $\bm((0,+\infty)^3)=1$.

\smallskip

Finally let us check that the measure $\bm$ is non-atomic. Indeed, due to the $\Phi_{\lcr;\Sigma}$-stationarity 
the only possible atoms of $\bm$ are points that are fixed by the rescaling. However for such points all the 
three distances should be zero or infinity, and we have already checked that all the three distances $d(B_i,B_j)$'s 
are almost surely positive and finite.
\end{proof}

\begin{rem}
In fact the argument used to ensure that distances $d(B_i,B_j)$'s are finite and positive naturally leads to 
the introduction of the percolation function $\theta_{\Sigma}$ associated to the Sierpi\'nski Gasket, see 
Definition~\ref{d:theta-sierp} in the last part of this section. 
\end{rem}

\begin{proof}[Proof of the Key Lemma]
Pick any $\lambda\in \Lambda$. Following the outline of the proof of Lemma~\ref{l:open},  
we are going to show that $\lambda-\eps\in\Lambda$ for sufficiently small~$\eps>0$. As in 
the case of the hierarchical graphs, in order to do so we are going to construct a measure $\mu$ 
and show that it is $(\lambda-\eps)$-zooming out (thus concluding by the analogue of Lemma~\ref{l:out}).

\smallskip

By definition $\lambda\in\Lambda$ means that there exists $A\in\R_+$ such that $\nu_{A,\lambda}\ne \ddelta_{(0,0,0)}$.
Copying the argument for hierarchical graphs, we will consider a larger cut-off: fixing $A'>A$, 
let us define the measure $\tmu:=\Phi_{A',\lambda;\Sigma}[\mu_{A,\lambda}]$. 

\smallskip

The next remark is that, for the measure $\tmu$, an analogue of the inequality~\eqref{eq:tnu-strict} holds, with measures of the 
monotone sets replacing the partition functions. Namely, we have the following:
\begin{lem}\label{l:positive}
For any nonempty monotone set $M\in\mcM_0$ one has the strict inequality
\[
\tmu(M)>\Phi_{\lambda;\Sigma}[\tmu](M).
\]
\end{lem}
\begin{proof}
Note that $\tmu=\Phi_{A',\lambda;\Sigma}[\mu_{A,\lambda}]\lo\Phi_{\lambda;\Sigma}[\nu_{A,\lambda}]$, hence for any $M\in\mcM_0$ we have
$\tmu(M)\ge \Phi_{\lambda;\Sigma}[\nu_{A,\lambda}](M)$, and it suffices to prove a (stronger) inequality: for any nonempty $M\in\mcM_0$ we want
\begin{equation}\label{eq:M-i}
\Phi_{\lambda;\Sigma}[\nu_{A,\lambda}](M)>\Phi_{\lambda;\Sigma}[\tmu](M).
\end{equation}
To do so, we will use the following argument. Assume that we have a coupling $(X,X')$ between two measures $\mu$ and $\mu'$ such that $X\le X'$ almost surely.
Take three independent copies $(X_i,X_i')$ of such a coupling, take any $A_1\le A_2$, and define
\[
Y=R_{A_1,\lambda}(X_1,X_2,X_3;\xi) \quad \text{and} \quad Y'=R_{A_2,\lambda}(X_1',X_2',X_3';\xi),
\]
where $\xi$ is distributed with respect to~$m$ and is independent of the~$(X_i,X_i')$'s.
Then $(Y, Y')$ is a coupling between $\Phi_{A_1,\lambda;\Sigma}[\mu]$ and $\Phi_{A_2,\lambda;\Sigma}[\mu']$, 
such that $Y\le Y'$ almost surely.
Such a construction can be done also if we allow to take $A_1$ and $A_2$ to be $+\infty$, 
corresponding to a simple~$R_{\lambda}$ and $\Phi_{\lambda;\Sigma}$.

\smallskip

Applying this for the diagonal coupling $(X,X)$ between $\nu_{A,\lambda}$ and itself, first with $(A_1,A_2)=(A,A')$ and then with $(+\infty,+\infty)$
to the resulting coupling $(Y,Y')$, we get the relaxed inequality in~\eqref{eq:M-i}:
\[
\Phi_{\lambda;\Sigma}[\nu_{A,\lambda}]\lo \Phi_{\lambda;\Sigma}[\tmu].
\] 
Though, this explicit construction of the coupling implies more.
After the first step, we get a coupling $(Y,Y')$ between $\nu_{A,\lambda}$ and $\tmu$ that has an atom at the point $((A,A,A),(A',A',A'))$.
Indeed, it suffices that the factor $\xi$ takes a value so large that all the three distances $d(B_i,B_j)$'s exceed~$A'$.  

Next consider the \emph{part} of the coupling $(Z,Z')$ between $\Phi_{\lambda;\Sigma}[\nu_{A,\lambda}]$ and $\Phi_{\lambda;\Sigma}[\tmu]$
 that we get with the second step, when we glue together three $(Y_i,Y_i')$'s that correspond to this atom at $((A,A,A),(A',A',A'))$.
 This gives a part that is supported on the couple formed by the equilateral triangle of side $2A\xi$ and its $\tfrac{A'}{A}$-rescaled image.
 
For any nonempty monotone set $M\in\mcM_0$, take the intersection point $(s_0,s_0,s_0)$ of the diagonal $\{(s,s,s)\}$ with
$\partial_+M$. In the part of the coupling described above, the factor $\xi$ belongs to the interval $\left(s_0/2A',s_0/2A\right)$
with positive probability (as the measure $m$ is fully supported on $\R_+$), and thus with positive probability we have $Z\in M$, $Z'\notin M$.
This gives the desired strict inequality~\eqref{eq:M-i}. 
\end{proof}
 
Recall that in the hierarchical graphs case we restricted ourselves to compare partition functions on the interval $[0,A']$ only. 
To do so in the present case, we introduce the following notation and then prove an easy fact:
\begin{dfn}
For any $s>0$ set $M_s:=\{(x,y,z) \mid \min(x,y,z)< s\}$. This is a monotone set in~$\mcM_0$.
\end{dfn}

\begin{lem}
If a measure $\mu$ is supported on $[0,A']^3$, then $\mu\lo\Phi_{\lambda;\Sigma}[\mu]$ if and only if 
\begin{equation}\label{eq:Phi-out}
\mu(M)\ge \Phi_{\lambda;\Sigma}[\mu](M)
\end{equation}
for any monotone set $M\subset M_{A'}$.
\end{lem}
\begin{proof}
A monotone set $M$ that is not contained in $M_{A'}$ necessarily contains the point $(A',A',A')$, and for 
such set~\eqref{eq:Phi-out} holds automatically, as the left hand side is equal to~$1$ in this case. 
\end{proof}

The next step is to obtain a strict inequality that cannot be destroyed by a small perturbation. To do so we set: 
\[
\mcM_{\delta,A'}:=\{M\in \mcM \mid (\delta,\delta,\delta)\in \overline{M},\, (A',A',A')\notin M\}=\{M_{f} \mid f(0,0,0)\in [\delta,A']\}.
\]
This replaces the segment $[\delta,A']$ for the case of hierarchical graphs. In the same way as before, 
we get the following lemma (that we will use as earlier, to perturb $\lambda$ to $\lambda-\eps$):
\begin{lem}\label{l:e1}
For any $\delta>0$ there exists $\eps_1>0$ such that for any monotone set $M\in \mcM_{\delta,A'}$
one has 
\[
\tmu(M)\ge \Phi_{\lambda;\Sigma}[\tmu](M) + \eps_1.
\]
\end{lem}
\begin{proof}
The set $\mcM_{\delta,A'}$ is compact (with respect to the topology introduced above). 
From Lemma~\ref{l:positive} and its proof we have for any $M\in\mcM_{\delta,A'}$ that
\[
\tmu(M)-\Phi_{\lambda;\Sigma}[\tmu](M) \ge \Phi_{\lambda;\Sigma}[\nu_{A,\lambda}](M)-\Phi_{\lambda;\Sigma}[\tmu](M)>0.
\]
The function that associates to the set $M\in\mcM_{\delta,A'}$ the difference 
$\Phi_{\lambda;\Sigma}[\nu_{A,\lambda}](M)-\Phi_{\lambda;\Sigma}[\tmu](M)$ is positive, and is also continuous 
due to Lemma~\ref{l:continuous}. Hence, it is bounded away from zero.
\end{proof}

In the same way as for the hierarchical graphs, we would like to handle the neighbourhood of $0$-faces:
when $M$ becomes thinner and thinner, both $\tmu(M)$ and $\Phi_{\lambda;\Sigma}[\tmu](M)$ tend to zero, 
as well as the difference between them. 
As in the proof of Lemma~\ref{l:open} (see \S{}\ref{ss:openness}), we will cope with this difficulty 
by mixing our initial measure $\wmu$ with the totally collapsed metric $\ddelta_{(0,0,0)}$,
and using the fact that the probability of small distances decreases superlinearly under the 
application of the glueing operator~$\Phi_{\lambda;\Sigma}$.
Namely, we consider
\[\wmu_p:=(1-p)\,\wmu + p\,\ddelta_{(0,0,0)}.\]

Take a coupling $(X,X')$ between $\wmu$ and $\wmu_p$, such that $\P(X\neq X')\le p$. For the 
associated coupling $(Y,Y')$ between $\Phi_{\lambda;\Sigma}[\wmu]$ and 
$\Phi_{\lambda;\Sigma}[\wmu_p]$, one thus has $\P(Y\neq Y')\le 3p$. Hence, for any 
$M\in\mcM_0$ we have 
\[
\wmu_p(M)-\Phi_{\lambda;\Sigma}[\wmu_p](M) \ge \left(\wmu(M)-\Phi_{\lambda;\Sigma}[\wmu](M)\right) -4p.
\]
Hence for any $\delta>0$, taking the corresponding $\eps_1$ from the conclusion of 
Lemma~\ref{l:e1} and defining $p_0:=\frac{\eps_1}{5}>0$, we have 
that for any $p<p_0$ and any monotone set $M\in\mcM_{\delta,A'}$ that 
\begin{equation}\label{eq:M-strict}
\wmu_p(M)- \Phi_{\lambda;\Sigma}[\wmu_p](M) \ge \eps_1 - 4p \ge \frac{1}{5}\eps_1.
 \end{equation}
 The delicate part is to handle the inequality for monotone sets which do not contain the point~$(\delta,\delta,\delta)$ 
 in their closure. Doing this will take us to the end of the proof of the Key Lemma.

\medskip

In the same way as in~\eqref{eq:diff}, the image $\Phi_{\lambda;\Sigma}[\wmu_p]$ decomposes as
\begin{equation}\label{eq:S3}
\Phi_{\lambda;\Sigma}[\wmu_p]=(1-3p)\,\Phi_{\lambda;\Sigma}[\wmu]+3p\, \Phi_{\lambda;\Sigma}'[\widetilde \mu]+O(p^2),
\end{equation}
with $O(p^2)$ uniform, and $\Phi_{\lambda;\Sigma}'[\mu]$ defined as the law for the result of glueing two 
independent \mbox{$\mu$-distributed triangles}, together with a third collapsed one (with all the three distances 
equal to~$0$), and where the three glued triangles have equal chances to be chosen as the one which is collapsed. 

As in Lemma~\ref{l:G-0-1}, the~$0$--$1$ law guarantees that the measure $\nu_{A,\lambda}$ 
does not charge any of the $0$-faces. Indeed, if it was concentrated on one of it, by symmetry
($\nu_{A,\lambda}$ is preserved by any permutation of the coordinates) it would be concentrated on 
their intersection, and hence we would have $\nu_{A,\lambda}=\ddelta_{(0,0,0)}$.

The same holds for its image $\wmu=\Phi_{\lambda;\Sigma}[\nu_{A,\lambda}]$ and hence for the image
$\Phi_{\lambda;\Sigma}'[\wmu]$ as well. Thus, there exists $\delta_0>0$ such that 
\[
\wmu ( M_{\delta_0})< \tfrac16, \quad \Phi_{\lambda;\Sigma}'[\wmu] ( M_{\delta_0}) < \tfrac16.
\]
Indeed, the sets $M_{\delta}$ converge to the union of the $0$-faces as $\delta\to 0$,
so the existence of such $\delta_0$ follows from the continuity of the measures~$\wmu$ 
and~$\Phi_{\lambda;\Sigma}'[\wmu]$.

\smallskip

Choose and fix such a $\delta_0$. For any nonempty monotone set $M$ that does not contain $(\delta_0,\delta_0,\delta_0)$, on the one hand
we have 
\begin{align*}
\wmu_p(M)&= p+(1-p) \wmu(M)\\
&=\wmu(M)+p\cdot (1-\wmu(M))\\
&\ge\wmu(M)+\tfrac{5}{6}p;
\end{align*}
on the other, using~\eqref{eq:S3} we get
\begin{align*}
\Phi_{\lambda;\Sigma}[\wmu_p](M)&
=(1-3p)\,\Phi_{\lambda;\Sigma}[\wmu](M)+3p\, \Phi_{\lambda;\Sigma}'[\widetilde \mu](M)+O(p^2)
\\
&=\Phi_{\lambda;\Sigma}[\wmu](M)+ 3p (\Phi_{\lambda;\Sigma}'[\wmu](M)-\Phi_{\lambda;\Sigma}[\wmu](M)) + O(p^2)
\\
&\le \wmu(M)+ 3\cdot \tfrac{1}{6}\cdot p + O(p^2).
\end{align*}
As the $O(p^2)$ here is uniform on the set $M$, there exists $p_1$ for which the right hand side does not 
exceed $\wmu(M)+\frac{4}{6}p$ for any $p\le p_1$ and any $M$ that does not contain $(\delta_0,\delta_0,\delta_0)$. 

\smallskip

The above arguments allow us to choose and fix $p_0$ for $\delta=\delta_0$ so that~\eqref{eq:M-strict} holds and we set $p:=\min(p_0,p_1)$. 

Then for any $M$ that does not contain $(\delta_0,\delta_0,\delta_0)$, one has
\[
\tmu_p(M)\ge \wmu(M)+\frac{5}{6}p= \left(\wmu(M)+\frac{4}{6}p\right)+ \frac{1}{6}p \ge \Phi_{\lambda;\Sigma}[\wmu_p](M)+ \frac{1}{6}p;
\]
for any $M\in \mcM_{\delta,A'}$ we have the inequality~\eqref{eq:M-strict}. Thus, we have strict inequality $\tmu_p(M)>\Phi_{\lambda;\Sigma}[\wmu_p](M)$ 
for all $M$ from the compact set~$\mcM_{0,A'}$. 

Again due to the continuity and compactness arguments, there exists $\eps>0$ such that 
$$
\tmu_p(M)>\Phi_{\lambda-\eps;\Sigma}[\wmu_p](M)
$$
for all monotone sets $M\in\mcM_{0,A'}$. Finally, for any monotone set $M$ that contains $(A',A',A')$ we have $\tmu_p(M)=1$, hence the desired inequality is satisfied automatically. Thus, $\tmu_p\lo \Phi_{\lambda-\eps;\Sigma}[\wmu_p]$, the measure $\tmu_p$ is $(\lambda-\eps)$-zooming out, and hence $\lambda-\eps$ is supercritical.
\end{proof}

\subsubsection{The function $\theta$}

We introduced the percolation function $\theta_\G$, associated to a hierarchical graph $\Gamma$ in order to 
understand the behaviour of the operator~$\Phi_{\lambda;\Gamma}$ at extremal values. 
It corresponds to the reduced problem, where the
distance takes only two values: ``zero'' and ``non-zero'', or, 
what is the same, ``finite'' and``infinite''. Such a function is defined on the interval $[0,1]$, its argument being the parameter 
of the Bernoulli percolation (that sets to ``zero'' the lengths of the randomly chosen edges).

In the case of the Sierpi\'nski Gasket, instead of two possible states,
we have to consider all the possible decompositions of the set of vertices into percolation clusters.
This motivates us to introduce the following:
\begin{dfn}\label{d:theta-sierp}
Let $\sigma$ be the simplex of probability measures on the set of decompositions of~$\{\vs_1,\vs_2,\vs_3\}$
into disjoint subsets, which we think as \emph{partitions into percolation clusters}. The map 
$\theta_{\Sigma}:\sigma\to\sigma$ is defined in the following way. Take three independent cluster partitions, chosen with respect to~$P$. 
Glue them together, and consider the cluster partition of the vertices of the resulting 
``large'' triangle (see Figure~\ref{f:clusters}). The distribution of the resulting partition is~$\theta_{\Sigma}(P)$.
\end{dfn}
\begin{figure}[ht]
\[
\includegraphics[width=.35\textwidth]{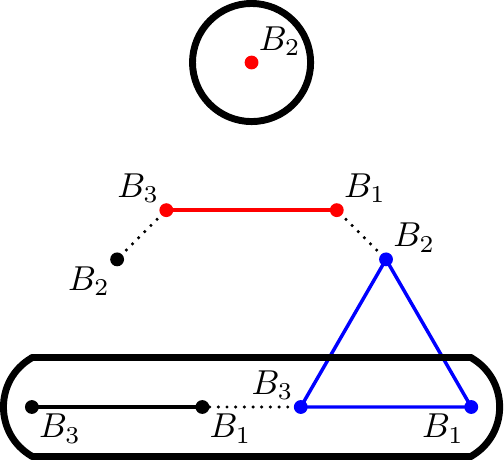}
\]
\caption{Glueing percolation clusters; clusters in the glued triangles are shown by edges, dotted lines indicate the identified points.}\label{f:clusters}
\end{figure}

As we have already discussed, the function $\theta_{\Sigma}$ behaves like the percolation 
function $\theta_\G$ associated to a pivotal graph with a bridge edge. In fact the arguments 
used in the proof of Theorem~\ref{t:stat_Sierpinski} produce also the following:
\begin{prop}\label{p:theta-sierp}
For any initial point $P\in\sigma$, the iterations $\theta_{\Sigma}^n(P)$ converge to one of the five
fixed points of $\theta_{\Sigma}$, that are all the Dirac measures concentrated on the five different
possible partitions of~$\{\vs_1,\vs_2,\vs_3\}$ (these are exactly the extremal points of the simplex $\sigma$).
\end{prop}
\begin{proof}
Let 
\[
Q_{1}(P)=P(\{(\{\vs_1,\vs_2\},\{\vs_3\}),(\{\vs_1,\vs_3\},\{\vs_2\}),(\{\vs_1,\vs_2,\vs_3\})\})
\]
be the probability that the point $\vs_1$ is connected to at least one of the two others. Then it 
is easy to see that
\[
Q_{1}(\theta_{\Sigma}(P))\le \theta_{\G}(Q_{1}(P)),
\]
where $\G$ is the ``racket''-graph (see Figure~\ref{f:pivotal}).
Indeed, this upper estimate comes from altering $P$ in such a way that $\vs_2$ and $\vs_3$ always
belong to the same cluster.

Hence, for any initial point $P\in\sigma$ either $Q_{1}(P)=1$ or $Q_1(\theta_{\Sigma}^n(P))$ goes to~$0$ 
as~$n$ goes to~$\infty$. Naturally the same applies to all other vertices.
\end{proof}

Notice that the map $\theta$ can be similarly defined for other self-similar length spaces. For the case of the Sierpi\'nski Gasket, its behaviour is pivotal-like, similar to the behaviour of $\theta_{\Gamma}$. However it is not clear what are the possible kinds of behaviour for $\theta_X$ when $X$ is a general self-similar length space. 
For instance, does an analogue of Moore-Shannon theorem hold? Is there some feature of hyperbolicity or expansivity with respect to some metric on the interior of~$\sigma$? Is it true that the number of fixed points in the interior of $\sigma$ does not exceed one? 

One can alter the Sierpi\'nski Gasket in such a way that the behaviour of the associated $\theta$ function is no longer pivotal. Indeed, we can modify the glueing procedure so that we glue six copies of the initial space: first, within both groups of three copies we glue them in the Sierpi\'nski-like, triangular way; then we identify the corresponding vertices of the two obtained spaces. It seems natural to expect the corresponding function $\theta$ to behave similarly as the function $\theta$ associated to the hierarchical diamond-graph, and in particular to have exactly one hyperbolic repelling fixed point inside the simplex~$\sigma$. All this motivates the following:

\begin{qn*}
Describe the behaviour of the function $\theta_X$ associated to a general self-similar length space~$X$.
\end{qn*}

\section{A very short summary: known and unknown results}

We conclude this paper by summarizing in Table~\ref{table} the known and conjectured properties of hierarchical spaces.

\begin{table}[h!]
\begin{center}
\begin{tabular}{|l|c|c|c|c|c|c|}
\hline 
& \includegraphics{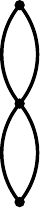} \quad \includegraphics{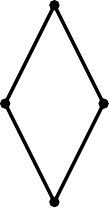}  & \includegraphics{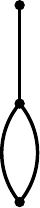} & \includegraphics{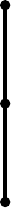} 
& \includegraphics{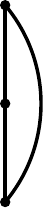} & \includegraphics{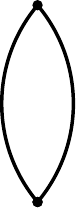} & \includegraphics{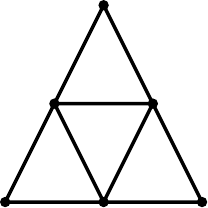}\\
& non-pivotal & bridge & interval & shortcut & BRW & Sierpi\'nski 
\\
\hline
Existence of & \checkmark & \checkmark& \checkmark & \checkmark& \checkmark& \checkmark\\
a stationary measure&Thms~\ref{p:metric}, \ref{t:stat} &Thm.~\ref{t:pivotal-bridge}&\cite{DL}&Thm.~\ref{t:pivotal-shortcut}&\cite{BRW}&Thm.~\ref{t:stat_Sierpinski}\\
\hline
Uniqueness of & \checkmark&conj. \checkmark&\checkmark& ? &\checkmark&conj.  \checkmark\\
a stationary measure &Thms.~\ref{t:uniqueness}, \ref{t:conv} &Rem.~\ref{r:conj-pivotal}&\cite{DL}&&\cite{BRW,aid}&Rem.~\ref{r:conj-pivotal}\\
\hline
Convergence to & \checkmark &conj. \checkmark&\checkmark& conj. \ding{56} & \ding{56} & conj. \checkmark \\
a stationary measure &Thm.~\ref{t:conv}  &Rem.~\ref{r:conj-pivotal}&\cite{DL}&Rem.~\ref{r:conj-pivotal}&\cite{BRW,aid}&Rem.~\ref{r:conj-pivotal}\\
\hline
Geometry of & \checkmark &conj. \checkmark&\checkmark&\ding{56}&\ding{56}& conj. \checkmark\\
the limit space &Thms.~\ref{t:metric}, \ref{t:geodesic}	&&\cite{KP,DL,BS}		&Rem. \ref{r:geometry-shortcut}	&Rem. \ref{r:geometry-shortcut}	&	\\
\hline
\end{tabular}
\end{center}
\caption{Hierarchical spaces and the properties of the associated process}\label{table}
\end{table}

\section*{Acknowledgements}

The results we presented here are at the same time a revision and a considerable extension of a 
previous version: the article has been seriously improved after the precious suggestions and the
excellent work of the two anonymous referees and that of Nicolas Curien and Fran\c cois B\'eguin 
(reviewers for the last named author's Ph.D.~thesis). In particular, we tried to answer to the 
most of a long list of interesting questions coming from one of the anonymous referees.

\smallskip

The authors would also like to express their gratitude to Dmitry Chelkak and Stanislav Smirnov
for having introduced us to the subject.
This work would not be the same without the many suggestions that Itai Benjamini gave us.
We thank Nicolas Curien for very fruitful remarks and for having enriched our knowledge on
quantum gravity, together with Jean-Fran\c{c}ois Le Gall and Bertrand Duplantier, during the
``2\`eme S\'eminaire Itzykson'' at the IH\'ES. 
This work has been carried on in various places and we thank for their hospitality the 
Chebyshev Laboratory in Saint Petersburg, the IRMAR in Rennes, the Poncelet Laboratory in 
Moscow, the IHP in Paris, the Todai University in Tokyo, the UMPA of the ENS-Lyon, the PUC 
in Rio de Janeiro, the summer school ``Contemporary Mathematics 2015'' in Dubna and the conference ``Global Dynamics Beyond
Uniform Hyperbolicity'' at Olmu\'e, Chile. 
We would also like to say \emph{merci} to the members of the UMPA at \'ENS-Lyon that have 
shown their interest since the early state of this work: Vincent Beffara, Christophe Garban, 
Gregory Miermont, Marielle Simon and in particular \'Etienne Ghys. We are also grateful to 
Thomas Duquesne for his interest in our work and the conversations that we had with him in Paris.

\smallskip

The first named author has been supported by the Chebyshev Laboratory (St.~Petersburg State University) 
under the RF Government grant 11.G34.31.0026 and by the JSC "Gazprom Neft".
The second named author has been supported by RFBR grant 13-01-00969-a, the project CSF CAPES and the 
R\'esau France-Br\'esil in Mathematics.
The third named author has been partially supported by the Grant-in-Aid for Scientific Research (S) 24224002, 
Japan Society for Promotion of Science, Japan and by the postdoctoral scholarship by CAPES, Brazil.

\end{document}